\documentclass{article}
\UseRawInputEncoding

\usepackage{epsfig}
\usepackage{amsmath,amsthm,amsfonts,amssymb,euscript}
\usepackage{graphicx}
\usepackage{pgf}
\usepackage{caption}
\usepackage{subcaption}
\usepackage{tikz,tkz-tab}
\usepackage{array,multirow,makecell}
\usepackage{tabularx}
\usepackage{hyperref}
\usepackage{makeidx}
\usepackage{geometry}
\usepackage{cancel}
\usepackage{xcolor}
\usepackage{enumitem} 
\usepackage{ulem}
\usepackage{authblk}
\def\XXint#1#2#3{{\setbox0=\hbox{$#1{#2#3}{\int}$}
    \vcenter{\hbox{$#2#3$}}\kern-.5\wd0}}

\def\CC{\color{blue}}
\def\BB{\color{black}}

\def\11{\mathbf{1}}
\begin{document}
\numberwithin{equation}{section}
\newtheorem{theoreme}{Theorem}[section]
\newtheorem{proposition}[theoreme]{Proposition}
\newtheorem{remarque}[theoreme]{Remark}
\newtheorem{lemme}[theoreme]{Lemma}
\newtheorem{corollaire}[theoreme]{Corollary}
\newtheorem{definition}[theoreme]{Definition}
\newtheorem{exemple}[theoreme]{Example}

\title{Homogenization of the Stokes system in a non-periodically perforated domain}
%\author[1]{X. Blanc}
\author[1]{S. Wolf}
\affil[1]{{\footnotesize Universit\'e de Paris, Sorbonne Paris-Cit\'e, Sorbonne Universit\'e, CNRS, Laboratoire Jacques-Louis Lions, F-75013 Paris}}

\maketitle
\begin{abstract}
In our recent work \cite{BW}, we have studied the homogenization of the Poisson equation in a class of non periodically perforated domains. In this paper, we examine the case of the Stokes system. We consider a porous medium in which the characteristic distance between two holes, denoted by~$\varepsilon$, is proportional to the characteristic size of the holes. It is well known (see \cite{allaire1989homogenization},\cite{sanchez1980fluid} and \cite{tartar127incompressible}) that, when the holes are periodically distributed in space, the velocity converges to a limit given by the Darcy's law when the size of the holes tends to zero. We generalize these results to the setting of \cite{BW}. The non-periodic domains are defined as a local perturbation of a periodic distribution of holes. We obtain classical results of the homogenization theory in perforated domains (existence of correctors and regularity estimates uniform in $\varepsilon$) and we prove $H^2-$convergence estimates for particular force fields. 
\end{abstract}
%%%%%%%%%%%%%%%%%%%%%%%%%%%%%%%%%%%%%%%%%%%%%%%%%%%%%%%%%%%%%%%%%%%%%%%%%%%%%%%

\tableofcontents

%%%%%%%%%%%%%%%%%

%%%%%%%%%%%%%%%%%

\section{Introduction}

In this paper, we study the three dimensional Stokes system in a perforated domain for an incompressible fluid with Dirichlet boundary conditions:
\begin{equation}
\begin{cases}
\begin{aligned}
-\Delta u_{\varepsilon} + \nabla p_{\varepsilon} & = f \quad \mathrm{in} \quad \Omega_{\varepsilon} \\
\mathrm{div} \ u_{\varepsilon} & = 0 \\
u_{\varepsilon} & = 0 \quad \mathrm{on} \quad \partial \Omega_{\varepsilon}.
\end{aligned}
\end{cases}
\label{eq:stokes}
\end{equation}
In Equation \eqref{eq:stokes}, $\Omega_{\varepsilon} \subset \mathbb{R}^3$ denotes the perforated domain, the vector valued function $f$ is the force field, the unknowns $u_{\varepsilon}$ and $p_{\varepsilon}$ refer respectively to the velocity and the pressure of the fluid. The distance between two neighbouring holes is denoted by $\varepsilon$. We assume that the charasteristic size of the holes is $\varepsilon$. Our purpose is to understand the limit of $(u_{\varepsilon},p_{\varepsilon})$ when $\varepsilon \rightarrow 0$. We construct classical objects of the homogenization theory such as correctors (Theorem~\ref{th:cor}) and we give new rates of convergence of $u_{\varepsilon}$ to its limit when $f$ is smooth, compactly supported and $\mathrm{div}(Af) = 0$ where $A$ is the so-called permeability tensor (see Theorem \ref{th:convergence}).

\medskip

To our knowledge, the first paper on the homogenization of the Stokes system in perforated domains is \cite{tartar127incompressible}. In this work, Equation \eqref{eq:stokes} is studied for a periodic distribution of perforations in the macroscopic domain $\Omega$ (that is, each cell of a periodic array of size $\varepsilon$ contains a perforation). It is in particular proved that $(u_{\varepsilon}/\varepsilon^2,p_{\varepsilon})$ converges in some sense to a couple $(u_0,p_0)$ given by the Darcy's law. This result can be guessed by performing a standard two scale expansion of $(u_{\varepsilon},p_{\varepsilon})$, see \cite{sanchez1980fluid}. Error estimates between $u_{\varepsilon}$ and its first order term in $\varepsilon$ are proved in \cite{maruvsic1997asymptotic,marusic1996error} for particular situations namely the two-dimensional case in \cite{marusic1996error} and the case of a periodic macroscopic domain in \cite{maruvsic1997asymptotic}. Sharp error estimates under general assumptions on $f$ have been obtained in \cite{shen}. The case of boundary layers in an infinite two-dimensional rectangular has been addressed in \cite{jager1994flow}. The results of \cite{tartar127incompressible} have been extended in \cite{allaire1989homogenization} to porous medium in which both solid and fluid parts are connected. The case of holes that scale differently as $\varepsilon$ is examined in \cite{allaire1991homogenization}. Recently, the homogenization of the Stokes system at higher order has been adressed in \cite{feppon2020high}.

\medskip

 In this paper, we adapt the results of \cite{tartar127incompressible} to the setting of \cite{BW}, that is to perforated domains that are defined as a local perturbation of the periodically perforated domain considered in \cite{tartar127incompressible}. This framework is inspired by the papers \cite{BLLMilan,BLLcpde,BLLfutur1}  (see \cite[Remark 1.5]{BW}). The purpose of these works is to study the homogenization of ellitptic PDEs with coefficients that are periodic and perturbed by a defect which belongs to $L^r$, $1 < r < +\infty$.

\medskip

%RAJOUTER ENONCE DES THEOREMES (numéro référencés dans l'article)

%Extensions: - dans le cas du non-périodique  ; - dans le cas périodique 

%The case of periodically distributed holes in space is understood. In particular, it is well known that the homogenized problem depends on the scale of the holes (see \cite{allaire1991homogenization}) and that many regimes can we distinguished. If the holes scale like $\varepsilon$, the limit equation is given by the Darcy's law. To our knowledge, this result was first proved by L. Tartar in \cite{tartar127incompressible} for periodic holes that are strictly included in each unit cell $Q_k := ]k - \frac{1}{2},k+\frac{1}{2} [^d$, $k \in \mathbb{Z}^d$. 

\medskip

The paper is organized as follows. We recall in subsection~\ref{section11} the main results of the homogenization of the Stokes system in the periodic case. We introduce in subsection~\ref{section12} the non-periodic setting. We state in Section~\ref{section2} the main results of this paper and we make some remarks. These results are proved in Section~\ref{sectproof}. Some technical Lemmas are given in Appendix~\ref{sect:appendix}. In Appendix~\ref{sec:geom}, we give more specific geometric assumptions on the non-periodic perforations that allow to obtain the results of Section~\ref{section2}.

\subsection{General notations} 

The canonical basis of $\mathbb{R}^3$ is denoted $e_1,e_2,e_3$. We denote the euclidian scalar product between two vectors $u$ and $v$ by $u \cdot v$. The euclidian distance to a subset $A \subset \mathbb{R}^3$ will be written $d(\cdot,A)$. The diameter of $A$ will be denoted by $\mathrm{diam}(A)$. If $A$ is a Lipschitz domain, we denote by $n$ the outward normal vector. $|\cdot|$ will be the Lebesgue measure on $\mathbb{R}^3$.

\medskip

 If $A,B$ are two real matrices, we write $A : B := \sum_{i,j=1}^3 A_{i,j}B_{i,j}$. If $X$ is a vector or a matrix, its transpose will be denoted by $X^T$. If $A \subset \mathbb{R}^3$, the complementary set of $A$ will be written $A^c$.
We define $Q := ]-\frac{1}{2},\frac{1}{2}[^3$ and, for $k \in \mathbb{Z}^3$, $Q_k := \prod_{j=1}^3\left]-\frac{1}{2} +k_j,\frac{1}{2}+k_j \right[^3 = Q + k$. If $x \in \mathbb{R}^3$ and $r > 0$, we denote by $B(x,r)$ the open ball centered in $x$ of radius $r$. 

\medskip

The gradient operator of a real or vector valued function will be denoted $\nabla \cdot$ and the second order derivative of a real or vector valued function will be written $D^2 \cdot$. The divergence operator will be denoted $\mathrm{div} \ \cdot$ and the scalar or vectorial  Laplacian $\Delta\cdot$.

\paragraph{Functional spaces.} If $\omega$ is an open subset of $\mathbb{R}^3$ and $1 \leq p \leq + \infty$, we denote by $L^p(\omega)$ the standard Lebesgue spaces and $H^s(\omega),W^{m,p}(\omega)$ the standard Sobolev spaces. For $s \in \mathbb{R}$ and $m\in\mathbb{N}^*$, we denote by $\left[L^p(\omega) \right]^3$, $\left[H^s(\omega) \right]^3$ and $\left[ W^{m,p}(\omega) \right]^3$ the spaces of vector valued functions whose components are respectively elements of $L^p(\omega)$, $H^s(\omega)$ and $W^{m,p}(\omega)$. The space $L^p(\omega)/\mathbb{R}$ corresponds to the equivalence classes for the relation $\sim$ defined by: for all $f,g \in L^p(\omega)$, $f \sim g$ if and only if $f - g$ is a.e constant in $\omega$.
$\mathcal{D}(\omega)$ will be the set of smooth and compactly supported functions in $\omega$. We denote by $\mathcal{C}^{\infty}(\omega)$ (resp. $\mathcal{C}^{\infty}(\overline{\omega})$) the set of smooth functions defined on $\omega$ (resp. $\overline{\omega}$).

\subsection{Review of the periodic case}
\label{section11}

In this subsection, we recall the results of the homogenization of the Stokes system in periodically perforated domains with large holes. For more details, see \cite{tartar127incompressible,sanchez1980fluid,allaire1989homogenization}.

%NOTATION : produit scalaire canonique, :, espaces fonctionnels H^1,per H^1,per^3, L^2/R ; distance euclidienne; notation : (IPP) ; strictly included , Lebesgue measure

\paragraph{Notations.}  We fix a locally Lipschitz bounded domain $\Omega \subset \mathbb{R}^3$ and a subset $\mathcal{O}_0^{\mathrm{per}}$ such that $\mathcal{O}_0^{\mathrm{per}} \subset \subset Q$, $\mathcal{O}_0^{\mathrm{per}}$ is of class $\mathcal{C}^{2,\alpha}$ and $Q \backslash \overline{\mathcal{O}_0^{\mathrm{per}}}$ is connected. We define for $k \in \mathbb{Z}^3$, $\mathcal{O}_k^{\mathrm{per}} := \mathcal{O}_0^{\mathrm{per}} + k$. $\mathcal{O}^{\mathrm{per}}$ will be the set of perforations, that is, $\mathcal{O}^{\mathrm{per}} := \bigcup_{k \in \mathbb{Z}^3} \mathcal{O}_k^{\mathrm{per}}$.

\medskip

We define some periodic functional spaces that will be used in the sequel. Using the notations of our problem, we set for $1 \leq p \leq + \infty$,
$$L^{p,\mathrm{per}}\left(Q \backslash \overline{\mathcal{O}_0^{\mathrm{per}}} \right) := \left\{ u \in L^p_{\mathrm{loc}}(\mathbb{R}^3 \backslash \overline{\mathcal{O}^{\mathrm{per}}}) \ \mathrm{s.t.} \ u \ \mathrm{is} \ Q-\mathrm{periodic} \right\}$$
and
$$H^{1,\mathrm{per}}\left( Q \backslash \overline{\mathcal{O}^{\mathrm{per}}_0} \right) := \left\{ u \in H^1_{\mathrm{loc}}(\mathbb{R}^3 \backslash \overline{\mathcal{O}^{\mathrm{per}}}) \ \mathrm{s.t.} \ u \ \mathrm{is} \ Q-\mathrm{periodic} \ \mathrm{and} \ \partial_i u \ \mathrm{are} \ Q-\mathrm{periodic}, i=1,2,3  \right\}.$$
The space of $H^1-$periodic vector valued functions will be $\left[ H^{1,\mathrm{per}}\left( Q \backslash \overline{\mathcal{O}^{\mathrm{per}}_0} \right) \right]^3$. The space of $H^1-$periodic functions that vanish on the perforations is
$$H^{1,\mathrm{per}}_0\left( Q \backslash \overline{\mathcal{O}^{\mathrm{per}}_0} \right) := \left\{ u \in H^{1,\mathrm{per}}\left( Q \backslash \overline{\mathcal{O}^{\mathrm{per}}_0} \right) \ \mathrm{s.t.} \ u = 0 \ \mathrm{on} \ \partial \mathcal{O}^{\mathrm{per}}_0 \right\}.$$
Similarly, we define $\left[ H^{1,\mathrm{per}}_0\left( Q \backslash \overline{\mathcal{O}^{\mathrm{per}}_0} \right) \right]^3$. In the sequel, we use the summation convention on repeated indices.

%It is understood that a function $u \in H^{1,\mathrm{per}}\left( Q \backslash \overline{\mathcal{O}^{\mathrm{per}}_0} \right)$ is extended by periodicity to $\mathbb{R}^3 \backslash \overline{\mathcal{O}^{\mathrm{per}}}$ and its extension is still denoted $u$.

\medskip

 For $\varepsilon > 0$, we denote $Y_{\varepsilon}^{\mathrm{per}} := \{ k \in \mathbb{Z}^3, \ \varepsilon Q_k \subset \Omega\}$. We define the periodically perforated domain~$\Omega_{\varepsilon}^{\mathrm{per}}$ by (see Figure~\ref{figper1/20})
 $$\Omega_{\varepsilon}^{\mathrm{per}} := \Omega\setminus \bigcup_{k \in Y_{\varepsilon}^{\mathrm{per}}} \varepsilon \overline{\mathcal{O}_k^{\mathrm{per}}}.$$ It is easily seen that $\Omega_{\varepsilon}^{\mathrm{per}}$ is open and connected. 
 
 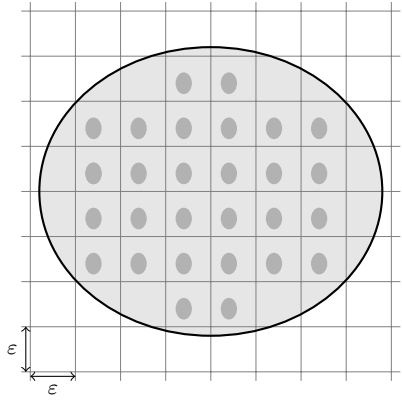
\begin{figure}[h!]
%periodic perforation: cell problem
  \centering
     \begin{tikzpicture}[scale=.12]\footnotesize
 \pgfmathsetmacro{\xone}{-21}
 \pgfmathsetmacro{\xtwo}{21}
 \pgfmathsetmacro{\yone}{-21}
 \pgfmathsetmacro{\ytwo}{21}
 
\fill[gray!20] (0,0) ellipse (19cm and 16cm);

\begin{scope}<+->;
  \draw[step=5cm,gray,very thin] (\xone,\yone) grid (\xtwo,\ytwo);
\end{scope}

\foreach \x in {-10,-5,0,5}
    \foreach \y in {-10,-5,0,5}
        \fill[gray!60] (\x+2,\y+2) ellipse (0.9cm and 1.2cm);

\draw[<->] (-20,-20.5) -- (-15,-20.5);
\draw (-17.5,-22) node[]{$\varepsilon$};
\draw[<->] (-20.5,-20) -- (-20.5,-15);
\draw (-22,-17.5) node[]{$\varepsilon$};
 \fill[gray!60] (-13,2) ellipse (0.9cm and 1.2cm);
  \fill[gray!60] (-13,-3) ellipse (0.9cm and 1.2cm);
   \fill[gray!60] (-13,-8) ellipse (0.9cm and 1.2cm);
      \fill[gray!60] (-13,7) ellipse (0.9cm and 1.2cm);
 \fill[gray!60] (12,2) ellipse (0.9cm and 1.2cm);
  \fill[gray!60] (12,-3) ellipse (0.9cm and 1.2cm);
   \fill[gray!60] (12,-8) ellipse (0.9cm and 1.2cm);
      \fill[gray!60] (12,7) ellipse (0.9cm and 1.2cm);
       \fill[gray!60] (-13,2) ellipse (0.9cm and 1.2cm);
  \fill[gray!60] (-3,-13) ellipse (0.9cm and 1.2cm);
   \fill[gray!60] (2,-13) ellipse (0.9cm and 1.2cm);
      \fill[gray!60] (-3,12) ellipse (0.9cm and 1.2cm);
  \fill[gray!60] (2,12) ellipse (0.9cm and 1.2cm);
\draw[black,thick] (0,0) ellipse (19cm and 16cm);
\end{tikzpicture}
\caption{Periodic domain $\Omega_{\varepsilon}^{\mathrm{per}}$}
\label{figper1/20}
\end{figure}

For $f \in \left[L^2(\Omega)\right]^3$, there exists a unique couple $(u_{\varepsilon},p_{\varepsilon}) \in \left[ H^1_0(\Omega_{\varepsilon}^{\mathrm{per}}) \right]^3 \times L^2(\Omega_{\varepsilon}^{\mathrm{per}})/\mathbb{R}$ solution of System \eqref{eq:stokes}. The Poincar\'e inequality in perforated domains (see e.g. \cite[Lemma 1]{tartar127incompressible}) and standard energy estimates yield the bound $$\| u_{\varepsilon} \|_{\left[L^2(\Omega_{\varepsilon}^{\mathrm{per}})\right]^3} \leq C \varepsilon^2$$ where $C$ is a constant independent of $\varepsilon$. Thus, after extraction of a subsequence, $u_{\varepsilon}/\varepsilon^2$ converges $L^2-$weakly to some limit velocity $u^*$. Besides, it can be proved (see \cite[Theorem 1]{tartar127incompressible}) that the pressure $p_{\varepsilon}$ converges $L^2(\Omega)/\mathbb{R}-$strongly to the macroscopic pressure $p_0$ which is defined up to the addition of a constant. The couple $(u^*,p_0)$ is determined by the Darcy's law which we recall here
 \begin{equation}
 \begin{cases}
 \begin{aligned}
 \mathrm{div}(u^*) & = 0 \quad \mathrm{in} \quad \Omega \\
 u^* & = A(f-\nabla p_0) \\
 u^* \cdot n & = 0 \quad \mathrm{on} \quad \partial \Omega.
 \end{aligned}
 \end{cases}
 \label{eq:darcy}
 \end{equation}
In \eqref{eq:darcy}, the symmetric and positive definite matrix $A$ is the so-called permeability tensor. Its coefficients are defined by
\begin{equation}
A_i^j = \int_{Q \setminus \overline{\mathcal{O}^{\mathrm{per}}_0}} w_j^{\mathrm{per}} \cdot e_i = \int_{Q \setminus \overline{\mathcal{O}^{\mathrm{per}}_0}} \nabla w_i^{\mathrm{per}} : \nabla w_j^{\mathrm{per}}, \ \ \ 1 \leq i,j \leq 3,
\label{eq:darcyA}
\end{equation} 
where the functions $w_j^{\mathrm{per}}, j = 1,2,3$ are the cell periodic first correctors and solve the following Stokes problems:
 \begin{equation}
 \begin{cases}
 \begin{aligned}
 - \Delta w_j^{\mathrm{per}} + \nabla p_j^{\mathrm{per}} & = e_j \quad \mathrm{in} \quad Q \setminus \overline{\mathcal{O}_0^{\mathrm{per}}} \\
 \mathrm{div} \ w_j^{\mathrm{per}} & = 0 \\
 w_{j}^{\mathrm{per}} & = 0 \quad \mathrm{on} \quad \partial \mathcal{O}^{\mathrm{per}}_0.
 \end{aligned}
 \end{cases}
 \label{eq:correcteur}
\end{equation} 
We note that for fixed $j \in \{1,2,3\}$, Problem \eqref{eq:correcteur} is well-posed in the space $\left[ H^1_0(Q \setminus \overline{\mathcal{O}_0^{\mathrm{per}}}) \right]^3 \times L^{2,\mathrm{per}}(Q \setminus \overline{\mathcal{O}_0^{\mathrm{per}}})/\mathbb{R}$ (see \cite{sanchez1980fluid}). A central point in the proof of the convergence of $p_{\varepsilon}$ to $p_0$ is the construction of an extension of the pressure $p_{\varepsilon}$ in the periodic holes. This extension is constructed in \cite{tartar127incompressible} by a duality argument. 
%It is then noticed in \cite{allaire1989homogenization} that $p_j^{\mathrm{per}}$ may be extended by its mean value on the periodic cell $Q \setminus \overline{\mathcal{O}_0^{\mathrm{per}}}$. 
\medskip

% use this fact and set
%\begin{equation}
%p_j^{\mathrm{per}} = \frac{1}{|Q \setminus \overline{\mathcal{O}%^{\mathrm{per}}_0}|}\int_{Q \setminus \overline{\mathcal{O}%^{\mathrm{per}}_0}} p_j^{\mathrm{per}} \ \mathrm{in} \ \mathcal{O}%^{\mathrm{per}}_0, \ \mathrm{for} \ j=1,2,3.
%\label{eq:extensionpressionperiodique}
%\end{equation}

The corrector equations \eqref{eq:correcteur} can be guessed by a standard two-scale expansion of $u_{\varepsilon}$ and $p_{\varepsilon}$ of the form
 $$u_{\varepsilon} = u_0 \left( x,\frac{x}{\varepsilon} \right) + \varepsilon u_1 \left(x,\frac{x}{\varepsilon} \right) + \varepsilon^2 u_2 \left(x,\frac{x}{\varepsilon}\right) + \varepsilon^3 u_3 \left(x,\frac{x}{\varepsilon} \right)+ \cdots,$$
  $$p_{\varepsilon} = p_0 \left( x,\frac{x}{\varepsilon} \right) + \varepsilon p_1 \left(x,\frac{x}{\varepsilon} \right) + \varepsilon^2 p_2 \left(x,\frac{x}{\varepsilon}\right) + \varepsilon^3 p_3 \left(x,\frac{x}{\varepsilon} \right)+ \cdots$$
  where the functions $u_i(x,\cdot)$ and $p_i(x,\cdot)$ are $Q-$periodic for fixed $x \in \Omega$ (see \cite[Section 7.2]{sanchez1980fluid}). It can be proved that the function $p_0$ is independent of the microscopic variable, that is $p_0(x,\frac{x}{\varepsilon}) = p_0(x)$ for all $x \in \Omega$ (which is coherent with \eqref{eq:darcy}). Besides, the functions $u_0$ and $u_1$ vanish and (we use, as indicated above, the summation convention over repeated indices)
  $$u_2 \left(x,\frac{x}{\varepsilon}\right) = w_j\left(\frac{x}{\varepsilon} \right)(f_j - \partial_j p_0)(x) \quad \mathrm{and} \quad  p_1\left(x,\frac{x}{\varepsilon}\right) = p_j\left(\frac{x}{\varepsilon}\right)(f_j - \partial_j p_0)(x).$$
   We define the remainders
  $$R_{\varepsilon} := u_{\varepsilon} - \varepsilon^2 w_j\left(\frac{\cdot}{\varepsilon} \right)(f_j - \partial_j p_0) \ \ \ \mathrm{and} \ \ \ \pi_{\varepsilon} := p_{\varepsilon} - p_0 - \varepsilon p_j \left(\frac{\cdot}{\varepsilon} \right)(f_j - \partial_j p_0).$$
  The strong convergence $R_{\varepsilon}/\varepsilon^2 \rightarrow 0$ in $L^2(\Omega_{\varepsilon}^{\mathrm{per}})-$norm is proved in \cite[Theorem 1.3]{allaire1991continuity}. An $H^1-$quantitative estimate of this convergence is given in \cite{shen}, provided that $\Omega$ is of class $\mathcal{C}^{2,\alpha}$. We will provide a new $H^2-$convergence estimate when $\mathrm{div}(Af) = 0$ and $f$ is compactly supported in $\Omega$ (see Theorem~\ref{th:convergence} and Remark \ref{re:bord} below).

%\begin{remarque}
\label{re:per}
%When $\Omega = Q$ and $\varepsilon = 1/m, m \in \mathbb{N}^*$, we may consider the problem
%\begin{equation}
%\begin{cases}
%\begin{aligned}
%- \Delta u_{\varepsilon} + \nabla p_{\varepsilon} & = f \ \mathrm{in} \ %\Omega_{\varepsilon}^{\mathrm{per}} \\
%\mathrm{div} (u_{\varepsilon}) & = 0 \\
%u_{\varepsilon} & = 0 \ \mathrm{on} \ \varepsilon\partial \mathcal{O}%^{\mathrm{per}} \cap \Omega \\
%u_{\varepsilon} & \ \mathrm{is} \ Q-\mathrm{periodic}
%\end{aligned}
%\end{cases}
%\label{eq:periodic}
%\end{equation}
%for $f \in L^{2,\mathrm{per}}(Q)$. It can easily be proved thanks to Theorem \ref{th:masmoudi} below that
%$$\left\| u_{\varepsilon} - \varepsilon^2 w_j (\cdot/\varepsilon)(f_j - \partial_j p_0) \right\|_{H^2(\Omega_{\varepsilon}^{\mathrm{per}})} \leq C \varepsilon.$$
%However, we restrict ourselves to equation \eqref{eq:stokes} in order to stick to the framework of \cite{BW}. We want here to point out that the assumptions on $f$ in Theorem \ref{th:convergence} or the macroscopic periodicity in equation \eqref{eq:periodic} make us avoid boundary effects. These boundary effects are much more difficult to treat for Stokes system as for Poisson equation (see \cite[Remark 1.4]{BW} for Poisson equation). 
%\end{remarque} 

\medskip

In what follows, we extend $w_j^{\mathrm{per}}$ by zero in the periodic perforations. The pressure $p_j^{\mathrm{per}}$ is extended by a constant $\lambda_j$ (for example zero) in the perforations.

\subsection{The non-periodic setting}
\label{section12}

We fix a periodic set of perforations as described in the previous subsection. We describe the non-periodic setting (see \cite{BW} for more details). For $k \in \mathbb{Z}^3$ and $\alpha > 0$, we define (see figure \ref{fig:a3})

$$\mathcal{O}_k^{\mathrm{per},+}(\alpha) := \{ x \in Q_k, \ \ d(x,\mathcal{O}_k^{\mathrm{per}}) < \alpha \},$$
and 
$$\mathcal{O}_k^{\mathrm{per},-}(\alpha) := \{ x \in \mathcal{O}_k^{\mathrm{per}}, \ \ d(x,\partial \mathcal{O}_k^{\mathrm{per}}) > \alpha \}.$$
For all $k \in \mathbb{Z}^3$, we fix an open subset $\mathcal{O}_k$ of $Q_k$. We suppose that
the sequence $(\mathcal{O}_k)_{k \in \mathbb{Z}^3}$ satisfies Assumptions \textbf{(A1)-(A5)} below. We define the non periodic set of perforations by $$\mathcal{O} := \bigcup_{k \in \mathbb{Z}^3} \mathcal{O}_k.$$

\vspace{0.2cm}

\noindent\textbf{(A1)} For all $k \in \mathbb{Z}^3$, we have $\mathcal{O}_k \subset \subset Q_k$ and $Q_k \setminus \overline{\mathcal{O}_k}$ is connected.

\vspace{0.2cm} 

\noindent\textbf{(A2)} For all $k \in \mathbb{Z}^3$, the perforation $\mathcal{O}_k$ is Lipschitz continuous.

\vspace{0.2cm}

\noindent\textbf{(A3)} There exists a sequence $(\alpha_k)_{k \in \mathbb{Z}^3} \in \ell^1(\mathbb{Z}^3)$ such that for all $k \in \mathbb{Z}^3$, $\alpha_k > 0$ and we have the following chain inclusion:
$$\mathcal{O}_k^{\mathrm{per},-}(\alpha_k) \subset \mathcal{O}_k \subset \mathcal{O}_k^{\mathrm{per},+}(\alpha_k).$$
We refer to figure \ref{fig:a3} for an illustration of \textbf{(A3)}.

\vspace{0.2cm}

%CONSTANTES DE REGULARITE / PROBLEME EN DIVERGENCE

The assumptions \textbf{(A1)-(A2)} are analogous to the one made on $\mathcal{O}^{\mathrm{per}}$ and guarantee connectedness and some regularity on the perforated domain. Assumption \textbf{(A3)} is the geometric assumption that makes precise that $(\mathcal{O}_k )_{k \in \mathbb{Z}^3}$ is a perturbation of $(\mathcal{O}_k^{\mathrm{per}})_{k \in \mathbb{Z}^3}$.
We recall (see \cite[Lemma A.1 and Lemma A.3]{BW}) that Assumptions \textbf{(A1)-(A3)} imply the following facts:
\begin{itemize}
\item There exists $\delta > 0$ such that for all $k \in \mathbb{Z}^3$, $d(\mathcal{O}_k,\partial Q_k) \geq \delta$. In other words, $\mathcal{O}_k$ is strictly included in $Q_k$, uniformly with respect to $k$. 
\item We have
\begin{equation}
\label{eq:diffsym}
\sum_{k \in \mathbb{Z}^3} |\mathcal{O}_k \Delta \mathcal{O}_k^{\mathrm{per}}| < + \infty
\end{equation} where $\Delta$ stands for the sets symmetric difference operator. 
\end{itemize}
Using the first point above, we can introduce two smooth open sets $Q'$ and $Q''$ such that (see Figure~\ref{fig5})
$Q' \subset \subset Q \subset \subset Q''$ and for all $k \in \mathbb{Z}^3$, 
$
(Q' + k) \cap \mathcal{O} = (Q'' + k) \cap \mathcal{O} = \mathcal{O}_k.
$
 We define, for $k \in \mathbb{Z}^3$, 
 \begin{equation} 
 Q'_k := Q' + k \quad \mathrm{and} \quad 
Q''_k := Q'' + k.
\label{eq:Q''}
\end{equation}
 The sets $Q'_k$ and $Q''_k$, $k \in \mathbb{Z}^3$ will be used several times in the sequel.

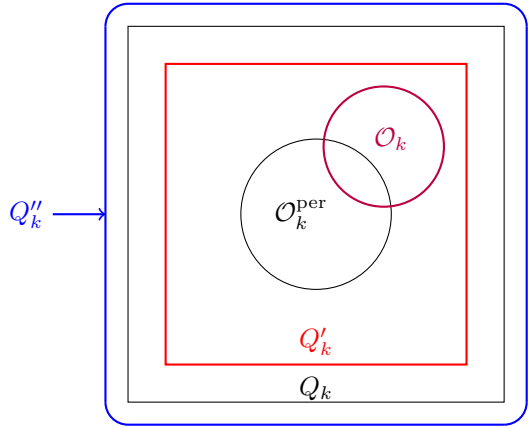
\begin{figure}[h!]
\centering
\begin{tikzpicture}
\draw (0,0)-- (5,0) -- (5,5) -- (0,5) -- cycle;
\draw (2.5,-0.1) node[above]{$Q_k$};
\draw [red,thick] (0.5,0.5)--(4.5,0.5)--(4.5,4.5)--(0.5,4.5) --cycle;
\draw[red,thick] (2.5,0.5) node[above]{$Q'_k$};
\draw (2.5,2.5) circle(1);
\draw[thick,purple] (3.4,3.4) circle(0.8);
\draw (2.3,2.5) node[]{$\mathcal{O}_k^{\mathrm{per}}$};
\draw[purple,thick] (3.5,3.5) node[]{$\mathcal{O}_k$};
\draw[blue,thick] (0,-0.3)--(5,-0.3);
\draw[blue,thick] (-0.3,0)--(-0.3,5);
\draw[blue,thick] (0,5.3)--(5,5.3);
\draw[blue,thick] (5.3,5)--(5.3,0);
\draw[blue,thick] (-0.3,0) arc(180:270:0.3);
\draw[blue,thick] (5,-0.3) arc(270:360:0.3);
\draw[blue,thick] (5.3,5) arc(0:90:0.3);
\draw[blue,thick] (0,5.3) arc(90:180:0.3);
\draw[blue,thick,->] (-1,2.5)--(-0.3,2.5);
\draw[blue] (-1,2.5) node[left]{$Q''_k$};
\end{tikzpicture}
\caption{A cell $Q_k$, $k \in \mathbb{Z}^3$}
\label{fig5}
\end{figure}

\vspace{0.2cm}

\noindent \textbf{(A4)} This assumption is divided into two sub-assumptions \textbf{(A4)$_0$} and \textbf{(A4)$_1$}.

\medskip

\noindent \textbf{(A4)$_0$} For all $1 < q < +\infty$, there exists a constant $C_q^0 > 0$ such that for all $k \in \mathbb{Z}^3$, the problem 
\begin{equation}
\begin{cases}
\begin{aligned}
\mathrm{div} \ v & = f \quad \mathrm{in} \quad Q_k \setminus \overline{\mathcal{O}_k} \\
v & = 0 \quad  \mathrm{on} \quad \partial \left[ Q_k \setminus \overline{\mathcal{O}_k} \right]
\end{aligned}
\end{cases}
\label{A4}
\end{equation}
with $f \in L^q(Q_k \setminus \overline{\mathcal{O}_k})$
completed with the compatibility condition
\begin{equation}
\int_{Q_k \setminus \overline{\mathcal{O}}_k} f = 0
\label{eq:comp}
\end{equation}
 admits a solution $v$ such that $v \in \left[ W^{1,q}(Q_k \backslash \overline{\mathcal{O}_k}) \right]^3$ and
\begin{equation}
\| v \|_{\left[ W^{1,q}(Q_k \backslash \overline{\mathcal{O}_k}) \right]^3} \leq C_q^0 \| f \|_{L^q(Q_k \backslash \overline{\mathcal{O}_k})}.
\label{eq:introdiv}
\end{equation}

\medskip

\noindent \textbf{(A4)$_1$} For all $1 < q < +\infty$, there exists a constant $C_q^1 > 0$ such that for all $k \in \mathbb{Z}^3$, Problem \eqref{A4} with $f \in W^{1,q}_0(Q_k \setminus \overline{\mathcal{O}_k})$ completed with the compatibility condition \eqref{eq:comp} admits a solution $v$ such that $v \in \left[ W^{2,q}_0(Q_k \setminus \overline{\mathcal{O}_k}) \right]^3$ and 
\begin{equation}
\| v \|_{\left[ W^{2,q}(Q_k \backslash \overline{\mathcal{O}_k}) \right]^3} \leq C_q^1 \| f \|_{W^{1,q}(Q_k \backslash \overline{\mathcal{O}_k})}.
\label{eq:A41}
\end{equation}

\vspace{0.2cm}

\noindent \textbf{(A5)} For all $1 < q < +\infty$, there exists a constant $C_q > 0$ such that for all $k \in \mathbb{Z}^3$, if $(v,p) \in \left[ W^{1,q}(Q_k \backslash \overline{\mathcal{O}_k}) \right]^3 \times L^{q}(Q_k \backslash \overline{\mathcal{O}_k})$ is solution to the Stokes problem 
\begin{equation}
\begin{cases}
\begin{aligned}
- \Delta v + \nabla p & = f \quad \mathrm{in} \quad Q_k'' \setminus \overline{\mathcal{O}_k} \\
\mathrm{div} \ v & = 0 \\
v & = 0 \quad \mathrm{on} \quad \partial \mathcal{O}_k
\end{aligned}
\end{cases}
\label{A5}
\end{equation}
with $f \in L^q(Q_k'' \backslash \overline{\mathcal{O}_k})$, then  $(v,p) \in \left[ W^{2,q}(Q_k \backslash \overline{\mathcal{O}_k}) \right]^3 \times W^{1,q}(Q_k \backslash \overline{\mathcal{O}_k})$ and
\begin{equation}
\| v \|_{\left[ W^{2,q}(Q_k \backslash \overline{\mathcal{O}_k})\right]^3} + \| p \|_{W^{1,q}(Q_k \backslash \overline{\mathcal{O}_k})} \leq C_q \big[ \| f \|_{L^q(Q''_k \backslash \overline{\mathcal{O}_k})^3} + \| v \|_{\left[ W^{1,q}(Q''_k \backslash \overline{\mathcal{O}_k})\right]^3} + \| p \|_{L^q(Q''_k \backslash \overline{\mathcal{O}_k})} \big].
\label{eq:introreg}
\end{equation}

\begin{remarque}
\label{re:galdi} 
For each fixed $k \in \mathbb{Z}^3$, the estimates \eqref{eq:introdiv} and \eqref{eq:A41} are satisfied with constants~$C_{q,k}^i$, $i=0,1$, depending on $k$, see \cite[Theorem III.3.3]{galdi2011introduction}. Similarly, as long as $\mathcal{O}_k$ is of class $\mathcal{C}^2$, \eqref{eq:introreg} is satisfied when $k$ is fixed (see~\cite[Theorem IV.5.1]{galdi2011introduction}). Assumptions \textbf{(A4)-(A5)} require that the constants appearing in \eqref{eq:introdiv}, \eqref{eq:A41} and \eqref{eq:introreg} are uniform with respect to $k \in \mathbb{Z}^3$. 

%We also point out that Assumption \textbf{(A4)} implies the following fact: if $f \in H^m_0(Q_k \setminus \overline{\mathcal{O}_k})$, $m \geq 1$, and $a = 0$, then Problem \eqref{A4} admits a solution $v \in \left[ H^{m+1}_0(Q_k \backslash \overline{\mathcal{O}_k}) \right]^3$ (see \cite[Theorem III.3.3]{galdi2011introduction}) such that
%\begin{equation}
%\label{eq:regaldi}
%\|v \|_{\left[ H^{m+1}(Q_k \setminus \overline{\mathcal{O}}_k) \right]^3} \leq C \| f \|_{H^{m}(Q_k \setminus \overline{\mathcal{O}_k)} },
%\end{equation}
%where $C$ is independent of $f$ and $k$.
\end{remarque}

\begin{figure}
\centering
\begin{subfigure}{.5\textwidth}
\centering
\begin{tikzpicture}[scale=1.2]\footnotesize
 \pgfmathsetmacro{\xone}{-0.2}
 \pgfmathsetmacro{\xtwo}{5.2}
 \pgfmathsetmacro{\yone}{-0.2}
 \pgfmathsetmacro{\ytwo}{5.2}
\begin{scope}<+->;
  \draw[step=5cm,gray,very thin] (\xone,\yone) grid (\xtwo,\ytwo);
\end{scope}

\fill[gray!20] (2.5,2.5) ellipse (1.9cm and 1.9cm);
\fill[gray!40] (2.5,2.5) ellipse (1.1cm and 1.1cm);
\draw[red,thick] (2.5,2.5) ellipse (1.5cm and 1.5cm);
\draw[red,dashed] (2.5,2.5) ellipse (1.9cm and 1.9cm);
\draw[red,dashed] (2.5,2.5) ellipse (1.1cm and 1.1cm);
\draw[<->, thick] (4.4,2.5) -- (4,2.5);
\draw[<->, thick] (4,2.5) -- (3.6,2.5);

\draw (2.5,2.5) node[]{$\mathcal{O}_k^{\mathrm{per},-}(\alpha)$};
\draw (3.8,2.7) node[]{$\alpha$};
\draw (4.2,2.7) node[]{$\alpha$};

\draw[<-] (1.1,2) -- (-0.5,2);
\draw (-1.1,2) node[]{$\mathcal{O}_k^{\mathrm{per}}$};
\draw[<-] (0.9,1.5) -- (-0.5,1);
\draw (-1.1,1) node[]{$\mathcal{O}_k^{\mathrm{per},+}(\alpha)$};
\end{tikzpicture}
\caption{Illustration of \textbf{(A3)}}
\label{fig:a3}
\end{subfigure}%
\begin{subfigure}{.5\textwidth}
\centering
\begin{tikzpicture}[scale=.15]\footnotesize
 \pgfmathsetmacro{\xone}{-21}
 \pgfmathsetmacro{\xtwo}{21}
 \pgfmathsetmacro{\yone}{-21}
 \pgfmathsetmacro{\ytwo}{21}
 
\begin{scope}<+->;
  \draw[step=2cm,gray,very thin] (\xone,\yone) grid (\xtwo,\ytwo);
\end{scope}

\foreach \x in {-20,-18,-16,-14,-12,-10,-8,-6,-4,-2,0,2,4,6,8,10,12,14,16,18}
    \foreach \y in {-20,-18,-16,-14,-12,-10,-8,-6,-4,-2,0,2,4,6,8,10,12,14,16,18}
        \fill[gray!60] (\x+0.8,\y+0.8) ellipse (0.4cm and 0.5cm);
        
\foreach \x in {-20,-18,-16.1,-14.2,-11.6,-9.4,-7.6,-6.2,-4.1,-2.2,0.5,2.4,4.3,6.7,8.6,10.1,12.2,14.1,16,18}
    \foreach \y in {-20,-18,-16,-13.7,-11.5,-10.2,-7.5,-6.2,-4.2,-2.1,0.1,2,3.8,6.5,8,10.2,12.2,14.1,15.9,18}
        \fill[gray!100] (\x+0.8,\y+0.8) ellipse (0.4cm and 0.5cm);

\draw[<->] (-20,-20.5) -- (-18,-20.5);
\draw (-19,-22) node[]{$\varepsilon$};
\draw[<->] (-20.5,-20) -- (-20.5,-18);
\draw (-22,-19) node[]{$\varepsilon$};
\end{tikzpicture}
\caption{A non periodically perforated grid}
\label{fig:nonper}
\end{subfigure}
\caption{The non-periodic setting}
\end{figure}
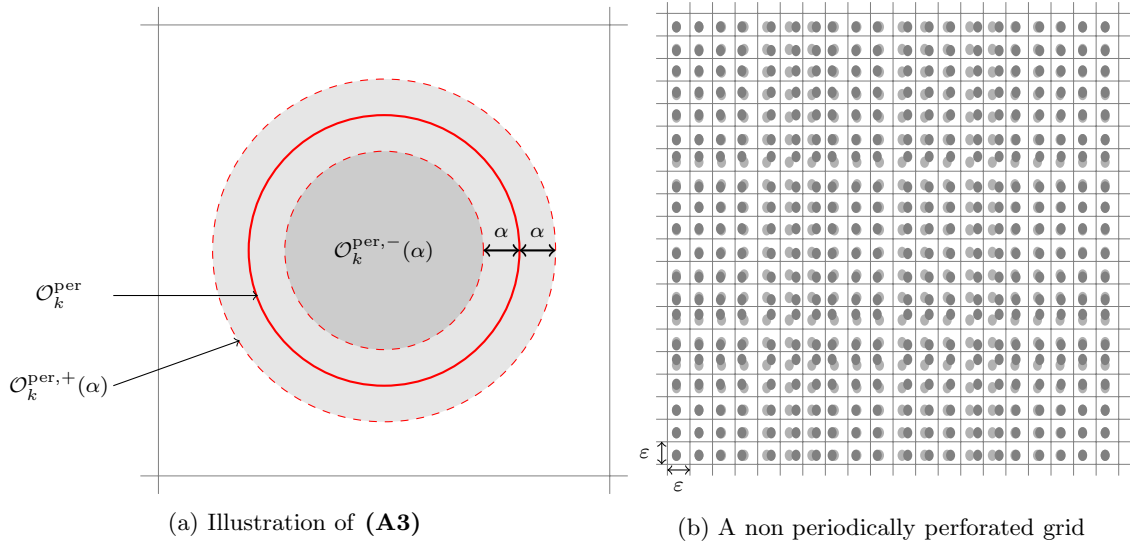

%Some consequences of the assumptions

%Domaine perforé, système de Stokes 

%fix a periodic configuration, smoothness of the solution, extension de w_j et p_j.

Assumptions \textbf{(A4)-(A5)} are the weakest possible given our method of proof. However, they are associated to PDEs and we would like a somewhat more geometric interpretation of these assumptions, in the spirit of \textbf{(A3)}. In fact, we may replace \textbf{(A4)-(A5)} by the  likely stronger (but geometric) Assumptions \textbf{(A4)'-(A5)'} below. 

\medskip

We suppose that there exist $r > 0$ and $M > 0$ such that for all $k \in \mathbb{Z}^3$ and for all $x \in \partial \mathcal{O}_k$, there exists $\zeta_x : U_x \rightarrow \mathbb{R}$ where $U_x \subset \mathbb{R}^2$, $0 \in U_x$ and $r_x > r$ such that, after eventually rotating and/or translating the local coordinate system, we have that $\zeta_x(0) = 0$ and
\begin{equation}
\left( Q_k \setminus \overline{\mathcal{O}_k} \right) \cap B(x,r_x) = \{(y_1,y_2,y_3) \in \mathbb{R}^3, y_3 > \zeta_x(y_1,y_2) \quad \mathrm{and} \quad (y_1,y_2) \in U_x \}.
\end{equation}
We assume the following uniform regularity properties:

\medskip

\noindent \textbf{(A4)'} The functions $\zeta_x, x \in \partial \mathcal{O}_k$ are Lipschitz functions with Lipschitz constant $\|\zeta_x\|_{\mathrm{Lip}(U_x)}$ satisfying $\|\zeta_x\|_{\mathrm{Lip}(U_x)} \leq M$.  

\medskip

\noindent \textbf{(A5)'} The functions $\zeta_x, x \in \partial \mathcal{O}_k$ are of class $\mathcal{C}^2$ and satisfy $\nabla \zeta_x(0)=0$ with the estimate $\| \zeta_x \|_{W^{2,\infty}(U_x)} \leq M$. 

\medskip

\noindent In Assumptions \textbf{(A4)'-(A5)'} above, we emphasize that $M$ is independent of $k$ and $x$. 

\medskip

We prove in Appendix \ref{sec:geom} that Assumptions \textbf{(A3)} and \textbf{(A4)'} imply Assumption \textbf{(A4)} and that Assumptions \textbf{(A3)} and \textbf{(A5)'} imply Assumption \textbf{(A5)}. We also note that \textbf{(A5)'} implies \textbf{(A4)'}.

%\begin{remarque}
%The above assumption is not necessarily satisfied for $k$ fixed i.e we can find locally Lipschitz domains for which each point of the boundary is represented by the graph of a function on a ball $B(x,r_x)$ wihth $r_x$ uniformly bounded from below. We can think for example of a boundary portion defined by triangles of basis $1/2^n$ and height $1/2^n$ around 0. Note that the triangles form perpendicular angles. The result seems also to be wrong for $\mathcal{C}^2$ boundary portions.
%\end{remarque}

\medskip

\begin{exemple}
We give some examples of perforations satisfying \textbf{(A1)-(A5)}:
\begin{itemize}
\item Compactly supported perturbations, that is, we change $\mathcal{O}_k^{\mathrm{per}}$ in a finite number of cells $Q_k$;
\item We remove a finite number of perforations;
\item We make $\ell^1-$translations of the periodic perforations that is we choose a sequence $(\delta_k)_{k \in \mathbb{Z}^3}$ such that $\delta_k \in \mathbb{R}^3$, $\sum_{k \in \mathbb{Z}^3} |\delta_k| < + \infty$ and for all $k \in \mathbb{Z}^3$, $\mathcal{O}_k \subset \subset Q_k$ and $\mathcal{O}_k = \mathcal{O}_k^{\mathrm{per}} + \delta_k$. 
\item We give on Figure \ref{fig6} in Appendix \ref{sec:geom} some counter-examples to Assumptions \textbf{(A3)-(A4)'-(A5)'}.

\end{itemize}
\end{exemple}

%Since $\mathcal{O}_k \subset \subset Q_k$, we get the existence of $Q'_k$ such that $\mathcal{O}_k \subset \subset Q'_k \subset \subset Q_k$. Besides, by hypothesis on the non-periodic perforations (Lemma A.3 of Poisson), the open set $Q'_k$ may be chosen uniformly in $k$ i.e of the form $Q' + k$ where $Q' \subset \subset Q$. 

%There exists a smooth open set $Q''$ such that $Q'' \subset \subset Q$ and for all $k \in \mathbb{Z}^3$, $Q_k \subset \subset Q''_k$ and $Q''_k \cap \mathcal{O} = \mathcal{O}_k$ (see figure \ref{fig5}) where $Q''_k = Q'' + k$.

\begin{remarque}
The assumption $\mathcal{O}_k \subset \subset Q_k$ is automatically implied by \textbf{(A3)} except for a finite number of cells. Dropping it would change some technical details but not the results of the paper.
\end{remarque}

\paragraph{The perforated domain.} We recall that $\Omega$ is a smooth bounded domain of $\mathbb{R}^3$. We denote 
\begin{equation}
Y_{\varepsilon} := \{k \in \mathbb{Z}^3, \varepsilon \mathcal{O}_k \subset Q_k\}
\label{eq:Y_eps}
\end{equation} and define (see Figure \ref{fig:nonper}) 
\begin{equation}
\Omega_{\varepsilon} := \Omega \setminus \bigcup_{k \in Y_{\varepsilon}} \varepsilon \overline{\mathcal{O}_k}.
\label{eq:Omega_eps}
\end{equation}
 We notice that $\Omega_{\varepsilon}$ is a bounded, locally Lipschitz and connected open subset of $\mathbb{R}^3$.

 For $f \in \left[L^2(\Omega) \right]^3$, there is a unique solution $(u_{\varepsilon},p_{\varepsilon})\in \left[H^1_0(\Omega_{\varepsilon}) \right]^3 \times L^2(\Omega_{\varepsilon})/\mathbb{R}$ to the Stokes system 
 \begin{equation}
\begin{cases}
\begin{aligned}
- \Delta u_{\varepsilon} + \nabla p_{\varepsilon} & = f \quad \mathrm{in} \quad \Omega_{\varepsilon}  \\
\mathrm{div} \ u_{\varepsilon} & = 0 \\
u_{\varepsilon} & = 0 \quad \mathrm{on} \quad \partial \Omega_{\varepsilon}.
\end{aligned}
\end{cases}
\label{eq:stokesnonper}
\end{equation}
In the sequel, we study the homogenization of $(u_{\varepsilon},p_{\varepsilon})$.  

\section{Results}
\label{section2}

The first result concerns the existence of the first order correctors. We can perform a two scale expansion of the form
\begin{equation}
u_{\varepsilon}(x) = \varepsilon^2 \left[ u_2 \left(x,\frac{x}{\varepsilon} \right) + \varepsilon u_3 \left(x,\frac{x}{\varepsilon} \right) + \cdots \right], \ \ p_{\varepsilon}(x) = p_0(x) + \varepsilon p_1 \left(x,\frac{x}{\varepsilon} \right) + \cdots
\label{eq:deuxechellesnonper}
\end{equation}
to \eqref{eq:stokesnonper} and find that 
\begin{equation}
u_2(x,y) = \sum_{j=1}^3 w_j(y)(f_j - \partial_j p_0)(x) \quad \mathrm{and} \quad p_1(x,y) = \sum_{j=1}^3 p_j(y)(f_j - \partial_j p_0)(x)
\label{eq:deuxechellesnonper2}
\end{equation}
where $f_1,f_2,f_3$ denote the components of the vector field $f$, $(w_j,p_j)$ is solution to the following Stokes system for $j=1,2,3$:
\begin{equation}
\begin{cases}
\begin{aligned}
- \Delta w_j + \nabla p_j & = e_j \ \ \mathrm{in} \quad \mathbb{R}^3 \setminus \overline{\mathcal{O}} \\
\mathrm{div} \ w_j & = 0 \\ 
w_{j} & = 0 \quad \mathrm{on} \quad \partial \mathcal{O}.
\end{aligned}
\end{cases}
\label{eq:cor}
\end{equation}
and $p_0$ is given by the Darcy's law \eqref{eq:darcy}.

\begin{theoreme}[Existence of correctors]
Suppose that Assumptions \textbf{(A1)-(A3)} and \textbf{(A4)$_0$} are satisfied. For all $j \in \{1,2,3\}$, System \eqref{eq:cor} admits a solution $(w_j,p_j)$ of the form
$$w_j = w_j^{\mathrm{per}} + \widetilde{w_j} \quad \mathrm{and} \quad p_j = p_j^{\mathrm{per}} + \widetilde{p_j}$$ where
$(\widetilde{w_j},\widetilde{p_j}) \in \left[ H^1(\mathbb{R}^3 \setminus \overline{\mathcal{O}}) \right]^3 \times  L^2_{\mathrm{loc}}(\mathbb{R}^3 \setminus \overline{\mathcal{O}}).$
Moreover, we have the following estimate
$$\| \widetilde{p_j} - \langle \widetilde{p_j} \rangle \|_{L^2(\frac{1}{\varepsilon}\Omega_\varepsilon)} \leq C \varepsilon^{- 1},$$
where $C$ is a constant independent of $\varepsilon$ and $\langle\widetilde{ p_j} \rangle$ denotes the mean value of $\widetilde{p_j}$ on $\frac{1}{\varepsilon} \Omega_{\varepsilon}$.
\label{th:cor}
\end{theoreme}

We define $$R_{\varepsilon} := u_{\varepsilon} - \varepsilon^2 \sum_{j=1}^3 w_j \left( \frac{\cdot}{\varepsilon} \right)(f_j - \partial_j p_0) \ \ \mathrm{and} \ \ \pi_{\varepsilon} := p_{\varepsilon} - p_0 - \varepsilon \sum_{j=1}^3 p_j \left( \frac{\cdot}{\varepsilon} \right)(f_j - \partial_j p_0).$$ Following the ideas of the proof of \cite[Theorem 1.3]{allaire1991continuity}, we can prove under the assumption $f \in \left[ W^{3,\infty}(\Omega) \right]^3$ that $R_{\varepsilon}/\varepsilon^2 \underset{\varepsilon \rightarrow 0}{\longrightarrow} 0$ in the non-periodic setting for the $\left[L^2(\Omega)\right]^3-$norm (where it is understood that $u_{\varepsilon}$ and $w_j, j=1,2,3$ are extended by zero in the perforations). This fact, though relevant because it makes \eqref{eq:deuxechellesnonper} rigorous, is not strong enough to justify the construction of the non-periodic correctors $(w_j,p_j), j=1,2,3$. Indeed, if we set $$R_{\varepsilon}^{\mathrm{per}} := u_{\varepsilon} - \varepsilon^2 \sum_{j=1}^3 w_j^{\mathrm{per}} \left( \frac{\cdot}{\varepsilon} \right)(f_j - \partial_j p_0),$$
we notice that $$R_{\varepsilon} = R_{\varepsilon}^{\mathrm{per}} - \varepsilon^2 \sum_{j=1}^3\widetilde{w_j} \left( \frac{\cdot}{\varepsilon} \right)(f_j - \partial_j p_0).$$
Since $\widetilde{w_j} \in \left[ L^2(\mathbb{R}^3 \setminus \overline{\mathcal{O}}) \right]^3$, one has for $j = 1,2,3$:
$$
\begin{aligned}\left\| \widetilde{w_j} \left(\frac{\cdot}{\varepsilon} \right)(f_j - \partial_j p_0) \right\|_{\left[L^2(\Omega_{\varepsilon})\right]^3} & = \varepsilon^{\frac{3}{2}} \left\| \widetilde{w_j}(f_j - \partial_j p_0)(\varepsilon \cdot) \right\|_{\left[L^2(\frac{1}{\varepsilon}\Omega_{\varepsilon})\right]^3} \\ & \leq \varepsilon^{\frac{3}{2}} \| \widetilde{w_j} \|_{\left[L^2(\mathbb{R}^3)\right]^3} \| f_j - \partial_j p_0 \|_{L^{\infty}(\Omega)} = C \varepsilon^{\frac{3}{2}}.
\end{aligned}$$ 
Thus $R_{\varepsilon}^{\mathrm{per}} / \varepsilon^2 = R_{\varepsilon}/\varepsilon^2 + O(\varepsilon^{3/2})$. This proves that $R_{\varepsilon}^{\mathrm{per}}/\varepsilon^2 \underset{\varepsilon \rightarrow 0}{\longrightarrow} 0$ for the $\left[L^2(\Omega)\right]^3-$norm. So, using $w_j^{\mathrm{per}}$ instead of $w_j$ does not change the convergence of $u_{\varepsilon}$ to its first order asymptotic expansion. 

\medskip

Yet, since $w_j$ and $p_j, j=1,2,3$ are the \textit{ad hoc} correctors for the non-periodic setting, there must be situations highlighting that the approximation of $u_{\varepsilon}$ (resp. $p_{\varepsilon}$) by $\varepsilon^2 w_j \left( \cdot / \varepsilon \right)(f_j - \partial_j p_0)$ (resp. $p_0 + \varepsilon p_j \left( \cdot / \varepsilon \right)(f_j - \partial_j p_0)$) is improved in some sense when we use $w_j$ instead of $w_j^{\mathrm{per}}$. 
We exhibit in Theorem \ref{th:convergence} such a situation (see Remark \ref{re:utilite}). 

\medskip

Before stating Theorem \ref{th:convergence}, we obtain in Theorem \ref{th:masmoudi} $H^2-$estimates for the solution of a Stokes system posed in $\Omega_{\varepsilon}$ (see \cite[Theorem 4.1]{masmoudi2004some} for the periodic case). 

\begin{theoreme}[Estimates for a Stokes problem] Suppose that Assumptions \textbf{(A4)$_0$} and \textbf{(A5)} are satisfied.
Let $f \in \left[ L^2(\Omega_{\varepsilon}) \right]^3$ and $(u,p) \in \left[ H^1_0(\Omega_{\varepsilon}) \right]^3 \times L^2(\Omega_{\varepsilon})/ \mathbb{R}$ be solution of
\begin{equation}
\begin{cases}
\begin{aligned}
- \Delta u_{\varepsilon} + \nabla p_{\varepsilon}  & = f \quad \mathrm{in} \quad \Omega_{\varepsilon} \\
\mathrm{div}(u_{\varepsilon}) & = 0 \\
u_{\varepsilon} & = 0 \quad \mathrm{on} \quad \partial \Omega_{\varepsilon}.
\end{aligned}
\end{cases}
\label{eq:masmoudi1}
\end{equation}
Then $(u_{\varepsilon},p_{\varepsilon}) \in \left[H^2(\Omega_{\varepsilon}) \right]^3 \times H^1(\Omega_{\varepsilon})/\mathbb{R}$ and there exists a constant $C > 0$ such that for any domain $\Omega'' \subset \subset \Omega$ and all $\varepsilon < \varepsilon_0(\Omega'')$,
$$
\begin{aligned}
\|D^2 u_{\varepsilon}\|_{\left[ L^2(\Omega'' \cap \Omega_{\varepsilon}) \right]^3} + \varepsilon^{-1} \| \nabla u_{\varepsilon}\|_{\left[ L^2(\Omega_{\varepsilon}) \right]^3} & + \varepsilon^{-2} \| u_{\varepsilon}\|_{\left[ L^2(\Omega_{\varepsilon}) \right]^3} \\ & + \|\nabla p_{\varepsilon}\|_{L^2(\Omega'' \cap \Omega_{\varepsilon})^3} + \| p_{\varepsilon}\|_{L^2(\Omega_{\varepsilon})/\mathbb{R}} \leq C \|f \|_{\left[ L^2(\Omega_{\varepsilon}) \right]^3}.
\end{aligned}
$$
Furthermore, the couple $(u_{\varepsilon},p_{\varepsilon})$ is unique in $\left[H^1(\Omega_{\varepsilon}) \right]^3 \times L^2(\Omega_{\varepsilon})/\mathbb{R}$.
\label{th:masmoudi}
\end{theoreme}

\begin{theoreme}[Convergence Theorem] Suppose that assumptions \textbf{(A1)-(A5)} are satisfied.
Let $f \in \left[ W^{3,\infty}(\Omega) \right]^3$ be such that $\mathrm{div}(Af) = 0$ and $f$ is compactly supported in $\Omega$. There exists a constant $C > 0$ such that for all $\varepsilon > 0$ small enough and all domain $\Omega'' \subset \subset \Omega$,
\begin{equation}
\begin{aligned}
\left\|D^2 \left[ u_{\varepsilon} - \varepsilon^2 w_j \left( \frac{\cdot}{\varepsilon} \right) f_j \right] \right\|_{\left[ L^2(\Omega'' \cap \Omega_{\varepsilon}) \right]^3}
+ & \varepsilon^{-1} \left\|\nabla \left[ u_{\varepsilon} - \varepsilon^2 w_j \left( \frac{\cdot}{\varepsilon} \right) f_j \right] \right\|_{\left[ L^2(\Omega_{\varepsilon}) \right]^3} \\
& + \varepsilon^{-2} \left\| u_{\varepsilon} - \varepsilon^2 w_j \left( \frac{\cdot}{\varepsilon} \right) f_j \right\|_{\left[L^2(\Omega_{\varepsilon}) \right]^3} \leq C \varepsilon
\end{aligned}
\end{equation}
and
\begin{equation}
\left\| \nabla \left[ p_{\varepsilon} - \varepsilon \left\{ p_j  \left( \frac{\cdot}{\varepsilon} \right) - \lambda_{\varepsilon}^j \right\} f_j \right] \right\|_{L^2(\Omega'' \cap \Omega_{\varepsilon})} + \left\| p_{\varepsilon} - \varepsilon \left\{ p_j  \left( \frac{\cdot}{\varepsilon} \right) - \lambda_{\varepsilon}^j \right\} f_j \right\|_{L^2(\Omega_{\varepsilon})/\mathbb{R}} \leq C \varepsilon,
\end{equation}
where $$\lambda_{\varepsilon}^j = \frac{1}{|\Omega_{\varepsilon}|}\int_{\Omega_{\varepsilon}} p_j\left(\frac{\cdot}{\varepsilon} \right).$$
\label{th:convergence}
\end{theoreme}

\begin{remarque}
We note that Theorem \ref{th:masmoudi} and Theorem \ref{th:convergence} are valid in the periodic case (that is in the framework of subsection \ref{section11}). This provides a new situation in which quantitative error estimates can be obtained, besides th ones of \cite{maruvsic1997asymptotic,shen}.
\end{remarque}

\begin{remarque}
%Question des conditions aux bords justification de l'hypothèse $\mathrm{div}(Af) = 0$, donnée aux bords périodique
\label{re:bord}
The assumptions $\mathrm{div}(Af) = 0$ and $f$ compactly supported in $\Omega$ make boundary effects disappear. Indeed, it is straightforward to see that in this case $\nabla p_0 = 0$ in $\Omega$ (see \eqref{eq:darcy}). Since $f$ is compactly supported, we have $\varepsilon^2 w_j(\cdot/\varepsilon)f_j = 0$ on $\partial \Omega$, so $u_{\varepsilon}$ and its first order expansion coincide on $\partial \Omega$. This explains why the $O(\varepsilon^2)$ $H^1-$convergence rate of $R_{\varepsilon}$ obtained in Theorem \ref{th:convergence} is sharper than the $O(\varepsilon^{3/2})$ $H^1-$convergence rate obtained in \cite[Theorem 1.1]{shen}.  %Another possibility to avoid boundary effects would be to consider macroscopic periodic domains, as in Remark \ref{re:per}. 
\end{remarque}

\begin{remarque}
\label{re:utilite}
By applying Theorem \ref{th:convergence}, we get that $R_{\varepsilon} \in H^2(\Omega_{\varepsilon})$. We now note that, in general, one has $R_{\varepsilon}^{\mathrm{per}} \notin H^2(\Omega_{\varepsilon})$. This follows from the fact that  $w_j^{\mathrm{per}} \left(\frac{\cdot}{\varepsilon}\right) \notin H^2(\Omega_{\varepsilon})$ (unless of course $\Omega_{\varepsilon} = \Omega_{\varepsilon}^{\mathrm{per}}$) for $j = 1,2,3$. This is due to the normal derivative jumps of $w_j^{\mathrm{per}}(\cdot/\varepsilon)$ along the parts of $\varepsilon \partial \mathcal{O}^{\mathrm{per}}$ that are included in $\Omega_{\varepsilon}$. This shows that, in the non-periodic case, using the periodic corrector in \eqref{eq:deuxechellesnonper2} does not give the expected convergence rate, contrary to the non-periodic corrector. 
%Justification de l'utilité du correcteur non-périodique
\end{remarque}

\begin{remarque}
Theorem~\ref{th:masmoudi} and Theorem~\ref{th:convergence} can be proved up to the boundary of $\Omega$ with the same convergence rates when $\Omega$ is of class $\mathcal{C}^2$. The proof is rather technical and will be omitted here. 
\end{remarque}

\begin{remarque} Theorem~\ref{th:masmoudi} can be proved for the $H^m-$norm, $m > 0$ in the periodic domain $\Omega_{\varepsilon}^{\mathrm{per}}$ (see \cite[Theorem 4.2]{masmoudi2004some}) and in the non-periodic domain $\Omega_{\varepsilon}$, provided that we require higher regularity of $\mathcal{O}_k$ in \textbf{(A5)'} (typically that $\mathcal{O}_k$ is uniformly with respect to $k$ of class $\mathcal{C}^{m+2}$, see \cite[Theorem IV.5.1]{galdi2011introduction}): if $f \in \left[ H^m (\Omega_{\varepsilon})\right]^3$, then $(u_{\varepsilon},p_{\varepsilon}) \in \left[ H^{m+2}(\Omega_{\varepsilon}) \right]^3 \times H^{m+1}(\Omega_{\varepsilon})/\mathbb{R}$ and
there exists a constant $C$ independent of $\varepsilon$ such that
$$\| D^{m+2} u_{\varepsilon} \|_{\left[ L^2(\Omega'' \cap \Omega_{\varepsilon})\right]^3} + \| D^{m+1} p_{\varepsilon} \|_{L^2(\Omega'' \cap \Omega_{\varepsilon})} \leq C \sum_{i=0}^m \frac{1}{ \varepsilon^{i}} \| D^{m-i} f \|_{\left[ L^2(\Omega_{\varepsilon}) \right]^3}.$$
\end{remarque}

\begin{remarque}
This paper presents only the three dimensional case. All that follows is true in dimension greater than 3. As for the two dimensional case, Theorem \ref{th:cor} and Theorem \ref{th:masmoudi} are valid.
%Autres dimensions ?
\end{remarque}

%Re: renvoyer à la question de la justification de l'utilité

%- \textbf{Notre problème} : problème de Stokes dans un ouvert perforé périodiquement avec des défauts. \\

%- \textbf{Littérature} : 1) ce qui existe déjà en périodique\\
%2) Poisson. \\

%- \textbf{Résultats} :\\
%1) Estimations $H^2$ de Masmoudi\\
%2) Existence de $w$. \\
%2') Ce qui est connu en périodique dans le cas $f$ général : %$R_{\varepsilon} \rightarrow 0$. En remarque : on l'a aussi en non-périodique mais le $\widetilde{w}$ ne sert à rien \\
%3) Existence de $z$ \\
%4) Cas où on voit $\widetilde{w}$. 

 The rest of the paper is devoted to proofs. In Section~\ref{sect:proofmasmoudi}, we give the proof of Theorem~\ref{th:masmoudi} in both periodic and non periodic perforated domains. We next prove in Section~\ref{sect:proofcorr} the existence of the non-periodic correctors. Finally, Section~\ref{sect:proofcvth} is devoted to the proof of the convergence Theorem~\ref{th:convergence}. Some technical Lemmas, especially concerning divergence problems, are postponed to Appendix \ref{sect:appendix}.

\section{Proofs}
\label{sectproof}

\subsection{Proof of Theorem \ref{th:masmoudi}}
\label{sect:proofmasmoudi}

We first state the following Poincar\'e-Friedrichs inequality:

\begin{lemme}  Suppose that Assumptions \textbf{(A1)} and \textbf{(A3)} are satisfied.
There exists a constant $C > 0$ independent of $\varepsilon$ such that for all $u \in \left[ H^1_0(\Omega_{\varepsilon}) \right]^3$, one has 
$$\int_{\Omega_{\varepsilon}} |u|^2 \leq C \varepsilon^2 \int_{\Omega_{\varepsilon}} | \nabla u |^2.$$
\label{lem:poinc}
\end{lemme}

\begin{proof}  
We recall that $Y_{\varepsilon}$ is defined by \eqref{eq:Y_eps} and we define $Z_{\varepsilon} := \{ k \in \mathbb{Z}^d, \varepsilon Q_k \cap \partial \Omega \neq \emptyset \}$. We have the decomposition 
\begin{equation}
\Omega_{\varepsilon} = \left( \bigcup_{k \in Y_{\varepsilon}} \varepsilon (\overline{Q_k} \setminus \overline{\mathcal{O}_k}) \right) \cup \left( \bigcup_{k \in Z_{\varepsilon}} \left[(\varepsilon \overline{Q_k}) \cap \Omega \right] \right).
\label{eq:poinc3}
\end{equation}
Thanks to Assumption \textbf{(A3)} and the proof of \cite[Lemma 3.2]{BW}, we know that there exists a constant $C > 0$ independent of $k$ and $\varepsilon$ such that for all $k \in Y_{\varepsilon}$, we have the inequality
\begin{equation}
\int_{\varepsilon (Q_k \setminus \overline{\mathcal{O}_k})} u^2 \leq C \varepsilon^2 \int_{\varepsilon (Q_k \setminus \overline{\mathcal{O}_k})} |\nabla u|^2.
\label{eq:poinc1}
\end{equation}
We now fix $k \in Z_{\varepsilon}$. Thanks to the proof of \cite[Lemma 1]{tartar127incompressible}, there exists a constant $C > 0$ which is independent of $k$ and $\varepsilon$ such that 
\begin{equation}
\int_{(\varepsilon Q_k) \cap \Omega} u^2 \leq C \varepsilon^2 \int_{(\varepsilon Q_k) \cap \Omega} |\nabla u|^2.
\label{eq:poinc2}
\end{equation}
Summing the estimate \eqref{eq:poinc1} over $k \in Y_{\varepsilon}$, the estimate \eqref{eq:poinc2} over $k \in Z_{\varepsilon}$ and using \eqref{eq:poinc3} concludes the proof of Lemma \ref{lem:poinc}.
\end{proof}

Let $(u_{\varepsilon},p_{\varepsilon})$ be the solution of \eqref{eq:stokesnonper}. We have by classical energy estimates the following inequalities:
\begin{equation}
\left(\int_{\Omega_{\varepsilon}} |\nabla u_{\varepsilon} |^2 \right)^{\frac{1}{2}} \leq C \varepsilon \| f \|_{\left[ L^2(\Omega_{\varepsilon})\right]^3} \ \ \ \ \mathrm{and} \ \ \ \ \left(\int_{\Omega_{\varepsilon}} | u_{\varepsilon} |^2 \right)^{\frac{1}{2}} \leq C \varepsilon^2 \| f \|_{\left[L^2(\Omega_{\varepsilon})\right]^3}
\label{eq:thmasmoudi00}
\end{equation}
which will be useful in the proof of Theorem \ref{th:masmoudi}.

\begin{proof}
[Proof of Theorem \ref{th:masmoudi}] In this proof, $C$ will denote various constants independent of $\varepsilon$ that can change from one line to another. We fix $\Omega'' \subset \subset \Omega$. We first show the following estimate:
\begin{equation}
\| D^2 u_{\varepsilon} \|_{\left[L^2(\Omega'' \cap \Omega_{\varepsilon})^3\right]^{3 \times 3}} + \| \nabla p_{\varepsilon} \|_{\left[ L^2(\Omega'' \cap \Omega_{\varepsilon}) \right]^3} \leq C \left[\varepsilon^{-1} \| \nabla u_{\varepsilon} \|_{\left[L^2(\Omega_{\varepsilon}) \right]^{3 \times 3}} + \varepsilon^{-2}\| u_{\varepsilon} \|_{\left[ L^2(\Omega_{\varepsilon}) \right]^3} + \| f \|_{\left( L^2(\Omega_{\varepsilon}) \right]^3} \right].
\label{eq:masmoudi0}
\end{equation}
\underline{Proof of \eqref{eq:masmoudi0}}:
we study Problem \eqref{eq:masmoudi1} on each periodic cell $Q_k \setminus \overline{\mathcal{O}_k}$. 
Let $k \in Y_{\varepsilon}$, where $Y_{\varepsilon}$ is defined in \eqref{eq:Y_eps}. We recall that $Q''_k$ is introduced in \eqref{eq:Q''} and we define in $Q''_k \setminus \overline{\mathcal{O}_k}$ the functions
$$
\begin{cases}
\begin{aligned}
U_{\varepsilon}^k & := \varepsilon^{-2} u_{\varepsilon}(\varepsilon \cdot) \\
P_{\varepsilon}^k & := \varepsilon^{-1} p_{\varepsilon}(\varepsilon \cdot) - \lambda_k \\
F_{\varepsilon}^k & := f(\varepsilon \cdot) 
\end{aligned}
\end{cases}
$$
where $\lambda_k \in \mathbb{R}$ is chosen such that 
$$\int_{Q''_k \setminus \overline{\mathcal{O}_k}} P_{\varepsilon}^k = 0.$$
Then $(U_{\varepsilon}^k,P_{\varepsilon}^k) \in \left[ H^1\left(Q''_k \setminus \overline{\mathcal{O}_k} \right) \right]^3 \times L^2 \left( Q''_k \setminus \overline{\mathcal{O}_k} \right)$ and $(U_{\varepsilon}^k,P_{\varepsilon}^k)$ is solution to the following Stokes system 
\begin{equation}
\begin{cases}
\begin{aligned}
- \Delta U_{\varepsilon}^k + \nabla P_{\varepsilon}^k & = F_{\varepsilon}^k \quad \mathrm{in} \quad Q''_k \setminus \overline{\mathcal{O}_k}  \\
\mathrm{div}(U_{\varepsilon}^k) & = 0 \\
U_{\varepsilon}^k & = 0 \quad \mathrm{on} \quad \partial \mathcal{O}_k.
\end{aligned}
\end{cases}
\label{eq:masmoudi6}
\end{equation}
%We define $\Omega_k := Q'' \backslash \left[\overline{\mathcal{O}_k} - k \right]$ and
%we apply Theorem IV.5.1 of \cite{galdi2011introduction} on $\Omega_k$ to \eqref{eq:masmoudi6}. 
%Un point me trouble dans ce Théorème. A priori, rien n'interdit de choisir $\Omega' $ tel que $\partial \Omega' \cap \partial \Omega_0 \supset \sigma$, un bout de la frontière n'apparaît donc pas dans le membre de droite...
By applying Assumption \textbf{(A5)} to System \eqref{eq:masmoudi6}, we get the estimate
\begin{equation}
\left\| U_{\varepsilon}^k \right\|_{\left[ H^2(Q_k \setminus  \overline{\mathcal{O}_k } )\right]^3} + \| P_{\varepsilon}^k \|_{H^1 (Q_k \setminus  \overline{\mathcal{O}_k } )} \leq C \left[ \left\| U_{\varepsilon}^k \right\|_{\left[ H^1 (Q''_k \setminus  \overline{\mathcal{O}_k } ) \right]^3} + \left\| P_{\varepsilon}^k \right\|_{ L^2 (Q''_k \setminus  \overline{\mathcal{O}_k } ) } + \left\| F_{\varepsilon}^k \right\|_{\left[ L^2 (Q''_k \setminus  \overline{\mathcal{O}_k } ) \right]^3 }\right].
\label{eq:thmasmoudi1}
\end{equation}
Assumption \textbf{(A4)$_0$} and \cite[Lemma III.3.2]{galdi2011introduction} applied with $\Omega_1 := Q_k \setminus \overline{\mathcal{O}_k}$ and $\Omega_2 := Q''_k \setminus Q_k$ give a function $v \in \left[H^1_0(Q''_k \setminus \mathcal{O}_k) \right]^3$ such that $\mathrm{div}(v) = P_{\varepsilon}^k$ and 
\begin{equation}\|v\|_{\left[H^1(Q''_k \setminus \mathcal{O}_k) \right]^3} \leq C \| P_{\varepsilon}^k \|_{L^2(Q''_k \setminus \mathcal{O}_k)},
\label{eq:thmasmoudi21}
\end{equation}
where $C$ is independent of $k$. Thus,
\begin{equation}
\|P_{\varepsilon}^k \|_{L^2(Q''_k \setminus \overline{\mathcal{O}_k})}^2 = \langle \nabla P_{\varepsilon}^k, v \rangle_{H^{-1}\times H^1_0(Q''_k \setminus \overline{\mathcal{O}_k})} \leq \| \nabla P_{\varepsilon}^k \|_{\left[ H^{-1}(Q''_k \setminus \overline{\mathcal{O}_k}) \right]^3} \| v \|_{\left[ H^1(Q''_k \setminus \overline{\mathcal{O}_k}) \right]^3}.
\label{eq:thmasmoudi22}
\end{equation}
Gathering together \eqref{eq:thmasmoudi21} and \eqref{eq:thmasmoudi22} yields
\begin{equation}
\|P_{\varepsilon}^k \|_{L^2(Q''_k \setminus \overline{\mathcal{O}_k})} \leq C \| \nabla P_{\varepsilon}^k \|_{\left[ H^{-1}(Q''_k \setminus \overline{\mathcal{O}_k}) \right]^3}.
\label{eq:thmasmoudi2}
\end{equation}
The triangle inequality applied to the first equation of \eqref{eq:masmoudi6} then provides the inequality
\begin{equation}
\begin{aligned}
\| \nabla P_{\varepsilon}^k \|_{\left[ H^{-1}(Q''_k \setminus \overline{\mathcal{O}_k}) \right]^3} & \leq \| \Delta U_{\varepsilon}^k \|_{\left[ H^{-1}(Q''_k \setminus \overline{\mathcal{O}_k})\right]^3} + \| F_{\varepsilon}^k \|_{\left[ H^{-1}(Q''_k \setminus \overline{\mathcal{O}_k}) \right]^3} \\
& \leq \| \nabla U_{\varepsilon}^k \|_{\left[ L^2(Q''_k \setminus \overline{\mathcal{O}_k})\right]^{3 \times 3}} + \| F_{\varepsilon}^k \|_{\left[L^2(Q''_k \setminus \overline{\mathcal{O}_k})\right]^3}.
\end{aligned}
\label{eq:thmasmoudi}
\end{equation}
Collecting \eqref{eq:thmasmoudi1}, \eqref{eq:thmasmoudi2} and \eqref{eq:thmasmoudi}, we get 
$$
\left\| U_{\varepsilon}^k \right\|_{\left[ H^2(Q_k \setminus \overline{\mathcal{O}_k} )\right]^3 } + \| P_{\varepsilon}^k \|_{H^1(Q_k \setminus \overline{\mathcal{O}_k} )} \leq C \left[ \left\| U_{\varepsilon}^k \right\|_{\left[H^1(Q''_k \setminus \overline{\mathcal{O}_k} )\right]^3} + \left\| F_{\varepsilon}^k \right\|_{\left[ L^2(Q''_k \setminus \overline{\mathcal{O}_k} ) \right]^3}\right].$$
In particular, we deduce 
\begin{equation}
\left\|D^2 U_{\varepsilon}^k \right\|_{\left[L^2(Q_k \setminus \overline{\mathcal{O}_k} )^3\right]^{3\times 3}} + \| \nabla P_{\varepsilon}^k \|_{\left[L^2(Q_k \setminus \overline{\mathcal{O}_k} ) \right]^3} \leq C \left[ \left\| U_{\varepsilon}^k \right\|_{\left[ H^1(Q''_k \setminus \overline{\mathcal{O}_k} ) \right]^3} + \left\| F_{\varepsilon}^k \right\|_{\left[L^2(Q''_k \setminus \overline{\mathcal{O}_k} )\right]^3}\right].
\label{eq:masmoudi7}
\end{equation}

\noindent Scaling back \eqref{eq:masmoudi7} gives
\begin{equation}
\begin{aligned}
\| D^2 u_{\varepsilon} \|_{\left[ L^2(\varepsilon Q_k \backslash \overline{\mathcal{O}_k})^3 \right]^{3 \times 3}} + \| \nabla p_{\varepsilon} \|_{\left[ L^2(\varepsilon Q_k \backslash \overline{\mathcal{O}_k}) \right]^3} & \leq C \big[\varepsilon^{-1} \| \nabla u_{\varepsilon} \|_{\left[ L^2(\varepsilon Q''_k \backslash \overline{\mathcal{O}_k}) \right]^{3 \times 3}}  \\ & + \varepsilon^{-2}\| u_{\varepsilon} \|_{\left[ L^2(\varepsilon Q''_k \backslash \overline{\mathcal{O}_k})\right]^3} + \| f \|_{\left[ L^2(\varepsilon Q''_k \backslash \overline{\mathcal{O}_k}) \right]^3} \big].
\label{eq:masmoudi2}
\end{aligned}
\end{equation}
Thus,
\begin{equation}
\begin{aligned}
\| D^2 u_{\varepsilon} \|_{\left[L^2(\varepsilon Q_k \backslash \overline{\mathcal{O}_k})^3\right]^{3 \times 3}}^2  +  \| \nabla p_{\varepsilon} \|_{\left[ L^2(\varepsilon Q_k \backslash \overline{\mathcal{O}_k})\right]^3}^2 
& \leq C \big[\varepsilon^{-2} \| \nabla u_{\varepsilon} \|_{\left[ L^2(\varepsilon Q''_k \backslash \overline{\mathcal{O}_k}) \right]^{3 \times 3}}^2 \\ & + \varepsilon^{-4}\| u_{\varepsilon} \|_{\left[ L^2(\varepsilon Q''_k \backslash \overline{\mathcal{O}_k})\right]^{3}}^2 + \| f \|_{\left[ L^2(\varepsilon Q''_k \backslash \overline{\mathcal{O}_k})\right]^3}^2 \big].
\label{eq:masmoudi5}
\end{aligned}
\end{equation}
We next sum \eqref{eq:masmoudi5} over $k \in \widetilde{Y_{\varepsilon}}$ where $$\widetilde{Y_{\varepsilon}} :=  Y_{\varepsilon} \setminus \left\{ k \in \mathbb{Z}^3, \ \mathrm{d}(\varepsilon Q_k,\Omega^c) > \varepsilon \right\}.$$ 
We note that for $\varepsilon < \varepsilon_0(\Omega'')$, we have the inclusion
$$\Omega'' \cap \Omega_{\varepsilon} \subset \bigcup_{k \in \widetilde{Y_{\varepsilon}}} \varepsilon Q''_k \backslash \overline{\mathcal{O}_k} \subset \Omega_{\varepsilon}.$$ 
We get
\begin{equation}
\begin{aligned}
\| D^2 u_{\varepsilon} \|_{\left[ L^2(\Omega'' \cap \Omega_{\varepsilon})^3 \right]^{3 \times 3}}^2 + \| \nabla p_{\varepsilon} \|_{\left[ L^2(\Omega'' \cap \Omega_{\varepsilon})\right]^3}^2 & \leq C \big[\varepsilon^{-2} \| \nabla u_{\varepsilon} \|_{\left[ L^2(\Omega_{\varepsilon})\right]^{3 \times 3}}^2 \\ & + \varepsilon^{-4}\| u_{\varepsilon} \|_{\left[L^2(\Omega_{\varepsilon})\right]^3}^2 + \|f \|^2_{L^2(\Omega_{\varepsilon})} \big].
\end{aligned}
\label{eq:masmoudi3}
\end{equation}
Estimate \eqref{eq:masmoudi0} is proved. We now conclude the proof of Theorem \ref{th:masmoudi}. We have, inserting \eqref{eq:thmasmoudi00} in the right hand side of \eqref{eq:masmoudi0},
$$\| D^2 u_{\varepsilon} \|_{\left[ L^2(\Omega'' \cap \Omega_{\varepsilon})^3\right]^{3 \times 3}} + \varepsilon^{-1} \| \nabla u_{\varepsilon} \|_{\left[L^2(\Omega_{\varepsilon})\right]^{3 \times 3}} + \varepsilon^{-2}\| u_{\varepsilon} \|_{\left[ L^2(\Omega_{\varepsilon})\right]^3} + \| \nabla p_{\varepsilon} \|_{\left[ L^2(\Omega'' \cap \Omega_{\varepsilon})\right]^3} \leq C \|f \|_{\left[ L^2(\Omega_{\varepsilon})\right]^3}.$$
It remains to show that
 \begin{equation}
 \| p_{\varepsilon} \|_{L^2(\Omega_{\varepsilon})/\mathbb{R}} \leq C \|f\|_{\left[L^2(\Omega_{\varepsilon})\right]^3}.
 \label{eq:thmasmoudi10}
 \end{equation} By Lemma \ref{lem:estimpres} stated in the appendix and the first line of \eqref{eq:masmoudi6}, we get
$$\| p_{\varepsilon}\|_{L^2(\Omega_{\varepsilon})/\mathbb{R}} \leq C \varepsilon^{-1} \left[ \| \nabla u_{\varepsilon} \|_{\left[ L^2(\Omega_{\varepsilon})\right]^{3\times 3}} + C\| f \|_{\left[ H^{-1}(\Omega_{\varepsilon})\right]^3} \right].$$
We now show that  
\begin{equation}
\| f \|_{\left[ H^{-1}(\Omega_{\varepsilon})\right]^3} \leq C \varepsilon \| f \|_{\left[L^2(\Omega_{\varepsilon})\right]^3}.
\label{eq:thmasmoudi8}
\end{equation}
Indeed, for any $\phi \in \left[H^1_0(\Omega_{\varepsilon})\right]^3$, we write that, using successively Cauchy-Schwarz inequality and Poincar\'e inequality (see Lemma \ref{lem:poinc}),
$$\begin{aligned} 
\langle f, \phi \rangle  = \int_{\Omega_{\varepsilon}} f \cdot \phi \leq \| f \|_{\left[ L^2(\Omega_{\varepsilon})\right]^3} \| \phi \|_{\left[L^2(\Omega_{\varepsilon})\right]^3} & \leq C \varepsilon \| f \|_{\left[ L^2(\Omega_{\varepsilon}) \right]^3} \| \nabla \phi \|_{\left[ L^2(\Omega_{\varepsilon})\right]^{3 \times 3}} \\ & \leq C \varepsilon \| f \|_{\left[L^2(\Omega_{\varepsilon})\right]^3} \| \phi \|_{\left[H^1_0(\Omega_{\varepsilon})\right]^3}
\end{aligned}$$  
Thus \eqref{eq:thmasmoudi8}. 
Finally, we conclude with the use of \eqref{eq:thmasmoudi00} that 
$$\| p_{\varepsilon} \|_{L^2(\Omega_{\varepsilon})/\mathbb{R}} \leq C \varepsilon^{-1} \| \nabla u_{\varepsilon} \|_{\left[ L^2(\Omega_{\varepsilon})\right]^{3 \times 3}} + C \| f \|_{\left[ L^2(\Omega_{\varepsilon})\right]^3} \leq C \|f\|_{\left[L^2(\Omega_{\varepsilon})\right]^3}.$$
This proves \eqref{eq:thmasmoudi10} and concludes the proof of Theorem \ref{th:masmoudi}.
\end{proof}

%Poincare + estimations H^1

%On refait Masmoudi dans un cadre Hilbertien avec des outils plus simples. \\

%\textbf{Théorème} : estimation $H^2$ en non périodique \\

%\textbf{Rem} : on peut avoir des estimations $H^s$, $s$ quelconque \\

%\underline{Estimation jusqu'au bord macro ?}

%\underline{Uniformité en les cellules des constantes de régularité?}

\subsection{Proof of Theorem \ref{th:cor}}
\label{sect:proofcorr}

We use the periodic correctors $(w_j^{\mathrm{per}},p_j^{\mathrm{per}})$ defined in \eqref{eq:correcteur} and we search $w_j$ and $p_j$ in the form $w_j = w_j^{\mathrm{per}} + \widetilde{w_j}$ and $p_j = p_j^{\mathrm{per}} + \widetilde{p_j}$. We recall (see the last paragraph of Subsection \ref{section12}) that $w_j^{\mathrm{per}}$ is extended by zero in $\mathcal{O}^{\mathrm{per}}$ and that $p_j^{\mathrm{per}}$ is extended by a constant $\lambda_j$. The Stokes system defining $(\widetilde{w_j},\widetilde{p_j})$ is
\begin{equation}
\begin{cases}
\begin{aligned}
- \Delta \widetilde{w_j} + \nabla \widetilde{p_j} & = e_j + \Delta w_j^{\mathrm{per}} - \nabla p_j^{\mathrm{per}} \quad \mathrm{in} \quad \mathbb{R}^3 \setminus \overline{\mathcal{O}} \\
\mathrm{div} \ \widetilde{w_j} &= 0 \\
\widetilde{w_j} & = \CC - \BB w_j^{\mathrm{per}} \quad \mathrm{on} \quad \partial \mathcal{O}.
\end{aligned}
\end{cases}
\label{eq:theoremcor1}
\end{equation}
The proof consists in applying Lax-Milgram's Lemma to  \eqref{eq:theoremcor1}. We first need to prove some preparatory Lemmas. In the sequel, we will use the notation $$T_j :=  e_j + \Delta w_j^{\mathrm{per}} - \nabla p_j^{\mathrm{per}}$$ for $j \in \{1,2,3\}$.

\begin{lemme} Suppose that Assumption \textbf{(A3)} is satisfied.
For all $1 < q < +\infty$, we have that $T_j \in \left[W^{-1,q'}\BB(\mathbb{R}^3 \setminus \overline{\mathcal{O}}) \right]^3$, where $q' = q/(q-1)$.
\label{lem:regularite_T}
\end{lemme}

\begin{proof}
Let $\phi \in \left[\mathcal{D}(\mathbb{R}^3 \setminus \overline{\mathcal{O}})\right]^3$. We extend $\phi$ by 0 in the perforations.
We estimate $\langle T_j,\phi \rangle$ by an integration by parts:
$$\begin{aligned}
\langle T_j,\phi \rangle & = \langle e_j + \Delta w_j^{\mathrm{per}} - \nabla p_j^{\mathrm{per}} , \phi \rangle \\
& = \int_{\mathbb{R}^3 \setminus \overline{\mathcal{O}}} e_j \cdot \phi - \int_{\mathbb{R}^3 \setminus\overline{\mathcal{O}}} \nabla w_j^{\mathrm{per}} : \nabla \phi + \int_{\mathbb{R}^3 \setminus \overline{\mathcal{O}}} (p_j^{\mathrm{per}} - \lambda_j) \mathrm{div}(\phi) \\
& = \int_{\mathbb{R}^3} e_j \cdot \phi - \int_{\mathbb{R}^3} \nabla w_j^{\mathrm{per}} : \nabla \phi + \int_{\mathbb{R}^3} \left(p_j^{\mathrm{per}} - \lambda_j \right) \mathrm{div}(\phi) \\
& = \int_{\mathbb{R}^3} e_j \cdot \phi - \int_{\mathbb{R}^3 \setminus \overline{\mathcal{O}^{\mathrm{per}}}} \nabla w_j^{\mathrm{per}} : \nabla \phi + \int_{\mathbb{R}^3 \setminus \overline{\mathcal{O}^{\mathrm{per}}}} ( p_j^{\mathrm{per}} - \lambda_j) \mathrm{div}(\phi).
\end{aligned}$$
Since $w_j^{\mathrm{per}}$ (resp. $p_j^{\mathrm{per}} - \lambda_j$) is of class $\mathcal{C}^{2,\alpha}$ (resp. of class $\mathcal{C}^{1,\alpha}$) in $\mathbb{R}^3 \setminus \mathcal{O}^{\mathrm{per}}$ (see \cite[Theorem IV.7.1]{galdi2011introduction}), we may integrate by parts and find that
$$\begin{aligned}
\int_{\mathbb{R}^3 \setminus \overline{\mathcal{O}^{\mathrm{per}}}} \nabla w_j^{\mathrm{per}} : \nabla \phi & = \int_{\partial \mathcal{O}^{\mathrm{per}}} \frac{\partial w_j^{\mathrm{per}}}{\partial n} \cdot \phi - \int_{\mathbb{R}^3 \setminus \overline{\mathcal{O}^{\mathrm{per}}}} \Delta w_j^{\mathrm{per}} \cdot \phi, \\
\int_{\mathbb{R}^3 \setminus \overline{\mathcal{O}^{\mathrm{per}}}} (p_j^{\mathrm{per}} - \lambda_j) \mathrm{div}(\phi) & = \int_{\partial \mathcal{O}^{\mathrm{per}}} (p_j^{\mathrm{per}}- \lambda_j) \phi \cdot n - \int_{\mathbb{R}^3 \setminus \overline{\mathcal{O}^{\mathrm{per}}}} \nabla p_j^{\mathrm{per}} \cdot \phi, 
\end{aligned}$$
where we use the notation\CC s\BB
$$\frac{\partial w_j^{\mathrm{per}}}{\partial n} := \left( \frac{\partial w_j^{1,\mathrm{per}}}{\partial n},\frac{\partial w_j^{2,\mathrm{per}}}{\partial n},\frac{\partial w_j^{3,\mathrm{per}}}{\partial n} \right)^T \quad \text{and} \quad w_j^{i,\mathrm{per}} = w_j^{\mathrm{per}} \cdot e_i,$$ for $i,j \in \{1,2,3\}$.
Thus,
$$
\begin{aligned} \langle T_j,\phi \rangle & = \int_{\mathbb{R}^3} e_j \cdot \phi + \int_{\mathbb{R}^3 \setminus \overline{\mathcal{O}^{\mathrm{per}}}} \left[ \Delta w_j^{\mathrm{per}} - \nabla p_j^{\mathrm{per}} \right] \cdot \phi 
+ \int_{\partial \mathcal{O}^{\mathrm{per}}} (p_j^{\mathrm{per}}- \lambda_j) \phi \cdot n 
- \int_{\partial \mathcal{O}^{\mathrm{per}}} \frac{\partial w_j^{\mathrm{per}}}{\partial n} \cdot \phi \\
& = \int_{\mathbb{R}^3} e_j \cdot \phi - \int_{\mathbb{R}^3 \setminus \overline{\mathcal{O}^{\mathrm{per}}}} e_j \cdot \phi + \int_{\partial \mathcal{O}^{\mathrm{per}}} (p_j^{\mathrm{per}}- \lambda_j) \phi \cdot n 
- \int_{\partial \mathcal{O}^{\mathrm{per}}} \frac{\partial w_j^{\mathrm{per}}}{\partial n} \cdot \phi \\
& = \int_{\mathcal{O}^{\mathrm{per}}\setminus \overline{\mathcal{O}}} e_j \cdot \phi 
 + \int_{\partial \mathcal{O}^{\mathrm{per}}} (p_j^{\mathrm{per}}- \lambda_j) \phi \cdot n 
- \int_{\partial \mathcal{O}^{\mathrm{per}}} \frac{\partial w_j^{\mathrm{per}}}{\partial n} \cdot \phi. \\
& = (A) + (B) + (C)
\end{aligned}$$
We treat each term separetely.

\medskip

\textbf{Term (A)}. By H\"older inequality and Assumption \textbf{(A3)} (more precisely \eqref{eq:diffsym}), we obtain that
$$\left| \int_{\mathcal{O}^{\mathrm{per}}\setminus \overline{\mathcal{O}}} e_j \cdot \phi \right| \leq \left|\mathcal{O}^{\mathrm{per}} \setminus \overline{\mathcal{O}} \right|^{\frac{1}{q'}} \| \phi \|_{\left[ L^q(\mathbb{R}^3 \setminus \overline{\mathcal{O}})\right]^3 } \leq C  \left\| \phi \right\|_{ \left[W^{1,q}(\mathbb{R}^3 \setminus \overline{\mathcal{O}})\right]^3}.$$

\medskip

\textbf{Term (B)}. We have by standard regularity results (see \cite[Theorem IV.7.1]{galdi2011introduction}) that $p_j^{\mathrm{per}} \in L^{\infty}(\partial \mathcal{O}_0^{\mathrm{per}})$. We apply a Trace Theorem $W^{1,1}(\mathcal{O}^{\mathrm{per}}_0) \rightarrow L^1(\partial \mathcal{O}^{\mathrm{per}}_0)$ (see e.g. \cite[Theorem 1, p. 258]{evans}) that yields a constant $C$, which is by translation invariance independent of $k$, such that for all $k \in \mathbb{Z}^3$, 
\begin{equation}
\| \phi \|_{\left[ L^1(\partial \mathcal{O}_k^{\mathrm{per}}) \right]^3} \leq C \| \phi \|_{\left[W^{1,1}(\mathcal{O}_k^{\mathrm{per}})\right]^3}.
\label{eq:trace}
\end{equation} By applying \eqref{eq:trace} in the second inequality, we get
$$\begin{aligned}
\left| \int_{\partial \mathcal{O}^{\mathrm{per}}} (p_j^{\mathrm{per}}- \lambda_j) \phi \cdot n \right|
& \leq \left\| p_j^{\mathrm{per}}- \lambda_j \right\|_{L^{\infty}(\partial \mathcal{O}^{\mathrm{per}})} \int_{\partial \mathcal{O}^{\mathrm{per}}} |\phi| \\
& = C \sum_{k \in \mathbb{Z}^3} \int_{\partial \mathcal{O}_k^{\mathrm{per}}} |\phi| \leq C \sum_{k \in \mathbb{Z}^3}  \int_{\mathcal{O}_k^{\mathrm{per}}}  |\phi| + | \nabla \phi| = C \int_{\mathcal{O}^{\mathrm{per}} \setminus \overline{\mathcal{O}}}  |\phi| + |\nabla \phi|,
\end{aligned}$$
where we used in the last equality that $\phi = 0$ in $\mathcal{O}$. Using  \eqref{eq:diffsym},
we conclude thanks to H\"older inequality that
$$\left| \int_{\partial \mathcal{O}^{\mathrm{per}}} (p_j^{\mathrm{per}}- \lambda_j) \phi \cdot n \right| \leq C \left|\mathcal{O}^{\mathrm{per}} \setminus \overline{\mathcal{O}} \right|^{\frac{1}{q'}} \left[ \| \phi \|_{\left[ L^q( \mathcal{O}^{\mathrm{per}} \setminus \overline{\mathcal{O}})\right]^3} + \| \nabla \phi \|_{\left[ L^q(\mathcal{O}^{\mathrm{per}} \setminus \overline{\mathcal{O}}) \right]^{3\times 3} } \right] \leq C \| \phi \|_{\left[ W^{1,q}(\mathbb{R}^3 \setminus \overline{\mathcal{O}}) \right]^3}.$$

\medskip

\textbf{Term (C).} The argument is similar to \textbf{Term (B)}. This gives the existence of a constant $C > 0$ such that:
$$\left| \int_{\partial \mathcal{O}^{\mathrm{per}}} \frac{\partial w_j^{\mathrm{per}}}{\partial n} \cdot \phi \right| \leq C \| \phi \|_{\left[ W^{1,q}(\mathbb{R}^3 \setminus \overline{\mathcal{O}})\right]^3},$$
where $C$ is independent of $\phi$. 
We conclude that there exists a constant $C = C(q) > 0$ such that
$$\forall \phi \in \left[ \mathcal{D} \left(\mathbb{R}^3 \setminus \overline{\mathcal{O}} \right) \right]^3, \quad |\langle T_j,\phi \rangle | \leq C \| \phi \|_{\left[W^{1,q}(\mathbb{R}^3 \setminus \overline{\mathcal{O}} )\right]^3}.$$
This proves the Lemma. 
\end{proof}

\begin{lemme} Suppose that Assumptions \textbf{(A1)} and \textbf{(A3)} are satisfied. For all $1 < q < + \infty$, there exists a function $\phi_j \in \left[ W^{1,q}(\mathbb{R}^3) \right]^3$ such that $\phi_j = w_j^{\mathrm{per}}$ on $\partial \mathcal{O}$.
\label{lem:fonctionaux}
\end{lemme}

\begin{proof}
By Assumption \textbf{(A3)}, there exists a sequence $(\alpha_k)_{k \in \mathbb{Z}^3} \in \ell^1(\mathbb{Z}^3)$ such that for all $k \in \mathbb{Z}^3$, $\alpha_k > 0$ and
$$\left\{ x \in \mathcal{O}_k^{\mathrm{per}}, \ \mathrm{d}(x,\partial \mathcal{O}_k^{\mathrm{per}}) > \alpha_k \right\} \subset \mathcal{O}_k \subset \{x \in Q_k, \ \mathrm{d}(x,\mathcal{O}_k^{\mathrm{per}}) < \alpha_k \}.$$

\medskip

\noindent Let $k \in \mathbb{Z}^3$. 

\medskip

\noindent If $\mathcal{O}_k = \mathcal{O}_k^{\mathrm{per}}$, then we define the function $\chi_k$ by $\chi_k(x) = 0$ for all $x\in Q_k$.

\medskip

\noindent If $\mathcal{O}_k \neq \mathcal{O}_k^{\mathrm{per}}$, there are two cases (see Figure \ref{fig:lem3.4}). 

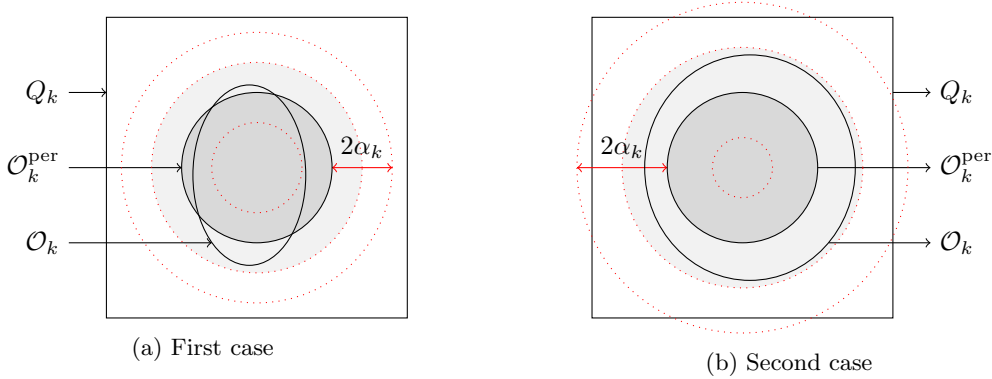
\begin{figure}[h!]
\centering
\begin{subfigure}{.5\textwidth}
\centering
\begin{tikzpicture}
\draw (-2,-2) rectangle (2,2);
\fill[gray!10] (0,0) circle (1.4);
\fill[gray!30] (0,0) circle(1);
\draw[->] (-2.5,0)--(-1,0);
\draw[->] (-2.5,1)--(-2,1);
\draw[->] (-2.5,-1)--(-0.6,-1);
\draw (-2.5,0) node[left]{$\mathcal{O}_k^{\mathrm{per}}$};
\draw (-2.5,1) node[left]{$Q_k$};
\draw (-2.5,-1) node[left]{$\mathcal{O}_k$};
\draw[red,dotted] (0,0) circle(1.4);
\draw[red,dotted] (0,0) circle(0.6);
\draw[red,dotted] (0,0) circle(1.8);
\draw[red,<->] (1,0)--(1.8,0);
\draw (1.4,0) node[above]{$2\alpha_k$};
\draw (0,0) circle(1);
\draw (-0.1,-0.1) ellipse(0.75cm and 1.2cm);
\end{tikzpicture}
\caption{First case}
\end{subfigure}%
\begin{subfigure}{.5\textwidth}
\centering
\begin{tikzpicture}
\draw (-2,-2) rectangle (2,2);
\fill[gray!10] (0,0) circle (1.6);
\fill[gray!30] (0,0) circle(1);
\draw[<-] (2.5,0)--(1,0);
\draw[<-] (2.5,1)--(2,1);
\draw[<-] (2.5,-1)--(1.15,-1);
\draw (2.5,0) node[right]{$\mathcal{O}_k^{\mathrm{per}}$};
\draw (2.5,1) node[right]{$Q_k$};
\draw (2.5,-1) node[right]{$\mathcal{O}_k$};
\draw[red,dotted] (0,0) circle(1.6);
\draw[red,dotted] (0,0) circle(0.4);
\draw[red,dotted] (0,0) circle(2.2);
\draw[red,<->] (-1,0)--(-2.2,0);
\draw (-1.6,0) node[above]{$2\alpha_k$};
\draw (0.1,0) ellipse (1.4cm and 1.5cm);
\draw (0,0) circle(1);
\end{tikzpicture}
\caption{Second case}
\end{subfigure}
\caption{Illustration of the proof of Lemma \ref{lem:fonctionaux}}
\label{fig:lem3.4}
\end{figure}

\medskip

\textbf{First case}. We have $\{x \in \mathbb{R}^3, \ \mathrm{d}(x,\mathcal{O}_k^{\mathrm{per}}) < 2 \alpha_k \} \subset Q_k.$
We consider a function $\chi_k$ which is smooth and compactly supported such that 
$$\begin{cases}
\begin{aligned}
\chi_k & = 1 \quad \mathrm{in} \quad \{x \in Q_k, \ \mathrm{d}(x,\mathcal{O}_k^{\mathrm{per}}) < \alpha_k \} \\
\chi_k & = 0 \quad \mathrm{in} \quad \{x \in Q_k, \ \mathrm{d}(x,\mathcal{O}_k^{\mathrm{per}}) < 2\alpha_k \}^c.
\end{aligned}
\end{cases}$$
We can choose $\chi_k$ such that the following estimates are satisfied:
\begin{equation}
\left| \chi_k \right| \leq 1 \ \ ; \ \ \left| \nabla \chi_k \right| \leq \frac{C}{\alpha_k} \ \ \mathrm{and} \ \ \ \left|\mathrm{supp}(\chi_k) \cap (Q_k \backslash \overline{\mathcal{O}_k^{\mathrm{per}}}) \right|  \leq C\alpha_k,
\label{eq:lemme8}
\end{equation}
where the constants $C$ are independent of $k$.

\medskip

\textbf{Second case}. We have 
$\{x \in \mathbb{R}^3, \ \mathrm{d}(x,\mathcal{O}_k^{\mathrm{per}}) < 2 \alpha_k \} \not\subset Q_k.$
We consider a smooth and compactly supported function $\chi_k$ such that
$$\begin{cases}
\begin{aligned}
\chi_k & = 1 \quad \mathrm{in} \quad \mathcal{O}_k \\
\chi_k & = 0 \quad \mathrm{outside} \quad \mathrm{of} \ Q_k.
\end{aligned}
\end{cases}$$
Because $\alpha_k \underset{|k| \rightarrow + \infty}{\longrightarrow} 0$ and because there exists $\delta > 0$ such that
$$\forall k \in \mathbb{Z}^3, \quad d(\mathcal{O}_k,\partial Q_k) \geq \delta,$$ there are only a finite number of such configurations. After possible changes of the constant $C$, we can suppose that \eqref{eq:lemme8} is valid for all $k \in \mathbb{Z}^3$. 

\medskip

\textbf{Conclusion}. We define
$$\phi_j := \left(\sum_{k \in \mathbb{Z}^3} \chi_k \right) w_j^{\mathrm{per}} \in \left[ W^{1,q}_{\mathrm{loc}}(\mathbb{R}^3) \right]^3.$$
We study the $W^{1,q}-$local norm of $\phi_j$. We fix $k \in \mathbb{Z}^3$ ; one has in $Q_k$:
$$\left| \nabla \phi_j \right| = \left| \nabla \left( \chi_k w_j^{\mathrm{per}} \right) \right| \leq \left| \nabla \chi_k \right| \left|w_j^{\mathrm{per}} \right| + \left|\nabla w_j^{\mathrm{per}} \right| \left|\chi_k \right|.$$
We now use that $\nabla w_j^{\mathrm{per}}$ is bounded and the inequalities \eqref{eq:lemme8}:  
$$\left| \nabla \phi_j \right| \leq C \alpha_k^{-1} |w_j^{\mathrm{per}}| + C.$$
To obtain that $|\nabla \phi_j|$ is bounded on its support, it suffices to show a bound of the type
$$|w_j^{\mathrm{per}}| \leq C \alpha_k \quad \mathrm{in} \quad \{x \in Q_k, \ \mathrm{d}(x,\mathcal{O}_k^{\mathrm{per}}) < 2 \alpha_k \}.$$ 
Since $w_j^{\mathrm{per}} = 0$ on $\mathcal{O}_k^{\mathrm{per}}$ and $\nabla w_j^{\mathrm{per}} \in L^{\infty}(Q)$, this estimate follows from a classical Taylor inequality. We conclude that
$$\exists C > 0, \ \forall k \in \mathbb{Z}^3, \ \forall x \in Q_k, \ \left| \nabla \phi_j (x)\right | \leq C.$$
Because
$$\forall k \in \mathbb{Z}^3, \ |\mathrm{supp}(\phi_j)\cap Q_k| = \left|\mathrm{supp}(\chi_k) \cap \left(Q_k \setminus \overline{\mathcal{O}_k^{\mathrm{per}}} \right) \right| = O(\alpha_k),$$
and because of Assumption \textbf{(A3)},
we conclude that $|\mathrm{supp}(\phi_j)| < + \infty$ and so $\nabla \phi_j \in \left[ L^q(\mathbb{R}^3) \right]^{3 \times 3}$. Similarly, $\phi_j \in \left[ L^q(\mathbb{R}^3) \right]^{3}$. This concludes the Lemma.
\end{proof}

We define, when $R>0$, $$\Omega^{R} := R \Omega \setminus \bigcup_{k \ \mathrm{s.t.} \ Q_k \subset R\Omega} \overline{\mathcal{O}_k}.$$ If $R = 1/\varepsilon$, one has $\Omega^R = \frac{1}{\varepsilon} \Omega_{\varepsilon}$.

\begin{lemme}
Let $T \in \left[ H^{-1}(\mathbb{R}^3 \setminus \overline{\mathcal{O}}) \right]^3$. The Stokes problem
\begin{equation}
\begin{cases}
\begin{aligned}
    - \Delta w + \nabla p & = T \quad \mathrm{in} \quad \mathbb{R}^3 \setminus \overline{\mathcal{O}} \\
    \mathrm{div}(w) & = 0 \\
    w & = 0 \quad \mathrm{on} \quad \partial \mathcal{O}
\end{aligned}
\end{cases}
\label{eq:lemmeEDP}
\end{equation}
admits a solution $(w,p)$ such that $(w,p) \in \left[ H^1_0 (\mathbb{R}^3 \setminus \overline{\mathcal{O}} ) \right]^3 \times L^2_{\mathrm{loc}}(\mathbb{R}^3 \setminus \overline{ \mathcal{O}})$ and $\nabla p \in \left[ H^{-1}(\mathbb{R}^3 \backslash \overline{\mathcal{O}}) \right]^3$. Moreover, for all $R > 0$, we have the estimate
\begin{equation}
    \left\| p - \lambda^R \right\|_{L^2(\Omega^R)} \leq C R \left[ \| \nabla w \|_{\left[ L^2(\mathbb{R}^3 \backslash \overline{\mathcal{O}}) \right]^{3 \times 3}} + \left\| T \right\|_{\left[H^{-1}(\mathbb{R}^3 \backslash \overline{\mathcal{O}})\right]^3} \right], \ \ \lambda^R = \frac{1}{|\Omega^R|} \int_{\Omega^R} p,
    \label{eq:estimationpression}
\end{equation}
\label{lem:stokes}
where $C$ is a constant independent of $T$ and $R$.
\end{lemme}

\begin{proof}
We consider the space 
$H := \{ v \in \left[H^1_0(\mathbb{R}^3 \setminus \overline{\mathcal{O}})\right]^3, \ \mathrm{div}(v) = 0 \}.$ This a Hilbert space as a closed subspace of $\left[ H^1_0(\mathbb{R}^3 \setminus \overline{\mathcal{O}}) \right]^3$. We formulate the following variational problem: find $w \in H$ such that
\begin{equation}
\forall v \in H, \ \ \int_{\mathbb{R}^3 \backslash \overline{\mathcal{O}}} \nabla w : \nabla v = \langle T,v \rangle.
\label{eq:preuvelemme7}
\end{equation}
We recall (see \cite[Proof of Lemma 3.2]{BW}) that we dispose of a Poincar\'e inequality on $\left[ H^1_0(\mathbb{R}^3 \setminus \overline{\mathcal{O}}) \right]^3$ and thus of a Poincar\'e inequality on $H$. We can apply Lax Milgram's Lemma and find a solution $w \in H$ of~\eqref{eq:preuvelemme7}. In particular, for each vector valued function $v \in \left[\mathcal{D}(\mathbb{R}^3 \setminus \overline{\mathcal{O}})\right]^3$ such that $\mathrm{div}(v) = 0$, we have
$$\langle \Delta w + T,v \rangle = 0.$$
Using \cite[Theorem 2.1]{amrouche1994decomposition}, this implies that there exists a distribution $p \in \mathcal{D}'(\mathbb{R}^3 \setminus \overline{ \mathcal{O}})$ such that $\Delta w + T = \nabla p.$
In particular, $\nabla p \in \left[H^{-1}(\mathbb{R}^3 \setminus \overline{\mathcal{O}})\right]^3$. 

Now, we fix $R > 0$. Since $\nabla p \in \left[ H^{-1}(\mathbb{R}^3 \setminus \overline{\mathcal{O}}) \right]^3$, we have $\nabla p \in \left[H^{-1}(\Omega^R) \right]^3$ and 
$$\| \nabla p \|_{\left[ H^{-1}(\Omega^R) \right]^3} \leq \| \nabla p \|_{\left[ H^{-1}(\mathbb{R}^3 \setminus \overline{\mathcal{O}}) \right]^3} \leq \| \nabla w \|_{\left[ L^2(\mathbb{R}^3 \setminus \overline{\mathcal{O}})\right]^{3 \times 3}} + \| T \|_{\left[ H^{-1}(\mathbb{R}^3 \setminus \overline{\mathcal{O}})\right]^3}$$ thanks to the triangle inequality. Lemma \ref{lem4} for $q=2$ furnishes the estimate \eqref{eq:estimationpression}. \\
\end{proof}

\begin{proof}
[Proof of Theorem \ref{th:cor}] We fix $j \in \{1,2,3\}$. Lemma \ref{lem:fonctionaux} gives a function $\phi_j \in \left[ H^1(\mathbb{R}^3 \backslash \overline{\mathcal{O}}) \right]^3$ such that $\phi_j = w_j^{\mathrm{per}}$ on $\partial \mathcal{O}$. The problem
$$
\begin{cases}
\begin{aligned}
\mathrm{div}(\widetilde{v}_j) & = \mathrm{div}(\phi_j) \quad \mathrm{in} \quad \mathbb{R}^3 \setminus \overline{\mathcal{O}} \\
\widetilde{v}_{j} & = 0 \quad \mathrm{on} \quad \partial \mathcal{O}
\end{aligned}
\end{cases}
$$
admits a solution $\widetilde{v}_j \in \left[H^1(\mathbb{R}^3 \setminus \overline{\mathcal{O}}) \right]^3$ thanks to Lemma \ref{lem:divdiv}. Indeed, we just have to check that
$$\forall k \in \mathbb{Z}^3, \ \int_{\partial \mathcal{O}_k} \phi_j \cdot n = \int_{\partial \mathcal{O}_k} w_j^{\mathrm{per}} \cdot n = \int_{\mathcal{O}_k} \mathrm{div}(w_j^{\mathrm{per}}) = 0.$$ Defining $v_j :=  \widetilde{v_j} - \phi_j$ yields a solution to the problem $$
\begin{cases}
\begin{aligned}
\mathrm{div}(v_j) & = 0 \quad \mathrm{in} \quad \mathbb{R}^3 \setminus \overline{\mathcal{O}} \\
v_{j} & = - w_j^{\mathrm{per}} \quad \mathrm{on} \quad \partial \mathcal{O}.
\end{aligned}
\end{cases}$$ By Lemma \ref{lem:stokes}, since $\Delta v_j \in \left[ H^{-1}(\mathbb{R}^d \setminus \mathcal{O}) \right]^3$, there exists a pair $(\widehat{v_j},\widehat{p_j}) \in \left[ H^1_0(\mathbb{R}^3 \setminus \overline{\mathcal{O}})\right]^3 \times L^2_{\mathrm{loc}}(\mathbb{R}^3 \setminus \overline{\mathcal{O}})$ solution of the Problem
\begin{equation}
    \begin{cases}
    \begin{aligned}
    - \Delta \widehat{v}_j + \nabla \widehat{p}_j & = T_j + \Delta v_j  \quad \mathrm{in} \quad \mathbb{R}^3 \setminus \overline{\mathcal{O}} \\
    \mathrm{div}(\widehat{v}_j) & = 0  \\
    \widehat{v}_{j} & = 0 \quad \mathrm{on} \quad \partial \mathcal{O} .
    \end{aligned}
    \end{cases}
\end{equation}
We set $\widetilde{w_j} := \widehat{v_j} + v_j$ and $\widetilde{p_j} = \widehat{p_j}$ and we finish the proof of Theorem \ref{th:cor}.
\end{proof}

%Preuve des résultats 2 et 2' \\

%\underline{Unicité de $\widetilde{w}$?

\subsection{Proof of Theorem \ref{th:convergence} }
\label{sect:proofcvth}

\subsubsection{Strategy of the proof}  

We introduce 
$$\mathcal{R}_{\varepsilon} := u_{\varepsilon} - \varepsilon^2 \sum_{j=1}^3 w_j \left( \frac{\cdot}{\varepsilon} \right) f_j \quad \mathrm{and} \quad \mathcal{P}_{\varepsilon} := p_{\varepsilon} - \varepsilon  \sum_{j=1}^3 p_j \left( \frac{\cdot}{\varepsilon} \right) f_j.$$ The strategy of the proof is to find a Stokes system satisfied by $(R_{\varepsilon},\pi_{\varepsilon})$ and then to apply Theorem~\ref{th:masmoudi}. We need to compute the quantities 
\begin{equation}
- \Delta \mathcal{R}_{\varepsilon} + \nabla \mathcal{P}_{\varepsilon} \quad \mathrm{and} \quad \mathrm{div} (\mathcal{R}_{\varepsilon}).
\label{eq:comput}
\end{equation}
The construction of auxiliary functions is necessary to correct the divergence equation satisfied by $\mathcal{R}_{\varepsilon}$, which doesn't have a suitable order in $\varepsilon$. This is done in subsection \ref{subsect:auxfunc} below (Lemma \ref{lem:cor2h2pasper}). The proof of Theorem \ref{th:convergence} is completed in subsection \ref{subsect:proof}, in particular the computations \eqref{eq:comput}.

\subsubsection{Some auxiliary functions}
\label{subsect:auxfunc}

We recall that the correctors $w_j$, $j \in \{1,2,3\}$ constructed in Theorem \ref{th:cor} are extended by zero in the non-periodic perforations. If $i \in \{1,2,3\}$, we denote $w_j^i := w_j \cdot e_i$ the $i^{\mathrm{th}}-$component of $w_j$. Similarly, $w_j^{i,\mathrm{per}}$ (resp. $\widetilde{w_j^i}$) will be the $i^{\mathrm{th}}-$component of $w_j^{\mathrm{per}}$ (resp. $\widetilde{w_j^i}$). We recall that the definition of the matrix $A$ is given in Equation  \eqref{eq:darcyA}.

\begin{lemme} Suppose that Assumption \textbf{(A4)$_1$} is satisfied.
Let $i,j \in \{1,2,3\}$ and $\chi$ be a function of class $\mathcal{C}^{\infty}$ with support in $Q \setminus \overline{Q}'$ such that $\int_Q \chi = 1$ where $Q'$ is defined in \eqref{eq:Q''} (see also Figure \ref{fig5}). We extend $\chi$ by periodicity to $\mathbb{R}^3 \setminus \overline{\mathcal{O}}$. The problem
\begin{equation}
\begin{cases}
\begin{aligned}
- \mathrm{div} z_j^i & = w_j^i -  \chi A_j^i \quad \mathrm{in} \quad \mathbb{R}^3 \setminus \overline{\mathcal{O}} \\
z_j^i & = 0 \quad \mathrm{on} \quad \partial \mathcal{O}
\end{aligned}
\end{cases}
\label{eq:35}
\end{equation}
admits a solution $z_j^i \in \left[ H^2_{0,\mathrm{loc}}(\mathbb{R}^3 \setminus \overline{\mathcal{O}}) \right]^3$. If we still denote $z_j^i$ the extension of $z_j^i$ by 0 in the perforations, we have the estimate
\begin{equation}
\| z_j^i \|_{\left[ H^2(\Omega/\varepsilon) \right]^3} \leq C \varepsilon^{-\frac{3}{2}} \| w_j^{i,\mathrm{per}} \|_{\left[ H^1(Q) \right]^3} + C \varepsilon^{-1}  \| \widetilde{w_j^i} \|_{\left[ H^1(\mathbb{R}^3)\right]^3}
\label{eq:36}
\end{equation}
for all $\varepsilon > 0$
where $C$ is a constant independent of $\varepsilon$.
\label{lem:cor2h2pasper}
\end{lemme}

\begin{proof} We fix $i,j \in \{1,2,3\}$.
We search $z_j^i$ under the form $z_j^i = \nabla \Psi_j^i + g_j^i$.

\textbf{Step 1}. We build a function $\Psi_j^i$ such that $\nabla \Psi_j^i \in \left[ H^{2,\mathrm{per}}(Q) \right]^3 + \left[ H^{2}_{\mathrm{loc}}(\mathbb{R}^3) \right]^3$ and
$$- \Delta \Psi_j^i = w_j^i - \chi A_j^i \quad \mathrm{on} \quad \mathbb{R}^3.$$
The periodic part of $\Psi_j^i$ is defined by solving the problem
\begin{equation}
\begin{cases}
\begin{aligned}
-\Delta \Psi_j^{i,\mathrm{per}} & = w_j^{i,\mathrm{per}} - \chi A_j^i \quad \mathrm{on} \quad Q \\
\Psi_j^{i,\mathrm{per}} & \in H^{1,\mathrm{per}}(Q).
\end{aligned}
\end{cases}
\label{eq:27}
\end{equation}
Since $\int_Q \big( w_j^{i,\mathrm{per}} - \chi A_j^i \big) = 0$, Problem \eqref{eq:27} is well posed in $H^{1,\mathrm{per}}(Q)/\mathbb{R}$. We choose $\Psi_j^{i,\mathrm{per}}$ such that $\int_Q \Psi_j^{i,\mathrm{per}} = 0$. Because $w_j^{i,\mathrm{per}} - \chi A_j^i \in H^{1,\mathrm{per}}(Q)$, standard elliptic regularity results state that $\nabla \Psi_j^{i,\mathrm{per}} \in \left[ H^{2,\mathrm{per}}(Q) \right]^3$. Besides, there exists a constant $C$ such that
\begin{equation}
\| \nabla \Psi_j^{i,\mathrm{per}} \|_{\left[ H^2(Q) \right]^3} \leq C \| w_j^{i,\mathrm{per}} - \chi A_j^i \|_{\left[ H^1(Q) \right]^3} \leq C \| w_j^{i,\mathrm{per}} \|_{\left[ H^1(Q) \right]^3}.
\label{eq:3.23}
\end{equation}
We now build the non-periodic part of $\Psi_j^i$.
We extend $\widetilde{w_j^i}$ by
 $- w_j^{i,\mathrm{per}}$ in $\mathcal{O}$. We note that, with this extension, $\widetilde{w_j^i} \in\left[H^1(\mathbb{R}^3) \right]^3$. We consider the problem
$$- \Delta \widetilde{\Psi_j^i} = \widetilde{w_j^i} \quad \mathrm{on} \quad \mathbb{R}^3, \quad \widetilde{\Psi_i^j} \underset{|x| \rightarrow +\infty}{\longrightarrow} 0.$$ 
 The solution is given by the Green function: $$\widetilde{\Psi_j^i} = C_3 \frac{1}{|\cdot|} \underset{\mathbb{R}^3}{*} \widetilde{w_j^i}.$$
Thanks to the remarks after the proof of \cite[Theorem 9.9]{gilbarg2015elliptic} (see \cite[p.235]{gilbarg2015elliptic}), we have that $D^2 \widetilde{\Psi_j^i} \in \left[ H^1(\mathbb{R}^3) \right]^{3 \times 3}$ and
\begin{equation}
\| D^2 \widetilde{\Psi_j^i} \|_{\left[L^2(\mathbb{R}^3)\right]^{3 \times 3}} = \| \widetilde{w_j^i} \|_{L^2(\mathbb{R}^3)} \ \ \ \mathrm{and} \ \ \ \| D^3 \widetilde{\Psi_j^i} \|_{\left[L^2(\mathbb{R}^3)^3\right]^{3 \times 3}} = \|\nabla \widetilde{w_j^i} \|_{\left[ L^2(\mathbb{R}^3) \right]^3}.
\label{eq:28}
\end{equation}
%In particular,
%$$\|D^2 \widetilde{\Psi_j^i} \|_{H^1(\Omega_{\varepsilon}/\varepsilon)} \leq C \| \widetilde{w_j^i} \|_{H^1(\mathbb{R}^3)}.$$
Using the Sobolev injection $ \overset{\cdot}{H^1}(\mathbb{R}^3) \hookrightarrow L^6(\mathbb{R}^3)$ for $\nabla \widetilde{\Psi_j^i}$, we deduce that $\nabla \widetilde{\Psi_j^i} \in \left[ L^6(\mathbb{R}^3) \right]^3$ and, using \eqref{eq:28}, that the estimate $$\| \nabla \widetilde{\Psi_j^i} \|_{\left[ L^6(\mathbb{R}^3) \right]^3} \leq C \| \widetilde{w_j^i} \|_{L^2(\mathbb{R}^3)}$$ holds true. In particular, $\nabla\widetilde{\Psi_j^i} \in \left[ L^{2}_{\mathrm{loc}}(\mathbb{R}^3) \right]^3$ and, thanks to H\"{o}lder inequality, we have
$$\| \nabla \widetilde{\Psi_j^i} \|_{\left[ L^2(\Omega_{\varepsilon}/\varepsilon) \right]^3} \leq \frac{C}{\varepsilon} \| \nabla \widetilde{\Psi_j^i} \|_{\left[ L^6(\mathbb{R}^3) \right]^3}.$$
We deduce that
\begin{equation}
\| \nabla \widetilde{\Psi_j^i} \|_{\left[ L^2(\Omega_{\varepsilon}/\varepsilon) \right]^3} \leq \frac{C}{\varepsilon} \| \nabla \widetilde{\Psi_j^i} \|_{\left[ L^6(\mathbb{R}^3) \right]^3}  \leq \frac{C}{\varepsilon} \| \widetilde{w_j^i} \|_{\left[ L^2(\mathbb{R}^3) \right]^3.}
\label{eq:29}
\end{equation}
Finally, collecting \eqref{eq:28} and \eqref{eq:29}, we get
\begin{equation}
\| \nabla \widetilde{\Psi_j^i} \|_{\left[ H^2(\Omega_{\varepsilon}/\varepsilon) \right]^3} \leq \frac{
C}{\varepsilon} \| \widetilde{w_j^i} \|_{\left[ H^1(\mathbb{R}^3) \right]^3}.
\label{eq:37}
\end{equation}
We define $\Psi_j^i := \Psi_j^{i,\mathrm{per}} + \widetilde{\Psi_j^i}$ and verify that  $$- \Delta \Psi_j^i = w_j^{i,\mathrm{per}} - \chi A_j^i + \widetilde{w_j^i} = w_j^i  - \chi A_j^i \quad \mathrm{on} \quad \mathbb{R}^3.$$
We use the periodicity of $\nabla \Psi_j^{i,\mathrm{per}}$ and write that
\begin{equation}
\begin{aligned}
\| \nabla \Psi_j^i \|_{\left[ H^2(\Omega_{\varepsilon}/\varepsilon) \right]^3} 
& \leq \| \nabla \Psi_j^{i,\mathrm{per}} \|_{ \left[ H^2(\Omega_{\varepsilon}/\varepsilon) \right]^3} + \| \nabla \widetilde{\Psi_j^i} \|_{ \left[ H^2(\Omega_{\varepsilon}/\varepsilon) \right]^3} \\
& \leq C \varepsilon^{-\frac{3}{2}}\| \nabla \Psi_j^{i,\mathrm{per}} \|_{\left[ H^2(Q) \right]^3} + \| \nabla \widetilde{\Psi_j^i} \|_{ \left[ H^2(\Omega_{\varepsilon}/\varepsilon) \right]^3},
\end{aligned}
\label{eq:38}
\end{equation}
where the constant $C$ is independent of $\varepsilon$.
We make use of \eqref{eq:37} and \eqref{eq:3.23} and deduce that
\begin{equation}
\| \nabla \Psi_j^i \|_{\left[ H^2(\Omega_{\varepsilon}/\varepsilon) \right]^3} \leq C \varepsilon^{-\frac{3}{2}} \| w_j^{i,\mathrm{per}} \|_{\left[ H^1(Q) \right]^3} + C \varepsilon^{-1} \| \widetilde{w_j^i} \|_{\left[ H^1(\mathbb{R}^3) \right]^3}.
\label{eq:38bis}
\end{equation}

\textbf{Step 2.} We introduce a cut-off function $\chi_1$ such that $\chi_1 = 1$ in $Q'$ and $\chi_1 = 0$ out of $Q$ (see Figure~\ref{fig5}). We fix $k \in \mathbb{Z}^3$ and define $\chi_1^k := \chi_1(\cdot + k)$. The goal of this step is to solve the following problem:
\begin{equation}
\begin{cases}
\begin{aligned}
\mathrm{div}(g_j^{i,k}) & = 0 \quad \mathrm{in} \quad Q_k \setminus \overline{\mathcal{O}_k} \\
g^{i,k}_{j} & = - \nabla \Psi_j^i \quad \mathrm{on} \quad \partial \mathcal{O}_k \\
g^{i,k}_{j} & = 0 \quad \mathrm{on} \quad \partial Q_k.
\end{aligned}
\end{cases}
\label{eq:39}
\end{equation} 
We first solve 
\begin{equation}
\begin{cases}
\begin{aligned}
\mathrm{div}(h_j^{i,k}) & = \mathrm{div}(\chi_1^k \nabla \Psi_j^i) \quad \mathrm{on} \quad Q_k \setminus \overline{\mathcal{O}_k} \\
h_j^{i,k} & \in \left[ H^2_0(Q_k \backslash \overline{\mathcal{O}_k}) \right]^3.
\end{aligned}
\end{cases}
\label{eq:39bis}
\end{equation}
The compatibility condition \eqref{eq:comp} is satisfied:
$$\int_{Q_k \setminus \overline{\mathcal{O}_k}} \mathrm{div}(\chi_1^k \nabla \Psi_j^i) = \int_{\partial \mathcal{O}_k} \chi_1^k \nabla \Psi_j^i \cdot n + \int_{\partial Q_k} \chi_1^k \nabla \Psi_j^i \cdot n = \int_{\mathcal{O}_k} \Delta \Psi_j^i = 0.$$
Since $\mathrm{div}(\chi_1^k \nabla \Psi_j^i) \in H^1_0(Q_k \setminus \overline{\mathcal{O}_k})$, we obtain by Assumption \textbf{(A4)$_1$} a solution $h_j^{i,k} \in \left[H^2_0(Q_k \setminus \overline{\mathcal{O}_k} )\right]^3$ to \eqref{eq:39bis} which satisfies the estimate
$$\| h_j^{i,k} \|_{\left[ H^2(Q_k \setminus \overline{\mathcal{O}_k}) \right]^3} \leq C \| \mathrm{div}(\chi_1^k \nabla \Psi_j^i) \|_{ H^1(Q_k \setminus \overline{\mathcal{O}_k}) } \leq C \| \nabla \Psi_j^i \|_{\left[ H^2(Q_k \setminus \overline{\mathcal{O}_k}) \right]^3} \leq C \| \nabla \Psi_j^i \|_{\left[ H^2(Q_k) \right]^3}.$$
We extend $h_j^{i,k}$ by 0 to $\mathbb{R}^3 \setminus \overline{\mathcal{O}}$.
We then define $g_j^{i,k} := h_j^{i,k} - \chi_1^k \nabla \Psi_j^i$. We note that $g_j^{i,k} = 0$ out of $Q_k$ and that $g_j^{i,k} \in \left[ H^2(\mathbb{R}^3 \setminus \overline{\mathcal{O}}) \right]^3$. Besides, $g_j^{i,k}$ solves Problem~\eqref{eq:39} and satisfies the estimate
\begin{equation}
\| g_j^{i,k} \|_{\left[ H^2(Q_k \setminus \overline{\mathcal{O}_k}) \right]^3} \leq C \| \nabla \Psi_j^i \|_{\left[H^2(Q_k)\right]^3}.
\label{eq:40}
\end{equation}

\textbf{Step 3}. We set
$$g_j^i(x) :=  g_j^{i,k}(x) \quad \mathrm{if} \quad x \in Q_k.$$
Then we have
$$\begin{cases}
\begin{aligned}
\mathrm{div}(g_j^i) & = 0 \quad \mathrm{in} \quad \mathbb{R}^3 \setminus \overline{\mathcal{O}} \\
g_j^i & = - \nabla \Psi_j^i \quad \mathrm{on} \quad \partial \mathcal{O}.
\end{aligned}
\end{cases}$$
Besides, $g_j^i \in \left[ H^2_{\mathrm{loc}}(\mathbb{R}^3 \setminus \overline{\mathcal{O}}) \right]^3$ and summing \eqref{eq:40} over $k \in Y_{\varepsilon}$ yields the estimate 
\begin{equation}
\| g_j^i \|_{\left[ H^2(\Omega_{\varepsilon}/\varepsilon) \right]^3} \leq C \| \nabla \Psi_j^i \|_{\left[H^2(\Omega_{\varepsilon}/\varepsilon)\right]^3}.
\label{eq:41}
\end{equation}
We define $z_j^i := \nabla \Psi_j^i + g_j^i$. We have $z_j^i \in \left[ H^2_{\mathrm{loc}}(\mathbb{R}^3 \backslash \overline{\mathcal{O}}) \right]^3.$ Besides,
$z_j^i$ is a solution of \eqref{eq:35} and, collecting \eqref{eq:38bis} and \eqref{eq:41}, we prove the estimate \eqref{eq:36}:
\begin{equation}
\| z_j^i \|_{\left[H^2(\Omega_{\varepsilon}/\varepsilon)\right]^3} \leq C \| \nabla \Psi_j^i \|_{\left[H^2(\Omega_{\varepsilon}/\varepsilon)\right]^3} \leq C \varepsilon^{-\frac{3}{2}} \| w_j^{i,\mathrm{per}} \|_{\left[H^1(Q)\right]^3} + C \varepsilon^{-1}  \| \widetilde{w_j^i} \|_{\left[H^1(\mathbb{R}^3)\right]^3}.
\end{equation}

It remains to prove that $z_j^i \in \left[ H^2_{0}(\mathbb{R}^3 \backslash \mathcal{O}) \right]^3.$ For that, we fix $k \in \mathbb{Z}^3$ and we notice that in a neighbourhood of the perforation $\partial \mathcal{O}_k$, the equality $z_i^j = h_j^{i,k} + (1 - \chi_1^k) \nabla \Psi_j^i = h_j^{i,k}$ is satsified. Since $h_j^{i,k} \in \left[ H^2_0(Q_k \backslash \overline{\mathcal{O}_k}) \right]^3$, it proves that 
$z_j^i \in \left[ H^2_{0,\mathrm{loc}}(\mathbb{R}^3 \backslash \mathcal{O}) \right]^3.$
This ends the proof.
\end{proof}

\subsubsection{Proof of convergence Theorem \ref{th:convergence}}
\label{subsect:proof}

\begin{proof}

\begin{figure}[h!]

\centering
\begin{tikzpicture}[scale=.12]\footnotesize
 \pgfmathsetmacro{\xone}{-21}
 \pgfmathsetmacro{\xtwo}{21}
 \pgfmathsetmacro{\yone}{-21}
 \pgfmathsetmacro{\ytwo}{21}
 
\fill[gray!20] (0,0) ellipse (19cm and 16cm);
\begin{scope}<+->;
  \draw[step=5cm,gray,very thin] (\xone,\yone) grid (\xtwo,\ytwo);
\end{scope}

\draw[<->] (-20,-20.5) -- (-15,-20.5);
\draw (-17.5,-22) node[]{$\varepsilon$};
\draw[<->] (-20.5,-20) -- (-20.5,-15);
\draw (-22,-17.5) node[]{$\varepsilon$};

\fill[gray!100] (2.5,2.5) ellipse (0.9cm and 1.2cm);
\fill[gray!100] (7,2) ellipse (1cm and 1.2cm);
\fill[gray!100] (13,3) ellipse (1.2cm and 1.2cm);
\fill[gray!100] (-2,3) ellipse (1.2cm and 1cm);
\fill[gray!100] (-8,2.5) ellipse (1.4cm and 1.2cm);
\fill[gray!100] (-13,1.8) ellipse (1cm and 1.2cm);

\fill[gray!100] (-12,-2) ellipse (0.9cm and 1.2cm);
\fill[gray!100] (-8,-2) ellipse (1.2cm and 1.2cm);
\fill[gray!100] (-3,-3) ellipse (1.2cm and 1cm);
\fill[gray!100] (2,-2.5) ellipse (1.4cm and 1.2cm);
\fill[gray!100] (7,-1.8) ellipse (1cm and 1.2cm);
\fill[gray!100] (13,-3) ellipse (1.2cm and 1.2cm);

\fill[gray!100] (-13,-8) ellipse (1cm and 1.2cm);
\fill[gray!100] (-7,-8) ellipse (0.7cm and 1.2cm);
\fill[gray!100] (-3,-8) ellipse (1cm and 1cm);
\fill[gray!100] (1.5,-7) ellipse (1cm and 1cm);
\fill[gray!100] (8,-7) ellipse (1.1cm and 1cm);
\fill[gray!100] (12,-7) ellipse (1.1cm and 1cm);

\fill[gray!100] (1.5,8) ellipse (1cm and 1.2cm);
\fill[gray!100] (7.5,7.5) ellipse (1.2cm and 1.2cm);
\fill[gray!100] (12,7) ellipse (1cm and 1cm);
\fill[gray!100] (-2,8) ellipse (1cm and 1cm);
\fill[gray!100] (-8.5,6.8) ellipse (1.1cm and 1cm);
\fill[gray!100] (-12,7) ellipse (1.1cm and 1cm);

\fill[gray!100] (1.5,-12) ellipse (1.1cm and 1cm);
\fill[gray!100] (-3,-13) ellipse (1.1cm and 1cm);

\fill[gray!100] (1.8,12.2) ellipse (0.7cm and 1.2cm);
\fill[gray!100] (-2.1,12.6) ellipse (1cm and 1cm);

\draw[black,thick] (0,0) ellipse (19cm and 16cm);
\draw[red,thick] (0,0) ellipse (13cm and 8.5cm);
\draw[blue,thick] (5,10)--(15,10)--(15,-10)--(5,-10)--(5,-15)--(-5,-15)--(-5,-10)--(-15,-10)--(-15,10)--(-5,10)--(-5,15)--(5,15)--cycle;
\draw[red,->] (-25,0)--(-13,0);
\draw[red] (-25,0) node[left]{$\Omega'$};

\draw (0,17.5) node[right]{$\Omega_{\varepsilon}$};
\draw[<-,blue] (15,0)--(25,0);
\draw[blue] (26,0) node[right]{$\bigcup_{k \in Y_{\varepsilon}} \varepsilon Q_k$};

\end{tikzpicture}
\caption{Proof of Theorem \ref{th:convergence}}
\label{fig3}
\end{figure}
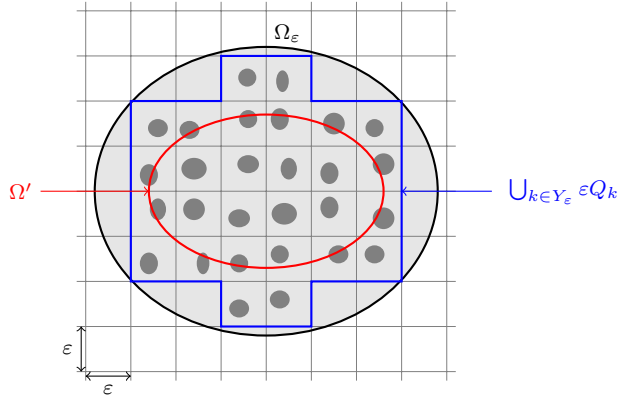

We choose $\varepsilon > 0$ small enough such that
$$\mathrm{supp}(f) \subset \bigcup_{k \in Y_{\varepsilon}} \varepsilon Q_k.$$
We define (see Figure~\ref{fig3}) $\Omega' := \{x \in \Omega \ \text{s.t.} \ f(x) \neq 0 \}$. 
We now set
$$u_{\varepsilon}^1 := \varepsilon^2 w_j \left( \frac{\cdot}{\varepsilon} \right) f_j + \varepsilon^3 z_j^i \left( \frac{\cdot}{\varepsilon} \right) \partial_i f_j$$
and 
$$p_{\varepsilon}^1 := \varepsilon \left[ p_{j} \left( \frac{\cdot}{\varepsilon} \right) - \lambda_{\varepsilon}^j \right] f_j, \ \ \ \lambda_{\varepsilon}^j := \frac{1}{|\frac{1}{\varepsilon}\Omega_{\varepsilon}|} \int_{\frac{1}{\varepsilon} \Omega_{\varepsilon}} p_j.$$
We have $u_{\varepsilon}^1 \in \left[ H^1_0(\Omega_{\varepsilon}) \right]^3$ and $p_{\varepsilon}^1 \in L^2(\Omega_{\varepsilon})$ 
and thus
$$- \Delta u_{\varepsilon}^1 + \nabla p_{\varepsilon}^1 \in \left[ H^{-1}(\Omega_{\varepsilon}) \right]^3.$$
Since (see Figure \ref{fig3}) $f = 0$ in $\Omega \setminus \Omega'$, we have that $u_{\varepsilon}^1$ and $p_{\varepsilon}^1$ are compactly supported in $\Omega$. It is thus sufficient to compute $- \Delta u_{\varepsilon}^1 + \nabla p_{\varepsilon}^1$ in $\Omega_{\varepsilon} \cap \Omega'$. We notice that
$\Omega_{\varepsilon} \cap \Omega' = \Omega' \setminus \varepsilon \overline{\mathcal{O}}.$
Besides, thanks to Lemma \ref{lem:cor2h2pasper}, we have $z_j^i (\cdot/\varepsilon) \in \left[ H^2(\Omega' \setminus \varepsilon \overline{\mathcal{O}}) \right]^3$. We compute in $\Omega' \setminus \varepsilon \overline{\mathcal{O}}$ : 
$$
\begin{aligned}
\Delta u_{\varepsilon}^1 = \Delta w_j \left( \frac{\cdot}{\varepsilon} \right) f_j + 2 \varepsilon \partial_k w_j \left( \frac{\cdot}{\varepsilon} \right) \partial_k f_j + & \varepsilon^2 w_j \left(\frac{\cdot}{\varepsilon} \right) \Delta f_j + \varepsilon\Delta z_j^i \left( \frac{\cdot}{\varepsilon} \right)\partial_i f_j \\ & + 2 \varepsilon^2 \partial_k z_j^i \left( \frac{\cdot}{\varepsilon} \right) \partial_k \partial_i f_j + \varepsilon^3 z_j^i \left(\frac{\cdot}{\varepsilon} \right) \Delta \partial_i f_j.
\end{aligned}$$ 
and 
$$\nabla p_{\varepsilon}^1 = \nabla p_j \left( \frac{\cdot}{\varepsilon} \right) f_j + \varepsilon \left\{ p_j \left( \frac{\cdot}{\varepsilon} \right) - \lambda_{\varepsilon}^j \right\} \nabla f_j.$$
Thus, 
\begin{equation}
\begin{aligned}
 \Delta u_{\varepsilon}^1 - \nabla p_{\varepsilon}^1 & = \Delta w_j \left( \frac{\cdot}{\varepsilon} \right) f_j + 2\varepsilon \partial_k w_j\left( \frac{\cdot}{\varepsilon} \right) \partial_k f_j + \varepsilon^2 w_j \left(\frac{\cdot}{\varepsilon} \right) \Delta f_j + \varepsilon \Delta z_j^i \left( \frac{\cdot}{\varepsilon} \right) \partial_i f_j \\
& + 2\varepsilon^2 \partial_k z_j^i\left( \frac{\cdot}{\varepsilon} \right) \partial_k \partial_i f_j + \varepsilon^3 z_j^i \left(\frac{\cdot}{\varepsilon} \right) \Delta \partial_i f_j - \nabla p_j \left(\frac{\cdot}{\varepsilon}\right) f_j - \varepsilon \left\{ p_j \left(\frac{\cdot}{\varepsilon} \right) - \lambda_{\varepsilon}^j \right\} \nabla f_j \\
& = - f_je_j + \varepsilon f_{\varepsilon} = - f + \varepsilon f_{\varepsilon},
\end{aligned}
\label{eq:thcvcalcul43}
\end{equation}
where
$$
\begin{aligned}
f_{\varepsilon} & := 2 \partial_k w_j\left( \frac{\cdot}{\varepsilon} \right) \partial_k f_j + \varepsilon w_j \left(\frac{\cdot}{\varepsilon} \right) \Delta f_j
+  \Delta z_j^i \left( \frac{\cdot}{\varepsilon} \right) \partial_i f_j + 2\varepsilon \partial_k z_j^i\left( \frac{\cdot}{\varepsilon} \right) \partial_k \partial_i f_j \\
& + \varepsilon^2 z_j^i \left(\frac{\cdot}{\varepsilon} \right) \Delta \partial_i f_j - \left\{ p_j \left(\frac{\cdot}{\varepsilon} \right) - \lambda_{\varepsilon}^j \right\} \nabla f_j.
\end{aligned}$$
Equation \eqref{eq:thcvcalcul43} is still valid in $\Omega_{\varepsilon} \setminus \Omega'$ (the LHS and RHS vanish). We define
$$R_{\varepsilon} := u_{\varepsilon} - u_{\varepsilon}^1 \ \ \ \mathrm{and} \ \ \ \pi_{\varepsilon} := p_{\varepsilon} - p_{\varepsilon}^1.$$
Thus $(R_{\varepsilon},\pi_{\varepsilon}) \in \left[ H^1_0(\Omega_{\varepsilon}) \right]^3 \times L^2(\Omega_{\varepsilon})$ and
$$- \Delta R_{\varepsilon} + \nabla \pi_{\varepsilon} = \varepsilon f_{\varepsilon} \quad \mathrm{in} \quad \Omega_{\varepsilon}, \quad f_{\varepsilon}  \in \left[ L^2(\Omega_{\varepsilon}) \right]^3.$$
Using that $f \in \left[ W^{3,\infty}(\Omega) \right]^3$, we infer
\begin{equation}
\begin{aligned}
\| f_{\varepsilon} \|_{\left[ L^2(\Omega_{\varepsilon})\right]^3} & \leq \left\| \nabla w_j \left( \frac{\cdot}{\varepsilon} \right) \right\|_{\left[L^2(\Omega' \setminus \varepsilon \overline{\mathcal{O}})\right]^{3 \times 3}} \| \nabla f_j \|_{\left[L^{\infty}(\Omega) \right]^3} + \varepsilon\left\| w_j \left( \frac{\cdot}{\varepsilon} \right) \right\|_{\left[L^2(\Omega' \setminus \varepsilon \overline{\mathcal{O}})\right]^3} \| \Delta f_j \|_{L^{\infty}(\Omega)} \\
& + \left\| \Delta z_j^i \left( \frac{\cdot}{\varepsilon} \right) \right\|_{\left[L^2(\Omega' \setminus \varepsilon \overline{\mathcal{O}})\right]^3} \| \partial_i f_j \|_{L^{\infty}(\Omega)} + \varepsilon\left\| \nabla z_j^i \left( \frac{\cdot}{\varepsilon} \right) \right\|_{\left[L^2(\Omega' \setminus \varepsilon \overline{\mathcal{O}})\right]^{3 \times 3}} \| \nabla \partial_i f_j \|_{\left[L^{\infty}(\Omega)\right]^3} \\
&  + \varepsilon^2 \left\| z_j^i \left( \frac{\cdot}{\varepsilon} \right) \right\|_{\left[ L^2(\Omega' \setminus \varepsilon \overline{\mathcal{O}}) \right]^3} \| \Delta \partial_i f_j \|_{L^{\infty}(\Omega)} 
+ \left\| p_j \left( \frac{\cdot}{\varepsilon} \right) - \lambda_{\varepsilon}^j \right\|_{L^2(\Omega' \setminus \varepsilon \overline{\mathcal{O}})} \| \nabla f_j \|_{\left[L^{\infty}(\Omega)\right]^3}. \\
& \leq C \varepsilon^{\frac{3}{2}} \left[ \| w_j \|_{\left[H^1(\frac{1}{\varepsilon} \Omega' \setminus \varepsilon \overline{\mathcal{O}})\right]^3} 
+ \| z_j^i \|_{\left[ H^2(\frac{1}{\varepsilon} \Omega' \setminus \varepsilon \overline{\mathcal{O}}) \right]^3} + \| p_j - \lambda_{\varepsilon}^j \|_{L^2(\frac{1}{\varepsilon} \Omega' \setminus \varepsilon \overline{\mathcal{O}})} \right] \\
& = C \varepsilon^{\frac{3}{2}} \left[ (A) + (B) + (C) \right].
\end{aligned}
\end{equation}
We treat each term separetely. For (A), we have 
\begin{equation}
\begin{aligned}
\| w_j \|_{\left[ H^1(\frac{1}{\varepsilon} \Omega' \setminus \varepsilon \overline{\mathcal{O}}) \right]^3} & \leq \| w_j^{\mathrm{per}}\|_{\left[ H^1(\frac{1}{\varepsilon} \Omega' \setminus \varepsilon \overline{\mathcal{O}})\right]^3} + \| \widetilde{w_j} \|_{\left[ H^1(\frac{1}{\varepsilon} \Omega' \setminus \varepsilon \overline{\mathcal{O}})\right]^3} \\ 
& \leq C \varepsilon^{-\frac{3}{2}} \| \nabla w_j^{\mathrm{per}} \|_{H^1(Q)} +  \| \widetilde{w_j} \|_{\left[H^1(\mathbb{R}^3 \setminus \overline{\mathcal{O}}) \right]^3} .
\label{eq:h2cv1}
\end{aligned}
\end{equation}
For (B), we apply Lemma \ref{lem:cor2h2pasper} (and especially \eqref{eq:36}):
\begin{equation}
\| z_j^i \|_{\left[ H^2(\frac{1}{\varepsilon} \Omega' \setminus \varepsilon \overline{\mathcal{O}}) \right]^3} \leq C \varepsilon^{-\frac{3}{2}}\| w_j^{i,\mathrm{per}} \|_{\left[H^1(Q)\right]^3} + C\varepsilon^{- 1} \| \widetilde{w_j^i} \|_{\left[H^1(\mathbb{R}^3)\right]^3}
\label{eq:h2cv2}
\end{equation}
For (C), Theorem \ref{th:cor} gives
\begin{equation}
\begin{aligned}
 \| p_j - \lambda_{\varepsilon}^j \|_{L^2(\frac{1}{\varepsilon} \Omega' \setminus \varepsilon \overline{\mathcal{O}})} & \leq  \| p_j^{\mathrm{per}} - \lambda_{\varepsilon}^{j,\mathrm{per}} \|_{L^2(\frac{1}{\varepsilon} \Omega' \setminus \varepsilon \overline{\mathcal{O}})} + 
\| \widetilde{ p_j } - \widetilde{\lambda_{\varepsilon}^j} \|_{L^2(\frac{1}{\varepsilon} \Omega' \setminus \varepsilon \overline{\mathcal{O}})} \\
& \leq C \varepsilon^{-\frac{3}{2}}\| p_j^{\mathrm{per}} \|_{L^2(Q)} + C \varepsilon^{ - 1}.
\label{eq:h2cv3}
\end{aligned}
\end{equation}
Collecting \eqref{eq:h2cv1},\eqref{eq:h2cv2} and \eqref{eq:h2cv3}, we conclude that there exists a constant $C > 0$ independent of $\varepsilon$ such that $$\| f_{\varepsilon} \|_{\left[L^2(\Omega_{\varepsilon})\right]^3} \leq C.$$ 
We now study $\mathrm{div}(R_{\varepsilon})$. Using Lemma \ref{lem:cor2h2pasper}, we have in $\Omega_{\varepsilon}$:
$$\mathrm{div}(R_{\varepsilon}) = - \varepsilon^2 \chi \left( \frac{\cdot}{\varepsilon} \right) A_j^i \partial_i f_j - \varepsilon^3 z_j^i \left( \frac{\cdot}{\varepsilon} \right) \cdot \nabla \partial_i f_j.$$
We recall that $\mathrm{div}(Af) = A_j^i \partial_i f_j = 0$. Thus,
$$- \mathrm{div}(R_{\varepsilon}) = \varepsilon^3 z_j^i \left( \frac{\cdot}{\varepsilon} \right) \cdot \nabla \partial_i f_j.$$
We have that
$\varepsilon^3 z_j^i \left( \frac{\cdot}{\varepsilon} \right) \cdot \nabla \partial_i f_j \in \left[ H^1_0(\Omega_{\varepsilon}) \right]^3$ and $\int_{\Omega_{\varepsilon}} \varepsilon^3 z_j^i \left( \frac{\cdot}{\varepsilon} \right) \cdot \nabla \partial_i f_j = 0$. By Lemma \ref{lem:h2} stated in the appendix, there exists $S_{\varepsilon} \in \left[ H^2_0(\Omega_{\varepsilon}) \right]^3$ such that
$$\mathrm{div}(S_{\varepsilon}) = \varepsilon^3 z_j^i \left( \frac{\cdot}{\varepsilon} \right) \cdot \nabla \partial_i f_j \ \ \ \mathrm{and} \ \ \  \| S_{\varepsilon} \|_{H^2(\Omega_{\varepsilon})} \leq C \varepsilon^2 \left\| z_j^i \left( \frac{\cdot}{\varepsilon} \right) \cdot \nabla \partial_i f_j \right\|_{H^1_0(\Omega_{\varepsilon})}.$$
Using that $f \in \left[W^{2,\infty}(\Omega)\right]^3$ and Lemma~\ref{lem:cor2h2pasper}, we get
$$\left\| z_j^i \left( \frac{\cdot}{\varepsilon} \right) \cdot \nabla \partial_i f_j \right\|_{\left[H^1(\Omega_{\varepsilon})\right]^3} \leq \frac{C}{\varepsilon}.$$
Thus 
\begin{equation}
\| S_{\varepsilon} \|_{H^2(\Omega_{\varepsilon})} \leq C \varepsilon.
\label{eq:seps}
\end{equation}
We now define 
$\widehat{R_{\varepsilon}} := R_{\varepsilon} + S_{\varepsilon}.$
The pair $(\widehat{R_{\varepsilon}},\pi_{\varepsilon}) \in \left[ H^1_0(\Omega_{\varepsilon}) \right]^3 \times L^2(\Omega_{\varepsilon})$ is solution to the following Stokes sytem:
\begin{equation}
\begin{cases}
\begin{aligned}
- \Delta \widehat{R_{\varepsilon}} + \nabla \pi_{\varepsilon} & = \varepsilon f_{\varepsilon} - \Delta S_{\varepsilon} \\
\mathrm{div}(\widehat{R_{\varepsilon})} & = 0 \\
\widehat{R_{\varepsilon}}_{|\partial \Omega_{\varepsilon}} & = 0.
\end{aligned}
\end{cases}
\end{equation} 
We notice that $\varepsilon f_{\varepsilon} - \Delta S_{\varepsilon} \in \left[ L^2(\Omega_{\varepsilon}) \right]^3$ thus we may apply Theorem \ref{th:masmoudi}: for all $\Omega'' \subset \Omega$, we have for $\varepsilon < \varepsilon_0(\Omega'')$,
$$\| D^2 \widehat{R_{\varepsilon}} \|_{L^2(\Omega \cap \Omega''_{\varepsilon})} \leq C \| \varepsilon f_{\varepsilon} - \Delta S_{\varepsilon} \|_{L^2(\Omega_{\varepsilon})} \leq C \varepsilon \| f_{\varepsilon} \|_{L^2(\Omega_{\varepsilon})} + \| S_{\varepsilon} \|_{H^2(\Omega_{\varepsilon})} \leq C \varepsilon,$$
and
$$\| \nabla \pi_{\varepsilon} \|_{L^2(\Omega'' \cap \Omega_{\varepsilon})} \leq C \varepsilon.$$
By the triangle inequality and \eqref{eq:seps}, we conclude that
$$\left\| D^2\left[ u_{\varepsilon} - \varepsilon^2 w_j \left( \frac{\cdot}{\varepsilon}\right) f_j \right] \right\|_{L^2(\Omega'' \cap \Omega_{\varepsilon})} \leq C \varepsilon \quad \mathrm{and} \quad \left\| \nabla \left[ p_{\varepsilon} - \varepsilon  \left\{ p_j  \left( \frac{\cdot}{\varepsilon} \right) - \lambda_{\varepsilon}^j \right\} f_j \right] \right\|_{L^2(\Omega'' \cap \Omega_{\varepsilon})} \leq C \varepsilon.$$
\end{proof}

\section*{Acknowledgments}

I am very grateful to my PhD advisor Xavier Blanc for many fruitful discussions and for careful reading of the manuscript. I also thank Claude le Bris for suggesting this subject to me and supporting this project.

%Preuve des résultats 3 et 4 \\

%\underline{Uniformité en les cellules des constantes de Galdi (pb en divergence?)}

\appendix

\section{Technical Lemmas}
\label{sect:appendix}

We recall that if $R > 0$, we define 
\begin{equation}
\Omega^{R} := R \Omega \setminus \bigcup_{k, \ Q_k \subset R\Omega} \mathcal{O}_k.
\label{eq:omegaR}
\end{equation}

\begin{lemme}
[Divergence Lemma on $\Omega^R$] Suppose that Assumption \textbf{(A4)$_0$} is satisfied. Let $1 < q < +\infty$ and $R > 0$.
Let $f \in L^q(\Omega^R)$ be such that
$$\int_{\Omega^{R}} f = 0.$$
The problem
\begin{equation}
    \begin{cases}
    \begin{aligned}
        - \mathrm{div}(v)  & = f \quad \mathrm{in} \quad \Omega_R \\
        v_{} & = 0 \quad \mathrm{on} \quad \partial \Omega_R
    \end{aligned}
    \end{cases}
\end{equation}
admits a solution $v \in \left[W^{1,q}(\Omega_R)\right]^3$ such that
\begin{equation}
    \|v\|_{\left[W^{1,q}(\Omega_R)\right]^3} \leq C R \| f \|_{L^q(\Omega_R)}
\label{eq:lem2}
\end{equation}
where $C > 0$ is a constant independent of $f$ and $R$.
\label{lem2}
\end{lemme}

\begin{proof}
We first extend $f$ by 0 in the perforations. We then solve the problem
\begin{equation}
\begin{cases}
\begin{aligned}
- \mathrm{div}(v_1) & = f \ \mathrm{in} \ R\Omega \\
v_1 & \in \left[ W^{1,q}_0(R\Omega) \right]^3.
\end{aligned}
\end{cases}
\label{eq:macro}
\end{equation}
By Lemma \cite[Theorem III.3.1]{galdi2011introduction} and a simple scaling argument, Problem \eqref{eq:macro} admits a solution $v_1$ such that
$$\| v_1 \|_{W^{1,q}(R\Omega)} \leq C R \| f \|_{L^q(R\Omega)}$$ 
with the constant $C$ being independent of $R$. For $k \in \mathbb{Z}^3$ such that $Q_k \subset R\Omega$, we consider the problem 
\begin{equation}
\begin{cases}
\begin{aligned}
\mathrm{div}(v_2^k) & = 0 \quad \mathrm{in} \quad  Q_k \setminus \overline{\mathcal{O}_k} \\
v^k_{2} & = 0 \quad \mathrm{on} \quad \partial Q_k \\
v^k_{2} & = - v_1 \quad \mathrm{on} \quad \partial \mathcal{O}_k.
\end{aligned}
\end{cases}
\label{eq:micro}
\end{equation}
The compatibility condition for \eqref{eq:micro} is satisfied:
$$- \int_{\partial \mathcal{O}_k} v_1\cdot n = - \int_{\mathcal{O}_k} \mathrm{div}(v_1) = \int_{\mathcal{O}_k} f = 0.$$ Arguing as for Problem \eqref{eq:39}, we show that Problem \eqref{eq:micro} admits a solution $v_2^k \in \left[ W^{1,q}(Q_k \setminus \overline{\mathcal{O}_k}) \right]^3$ such that (the constant $C$ is independent of $k$ thanks to Assumption \textbf{(A4)$_0$}):
\begin{equation}
\| v_2^k \|_{\left[W^{1,q}(Q_k \backslash \overline{\mathcal{O}_k})\right]^3} \leq C \|v_1\|_{\left[W^{1,q}(Q_k \setminus\overline{\mathcal{O}_k})\right]^3}.
\label{eq:v2k}
\end{equation}
We extend $v_2^k$ by zero to $\mathbb{R}^3 \setminus \overline{\mathcal{O}}$. We define the function
$$v_2 := \sum_{k, Q_k \subset R\Omega} v_2^k 1_{Q_k \backslash \overline{\mathcal{O}_k}},$$
Summing \eqref{eq:v2k} over $k$ such that $Q_k \subset R \Omega$ yields
$$\| v_2 \|_{W^{1,q}(\Omega_R)} \leq C \| v_1 \|_{W^{1,q}(R\Omega)} \leq CR \|f \|_{L^q(Q_R)}.$$
We set $v = v_1 + v_2$ and notice that $v$ satisfies the conclusion of Lemma \ref{lem2}.
\end{proof}

\begin{lemme}  Suppose that Assumption \textbf{(A4)$_0$} is satisfied.
Let $1 < q < +\infty$ and $R > 0$. 
%Let $\Omega \subset \mathbb{R}^3$ be a bounded Lipschitz domain  and let $\Omega_R$ be the associated perforated domain.
 Let $f \in \mathcal{D}'(\Omega^R)$ be such that $\nabla f \in \left[ W^{-1,q}(\Omega^R) \right]^3$. Then $f \in L^q(\Omega^R)/\mathbb{R}$ and
\begin{equation}
    \| f \|_{L^q(\Omega^R)/\mathbb{R}} \leq C R \| \nabla f \|_{\left[ W^{-1,q}(\Omega^R) \right]^3}
\label{eq:lemme3}
\end{equation}
\label{lem4}
where $C$ is a constant independent of $f$ and $R$. 
\end{lemme}

\begin{proof}
The fact that $f \in L^q(\Omega_R)/\mathbb{R}$ follows from \cite[Lemma 2.7]{amrouche1994decomposition}.
We now show the estimate \eqref{eq:lemme3}. For $u \in L^1(\Omega^R)$, we denote 
$\lambda_u := \frac{1}{|\Omega^R|}\int_{\Omega^R} u.$ We prove that there exists a constant $C$ independent of $R$ such that
\begin{equation}
    \| f - \lambda_f \|_{L^q(\Omega_R)} \leq C R \| \nabla f \|_{\left[W^{-1,q}(\Omega^R) \right]^3}.
    \label{eq:auxlemme3}
\end{equation}
We argue by duality. We set $q' = q/(q-1)$. We fix a function $g \in L^{q'}(\Omega^R)$ and we define
$\overline{g} := g - \lambda_g.$
We apply Lemma \ref{lem2} to $\overline{g}$: there exists a function $v_g \in \left[W^{1,q'}_0(\Omega^R)\right]^3$ such that 
$$\begin{cases}
\begin{aligned}
- \mathrm{div}(v_g) & = \overline{g} \\
\|v_g \|_{\left[W^{1,q'}(\Omega^R)\right]^3} & \leq C R \| \overline{g} \|_{L^{q'}(\Omega^R)}.
\end{aligned}
\end{cases}$$
Since $\| \overline{g} \|_{L^{q'}(\Omega_R)} \leq 2 \|g\|_{L^{q'}(\Omega_R)}$, we have $\|v_g \|_{\left[W^{1,q'}(\Omega_R)\right]^3} \leq C R \| g \|_{L^{q'}(\Omega_R)}$.
We now write :
$$\langle \nabla f,v_g \rangle_{\left[ W^{-1,q} \times W^{1,q'}_0(\Omega^R)\right]^3} = - \int_{\Omega^R} (f - \lambda_f)\mathrm{div}(v_g) = -\int_{\Omega^R} (f - \lambda_f)(g - \lambda_g) = -\int_{\Omega^R} (f - \lambda_f)g.$$
Thus
$$\left|\int_{\Omega^R} (f - \lambda_f)g \right| \leq \| \nabla f \|_{\left[W^{-1,q}(\Omega^R) \right]^3} \| v_g \|_{\left[W^{1,q'}_0(\Omega^R)\right]^3} \leq C R \| \nabla f \|_{\left[W^{-1,q}(\Omega^R) \right]^3} \| g \|_{L^{q'}(\Omega^R)}.$$
Taking the supremum over $g$, we conclude the proof of the Lemma.
\end{proof}

\begin{lemme}[Scaling] Suppose that Assumption \textbf{(A4)$_0$} is satisfied. Let $1 < q < + \infty$. Let $\varepsilon > 0$ and $\Omega_{\varepsilon}$ be defined by \eqref{eq:Omega_eps}. There exists a constant $C > 0$ independent of $\varepsilon$ such that for all $f \in \mathcal{D}'(\Omega_{\varepsilon})$ such that $\nabla f \in W^{-1,q}(\Omega_{\varepsilon})$, we have $f \in L^q(\Omega_{\varepsilon})/\mathbb{R}$ and the estimate
$$\| f \|_{L^q(\Omega_{\varepsilon})/\mathbb{R}} \leq C \varepsilon^{-1} \| \nabla f \|_{\left[W^{-1,q}(\Omega_{\varepsilon})\right]^3}.$$
\label{lem:estimpres}
\end{lemme}
\begin{proof}
We apply Lemma \ref{lem4} with $R = 1/\varepsilon$ and use a scaling argument  %Let $f \in \mathcal{D}'(\Omega_{\varepsilon})$ be such that $\nabla f \in W^{-1,q}(\Omega_{\varepsilon})$. We define $g = f(\varepsilon \cdot)$, thus $\nabla g \in W^{-1,q}(\Omega_{R})$. Besides, 
%$$\| \nabla g \|_{W^{-1,q}(\frac{1}{\varepsilon}\Omega_{\varepsilon})} \leq C \varepsilon^{-\frac{d}{p}} \| \nabla f \|_{W^{-1,q}(\Omega_{\varepsilon})} \ \ \ \mathrm{et} \ \ \ \ \| g \|_{L^q(\frac{1}{\varepsilon}\Omega_{\varepsilon})/\mathbb{R}} = \varepsilon^{- \frac{d}{p}}\| f \|_{L^q(\Omega_{\varepsilon})/\mathbb{R}}.$$ 
%We then conclude that
%$$\| f \|_{L^q(\Omega_{\varepsilon})/\mathbb{R}} = \varepsilon^{\frac{d}{p}} \| g \|_{L^q(\frac{1}{\varepsilon}\Omega_{\varepsilon})/\mathbb{R}} \leq C \varepsilon^{-1} \varepsilon^{\frac{d}{p}} \| \nabla g \|_{W^{-1,q}(\frac{1}{\varepsilon}\Omega_{\varepsilon})} \leq C \varepsilon^{-1} \| \nabla f \|_{W^{-1,q}(\Omega_{\varepsilon})}.$$
\end{proof}

\begin{lemme} Suppose that Assumption \textbf{(A4)$_0$} is satisfied.
Let $1 < q < + \infty$ and $F \in \left[W^{1,q}(\mathbb{R}^3)\right]^3$. Suppose that for all $k \in \mathbb{Z}^3$,
\begin{equation} 
\int_{\partial \mathcal{O}_k} F \cdot n = 0.
\label{eq:int}
\end{equation}
The problem
\begin{equation}
    \begin{cases}
    \begin{aligned}
    - \mathrm{div}(v) & = \mathrm{div}(F) \quad \mathrm{in} \quad \mathbb{R}^3 \setminus \overline{\mathcal{O}} \\
    v & = 0 \quad \mathrm{on}\quad \partial \mathcal{O}
    \end{aligned}
    \end{cases}
\end{equation}
admits a solution $v \in \left[W^{1,q}(\mathbb{R}^3 \setminus \overline{\mathcal{O}}) \right]^3$ such that
$$\| v \|_{\left[W^{1,q}(\mathbb{R}^3 \setminus \overline{\mathcal{O}})\right]^3} \leq C \| F \|_{\left[W^{1,q}(\mathbb{R}^3 \setminus \overline{\mathcal{O}})\right]^3}$$ where $C$ is a constant independent of $F$.
\label{lem:divdiv}
\end{lemme}

\begin{proof} 
As in the proof of Lemma \ref{lem:cor2h2pasper}, we search the function $v$ under the form $v = \nabla \Psi + v_1$ where 
$$- \Delta \Psi = \mathrm{div}(F) \quad \mathrm{on} \quad \mathbb{R}^3, \quad \mathrm{that \ is} \quad \Psi(x) = C \int_{\mathbb{R}^3} \frac{F(y)\cdot (x-y)}{|x-y|^3} \mathrm{d}y.$$
and 
$$
\begin{cases}
\begin{aligned}
\mathrm{div}(v_1) & = 0 \quad \mathrm{in} \quad \mathbb{R}^3 \setminus \overline{\mathcal{O}} \\
v_1 &= - \nabla \Psi \quad \mathrm{on} \quad \partial \mathcal{O}.
\end{aligned}
\end{cases}$$
Since $F \in \left[L^q(\mathbb{R}^3)\right]^3$, we know that $\nabla \Psi \in \left[L^q(\mathbb{R}^3)\right]^3$ and that there exists a constant $C > 0$ such that $\|\nabla \Psi \|_{\left[L^q(\mathbb{R}^3)\right]^3} \leq C \|F \|_{\left[L^q(\mathbb{R}^3)\right]^3}$ (see e.g. \cite[Exercice II.11.9]{galdi2011introduction}). Besides, since $\mathrm{div}(F) \in L^q(\mathbb{R}^3)$, the estimate $\| D^2 \Psi \|_{\left[L^q(\mathbb{R}^3)\right]^{3\times 3}} \leq C\| \mathrm{div}(F) \|_{L^q(\mathbb{R}^3)}$ holds true (see e.g. \cite[Theorem 9.9 \& p. 235]{gilbarg2015elliptic}). Thus,
$$\nabla \Psi \in \left[W^{1,q}(\mathbb{R}^3) \right]^3 \ \ \ \mathrm{and} \ \ \ \| \nabla \Psi \|_{\left[W^{1,q}(\mathbb{R}^3)\right]^3} \leq C \| F \|_{\left[W^{1,q}(\mathbb{R}^3)\right]^3}.$$
We define the function $v_1$ on each cell $Q_k \setminus \overline{\mathcal{O}_k}$ as a solution of
    \begin{equation}
    \begin{cases}
    \begin{aligned}
    - \mathrm{div}(v_1^k) & = 0 \quad \mathrm{in} \ Q_k \setminus \overline{\mathcal{O}_k} \\
    v^k_{1} & = 0 \quad \mathrm{on} \quad \partial Q_k \\
    v^k_{1} & = - \nabla \Psi \quad \mathrm{on} \quad \partial \mathcal{O}_k.
    \end{aligned}
    \end{cases}
    \label{eq:lemme2}
    \end{equation}
Assumption \textbf{(A4)$_0$} together with \eqref{eq:int} guarantee that Problem \eqref{eq:lemme2} admits a solution that satisfies the estimate 
    $\| v^k_1 \|_{\left[W^{1,q}(Q_k \setminus \overline{\mathcal{O}_k})\right]^3} \leq C \| \nabla \Psi \|_{\left[W^{1,q}(Q_k)\right]^3}$. This proves the Lemma.

\end{proof}

\begin{lemme} Suppose that Assumption \textbf{(A4)$_1$} is satisfied.
Let $g \in H^1_0(\Omega_{\varepsilon})$ be such that $$\int_{\Omega_{\varepsilon}} g = 0.$$ The problem 
\begin{equation}
\begin{cases}
\begin{aligned}
- \mathrm{div}(u) & = g \quad \mathrm{in} \quad \Omega_{\varepsilon} \\
u & = 0  \quad \mathrm{on} \quad \partial \Omega_{\varepsilon} 
\end{aligned}
\end{cases}
\label{eq:lemh2}
\end{equation}
admits a solution $u \in \left[ H^2_0(\Omega_{\varepsilon}) \right]^3$ such that
\begin{equation}
\| u \|_{H^2(\Omega_{\varepsilon})} \leq \frac{C}{\varepsilon} \| g \|_{H^1(\Omega_{\varepsilon})},
\label{eq:lemh2estim}
\end{equation}
where the constant $C$ is independent of $\varepsilon$.
\label{lem:h2}
\end{lemme}

\begin{proof} 
The proof is very similar to the proof of Lemma \ref{lem:cor2h2pasper}. We explain here only the main lines and refer to Subsection \ref{subsect:auxfunc} for details.
We first extend $g$ by $0$ in the perforations. We notice that $$g \in H^1_0(\Omega) \quad \mathrm{and} \quad \int_{\Omega} g = 0.$$ 
 We consider the problem
\begin{equation}
\begin{cases}
\begin{aligned}
- \mathrm{div}(v) & = g \quad \mathrm{in} \quad \Omega \\
v & = 0 \quad \mathrm{on} \quad \partial \Omega.
\end{aligned}
\end{cases}
\label{eq:pbmacro}
\end{equation}
Thanks to \cite[Theorem III.3.3]{galdi2011introduction}, Problem \eqref{eq:pbmacro} admits a solution $v \in \left[H^2_0(\Omega)\right]^3$ such that 
\begin{equation}
\| \nabla v \|_{H^1(\Omega)} \leq C(\Omega) \| g \|_{H^1(\Omega)} \ \ \mathrm{and} \ \ \| v \|_{H^1(\Omega)} \leq C(\Omega) \| g \|_{L^2(\Omega)} .
\label{eq:estimmacro}
\end{equation}
We fix a cell $Q_k$ such that $\varepsilon Q_k \subset \Omega$. We build a function $v_1^k \in \left[ H^2(\varepsilon \big[ Q_k \setminus \overline{\mathcal{O}_k} \big])\right]^3$ such that 
\begin{equation}
\begin{cases}
\begin{aligned}
\mathrm{div}(v_1^k) & = 0 \quad \mathrm{in} \quad \varepsilon \left[ Q_k \setminus \overline{\mathcal{O}_k} \right] \\
v_1^k & = - v \ \ \mathrm{on} \quad \varepsilon \partial \mathcal{O}_k \\
\nabla v_1^k & = - \nabla v \quad \mathrm{on} \quad \varepsilon \partial \mathcal{O}_k.
\end{aligned}
\end{cases}
\label{eq:wk}
\end{equation} 
 For that, we use a cut-off function $\chi^k_{\varepsilon} := \chi\left(\varepsilon[\cdot + k] \right)$ as in \textbf{Step 2} of the proof of Lemma \ref{lem:cor2h2pasper}. We solve 
\begin{equation}
\begin{cases}
\begin{aligned}
\mathrm{div}(w^k) & = \mathrm{div}(\chi_k v)  \quad \mathrm{in} \quad \varepsilon  \left[Q_k \setminus \overline{\mathcal{O}_k} \right] \\
w^k & = 0 \quad \mathrm{on} \quad \varepsilon \partial \left[  Q_k \setminus \overline{\mathcal{O}_k} \right]
\end{aligned}
\end{cases}
\end{equation}
and then set $v_1^k := w^k - \chi^k_{\varepsilon} v$. \cite[Theorem III.3.3]{galdi2011introduction} together with Assumption \textbf{(A4)$_1$} and a standard scaling argument show that Problem \eqref{eq:wk} admits a solution such that 
\begin{equation}
\begin{aligned}
\varepsilon^2 \| D^2 v_1^k \|_{L^2(\varepsilon Q_k \backslash \overline{\mathcal{O}_k})} & + \varepsilon \| \nabla v_1^k \|_{L^2(\varepsilon Q_k \backslash \overline{\mathcal{O}_k})} + \| v_1^k \|_{L^2(\varepsilon Q_k \backslash \overline{\mathcal{O}_k})} \\ & \leq C( \varepsilon^2 \| D^2 v \|_{L^2(\varepsilon Q_k)} + \varepsilon \| \nabla v \|_{L^2(\varepsilon Q_k)} + \| v \|_{L^2(\varepsilon Q_k)}),
\end{aligned}
\label{eq:estimh2-3}
\end{equation}
where the constant $C$ is independent of $k$ and $\varepsilon$. We extend $v_1^k$ by zero to $\Omega_{\varepsilon}$.  
We define 
$$v_1 := \sum_{k\in Y_{\varepsilon}} v^k_1.$$
Then, after summation of \eqref{eq:estimh2-3} over $k$, the estimate
\begin{equation}
\begin{aligned}
\varepsilon^2 \| D^2 v_1 \|_{\left[L^2(\Omega_{\varepsilon})^3\right]^{3 \times 3}} + & \varepsilon \| \nabla v_1 \|_{\left[L^2(\Omega_{\varepsilon})\right]^{3 \times 3}}  + \| v_1 \|_{\left[L^2(\Omega_{\varepsilon})\right]^3} \\ & \leq C\left[ \varepsilon^2 \| D^2 v \|_{\left[L^2(\Omega)^3\right]^{3 \times 3}} + \varepsilon \| \nabla v \|_{\left[L^2(\Omega)\right]^{3 \times 3}} + \| v \|_{\left[L^2(\Omega)\right]^3} \right]
\end{aligned}
\label{eq:estimh2-4}
\end{equation}
holds true.
We note that
the function $u := v + v_1$ satisfies the conclusion of Lemma \ref{lem:h2}. Furthermore, using \eqref{eq:estimmacro} and \eqref{eq:estimh2-4}, we get
\begin{equation}
\begin{aligned}
\varepsilon^2 \| D^2 u \|_{\left[L^2(\Omega_{\varepsilon})^3 \right]^{3 \times 3}} & + \varepsilon \| \nabla u \|_{\left[L^2(\Omega_{\varepsilon})\right]^{3 \times 3}} + \| u \|_{\left[L^2(\Omega_{\varepsilon})\right]^3} \\
& \leq C \left[\varepsilon^2 \| D^2 v \|_{\left[L^2(\Omega)^3\right]^{3 \times 3}} + \varepsilon \| \nabla v \|_{\left[L^2(\Omega)\right]^{3 \times 3}} + \| v \|_{\left[L^2(\Omega)\right]^3} \right] \\ 
& \leq C \left[ \varepsilon^2 \| g \|_{H^1(\Omega)} + \|  v \|_{\left[H^1(\Omega)\right]^3} \right] \leq C \left[ \varepsilon^2 \| g \|_{H^1(\Omega)} + \| g \|_{L^2(\Omega_{\varepsilon})} \right] \\
& \leq C \left[ \varepsilon^2 \| g \|_{H^1(\Omega)} + \varepsilon \| g \|_{H^1(\Omega_{\varepsilon})} \right] \leq C \varepsilon \| g \|_{H^1(\Omega)}.
\end{aligned}
\label{eq:estimh2-5}
\end{equation}
where we used Lemma \ref{lem:poinc} on $g$ in the last inequality. Thus \eqref{eq:lemh2estim} is proved.
\end{proof}

\section{Geometric assumptions} 
\label{sec:geom}

We prove in this section that Assumptions \textbf{(A3)} and \textbf{(A4)'} imply Assumption \textbf{(A4)} and that Assumptions \textbf{(A3)} and \textbf{(A5)'} imply Assumption \textbf{(A5)}. Appendix \ref{sec:geom} follows the proofs of \cite[Theorem III.3.1]{galdi2011introduction} and \cite[Theorem IV.5.1]{galdi2011introduction} and makes precise the dependance of the constants appearing in these arguments. We begin by a covering Lemma.

\begin{lemme} Suppose that Assumption \textbf{(A3)} is satisfied.
Let $0 < \rho < d(\partial \mathcal{O}_0^{\mathrm{per}},\partial Q)$. There exists $N \in \mathbb{N}^*$ such that for all $k \in \mathbb{Z}^3$, there exist $2N$ balls $B_i^k, i = 1,...,2N$ such that
\begin{enumerate}[label=(\roman*)]
\item \label{lemB1i} for all $i=1,...,N$, we have that $B_i^k = B(\xi_i^k,\rho)$, $\xi_i^k \in \partial \mathcal{O}_k$ and $\{x \in Q_k \setminus \overline{\mathcal{O}_k}, \ d(x,\partial \mathcal{O}_k) < 3 \rho/16\} \subset \bigcup_{i=1}^N B_i^k$ ;
\item \label{lemB1ii} for all $i = N+1,...,2N$, we have that $B_i^k = B(\xi_i^k,\rho/32)$, $\xi_i^k \in \{x \in Q_k \setminus \overline{\mathcal{O}_k}, \ d(x,\partial \mathcal{O}_k) > \rho/16\}$ and $\{x \in Q_k \setminus \overline{\mathcal{O}_k}, \ d(x,\partial \mathcal{O}_k) \geq 3 \rho/16\} \subset \bigcup_{i=N+1}^{2N} B_i^k$.
\end{enumerate}
Moreover, there exist $2N$ balls $B_{i}^{0,\mathrm{per}}$, $i=1,...,2N$ and $\eta = \eta(\rho) > 0$ such that 
\begin{enumerate}[label=(\roman*)]
\setcounter{enumi}{2}
\item \label{lemB1iii} for all $i=1,...,2N$, $B_{i}^{0,\mathrm{per}} \subset Q$ and $\left\{x \ \mathrm{s.t.} \ d \left(x,Q \setminus \overline{\mathcal{O}_0^{\mathrm{per}}} \right) < \eta \right\} \subset \bigcup_{i=1}^{2N} B_i^{0,\mathrm{per}}$. 
\item \label{lemB1iv} there exists a bijection $\sigma : \{1,...,2N\} \rightarrow \{1,...,2N\}$ such that for all $i \in \{1,...,2N-1\}$, we have that 
 $$\Omega_{\sigma(i)}^{0,\mathrm{per}} \cap \left( \bigcup_{s=i+1}^{2N} \Omega_{\sigma(s)}^{0,\mathrm{per}} \right) \neq \emptyset \quad \mathrm{and} \quad \Omega_{j}^{0,\mathrm{per}} := B_{j}^{0,\mathrm{per}} \cap \left( Q \setminus \overline{\mathcal{O}_0^{\mathrm{per}}} \right).$$
\item \label{lemB1v} for all but a finite number of $k \in \mathbb{Z}^3$, we have that $B_i^{k,\mathrm{per}} \subset B_i^k$ for all $i=1,...,2N$ and $\left\{x \ \mathrm{s.t.} \ d \left(x,Q_k \setminus \overline{\mathcal{O}_k} \right) < \eta/2 \right\} \subset \bigcup_{i=1}^{2N} B_i^{k,\mathrm{per}}$, where $B_i^{k,\mathrm{per}} := B_i^{0,\mathrm{per}} + k$.
\end{enumerate} 
\label{lem:cover}
\end{lemme}

\begin{remarque}
Lemma \ref{lem:cover}.\ref{lemB1iv} means that we can relabel the family $B_i^{0,\mathrm{per}}, i =1,...,2N$ such that for all $i \in \{1,...,2N-1\}$, we have that $\Omega_i^{0,\mathrm{per}} \cap \left( \Omega_{i+1}^{0,\mathrm{per}} \cup \cdots \cup \Omega_{2N}^{0,\mathrm{per}} \right) \neq \emptyset$.
\label{rem:B2}
\end{remarque}

\begin{proof} 
% Now, we need a suitable covering of $Q''_k \setminus \overline{\mathcal{O}_k}$ to glue together estimates \eqref{eq:geom2} and \eqref{eq:geom1}. We show that there exists $N \in \mathbb{N}^*$ independent of $k$ such that for all $k \in \mathbb{Z}^3$, there exist at most $N$ balls $B_i^k = B(w_i^k,\rho)$, $i =1,...,N$ satisfying $x_i^k \in \partial \mathcal{O}_k$ and 
% \begin{equation}
% \left\{x \in Q_k, \quad d(x,\partial \mathcal{O}_k) < \rho/8 \right\} \subset \bigcup_{i=1}^N B_i^k.
% \label{eq:A52}
% \end{equation}
 
The proof of Lemma \ref{lem:cover} relies on the periodic structure and on Assumption \textbf{(A3)}. We first fix by compactness $N_0$ balls $B_i^{0,\mathrm{per}} = B(x_i,\rho/2)$, $i=1,...,N_0$ such that 
\begin{equation} 
 \left\{ x \in Q, \quad d(x,\partial \mathcal{O}_0^{\mathrm{per}}) \leq \rho/4 \right\} \subset \bigcup_{i=1}^{N_0} B_i^{0,\mathrm{per}} \quad \mathrm{and} \quad x_i \in \partial \mathcal{O}_0^{\mathrm{per}}.
 \label{eq:A5}
 \end{equation}
 We note that there exists $\widehat{\rho} > 0$ such that for all $i \in \{1,...,N_0\}$, there exist two points $y_i \in B_i^{0,\mathrm{per}} \cap \mathcal{O}_0^{\mathrm{per}}$ and $z_i \in B_i^{0,\mathrm{per}} \setminus \mathcal{O}_0^{\mathrm{per}}$ satisfying $d(y_i,\partial \mathcal{O}_0^{\mathrm{per}}) > \widehat{\rho}$ and $d(z_i,\partial \mathcal{O}_0^{\mathrm{per}}) > \widehat{\rho}$.
We define for each $k \in \mathbb{Z}^3$, $x_i^k := x_i +k$, $y_i^k := y_i + k$, $z_i^k := z_i + k$ and $B_i^{k,\mathrm{per}} := B_i^{0,\mathrm{per}} + k = B(x_i^k,\rho/2)$. By translation invariance, we obviously have \eqref{eq:A5} with 0 replaced by any $k \in \mathbb{Z}^3$.
 
 \medskip

 We consider $k \in \mathbb{Z}^3$ such that $\alpha_k < \min(\rho/16,\widehat{\rho})$ (where we recall that $\alpha_k$ is introduced in \textbf{(A3)}). Then, by Assumption \textbf{(A3)} and \eqref{eq:A5}, we have that
 \begin{equation}
 \left\{x \in Q_k, \quad d(x,\partial \mathcal{O}_k) < 3\rho/16 \right\} \subset \left\{x \in Q_k, \quad d(x,\partial \mathcal{O}_k^{\mathrm{per}}) < \rho/4 \right\} \subset \bigcup_{i=1}^{N_0} B_i^{k,\mathrm{per}}.
 \label{eq:A51}
 \end{equation}
 %In particular, $\partial \mathcal{O}_k$ is covered by $B_i^{k,\mathrm{per}}$, $i=1,...,N_0$.
 We next claim that each ball $B_i^{k,\mathrm{per}}$, $i=1,...,N_0$ intersects $\partial \mathcal{O}_k$. By definition, we have that $y_i^k \in B_i^{k,\mathrm{per}} \cap \mathcal{O}_k^{\mathrm{per}}$ and that $d(y_i^k,\partial \mathcal{O}_k^{\mathrm{per}}) > \widehat{\rho} > \alpha_k$. Thus, by \textbf{(A3)}, we get that $y_i^k \in B_i^{k,\mathrm{per}} \cap \mathcal{O}_k$. Similarly, we have that $z_i^k \in B_i^{k,\mathrm{per}} \setminus \mathcal{O}_k$. Thus, there exists $\xi_i^k \in [y_i^k,z_i^k] \cap \partial \mathcal{O}_k$, proving that $\partial \mathcal{O}_k \cap B_i^{k,\mathrm{per}} \neq \emptyset$.
%Consider $f : t \mapsto d(x_t,\mathcal{O}_k^{\mathrm{per},c})$ where $x_t = tx_i^k + (1-t)y_i^k$. It is a continuous function. Choose $t^*$ such that $t^* = arginf\{ f(t) = 0\}$. By continuity, $f(t^*) = 0$ so $x_{t*} \in \overline{\mathcal{O}_k^{\mathrm{per},c}}$. If $x_{t*} \notin \overline{\mathcal{O}_k^{\mathrm{per}}}$ then $x_{t*} \in Q \setminus \overline{\mathcal{O}_k^{\mathrm{per}}}$ who is open. Thus, there exists $t < t^*$ such that $x_t \in Q \setminus \overline{\mathcal{O}_k^{\mathrm{per}}} \subset Q \setminus\mathcal{O}_k^{\mathrm{per}}$. This contradicts that $t^*$ is inf. 
  We fix an arbitrary point $\xi_i^k \in \partial \mathcal{O}_k \cap  B_i^{k,\mathrm{per}}$ and we notice that $ B_i^{k,\mathrm{per}} \subset B(\xi_i^k,\rho)$. By \eqref{eq:A51}, we conclude that 
 \begin{equation}
 \{x \in Q_k, \quad d(x,\partial \mathcal{O}_k) < 3\rho/16\} \subset \bigcup_{i=1}^{N_0} B_i^k.
 \label{eq:LEMMAB1}
 \end{equation}
 It remains to cover $\{x \in Q_k \setminus \overline{\mathcal{O}_k}, \ d(x,\partial\mathcal{O}_k) \geq 3\rho/16\}$. By \textbf{(A3)}, we have that
 \begin{equation}
 \{x \in Q_k \setminus \overline{\mathcal{O}_k}, \ d(x,\partial\mathcal{O}_k) \geq 3\rho/16\} \subset \{x \in \overline{Q_k \setminus \mathcal{O}_k^{\mathrm{per}}}, \ d(x,\partial\mathcal{O}_k^{\mathrm{per}}) \geq \rho/8\}.
 \label{eq:appendixB}
 \end{equation}
By compactness and translation invariance, we can cover the right hand side of \eqref{eq:appendixB} by  $N_1$ balls $B_i^{k,\mathrm{per}} = B(x_i^k,\rho/32)$, $i=N_0+1,...,N_0+N_1$ where $x_i^k$ is of the form $x_i^k = x_i +k$ and $x_i \in \{x \in Q \setminus \overline{\mathcal{O}_0^{\mathrm{per}}}, \ d(x,\partial \mathcal{O}_0^{\mathrm{per}}) \geq \rho/8\}$. We set $B_i^k := B_i^{k,\mathrm{per}}$ and $\xi_i^k := x_i^k$. By \textbf{(A3)}, we get that $\xi_i^k \in \{ Q_k \setminus \overline{\mathcal{O}_k}, \quad d(x,\partial \mathcal{O}_k) > \rho/16\}$. 
With $N$ to be fixed later, we have proved \ref{lemB1i}-\ref{lemB1ii} for $k \in \mathbb{Z}^3$ such that $\alpha_k < \min(\rho/16,\widehat{\rho})$.
\medskip

We fix $k \in \mathbb{Z}^3$ such that $\alpha_k \geq \min(\rho/16,\widehat{\rho})$. We take any covering of $\{x \in Q_k, \ d(x,\partial \mathcal{O}_k) < 3\rho/16\}$ with balls $B_i^k = B(\xi_i^k,\rho)$, $\xi_i^k \in \partial \mathcal{O}_k$ and $i \in \{1,...,N_0^k\}$. We then take any covering of $\{x \in Q_k \setminus \overline{\mathcal{O}_k}, \ d(x,\partial \mathcal{O}_k) \geq 3\rho/16 \}$ with balls $B_i^k = B(\xi_i^k,\rho/32)$, $\xi_i^k \in \{x \in Q_k \setminus \overline{\mathcal{O}_k}, \ d(x,\partial \mathcal{O}_k) \geq 3\rho/16\}$ and $i \in \{N_0^k + 1,...,N_0^k + N_1^k\}$.
 
 \medskip
 
 We set 
 $N := \max_{k \in \mathbb{Z}^3}\{ N_0^k,N_1^k\}$ where $N_0^k = N_0$ and $N_1^k = N_1$ if $\alpha_k < \min(\rho/16,\widehat{\rho})$. Note that because of \textbf{(A3)}, we have that $N < +\infty$. If $N^k_0 < N$ or $N_1^k < N$, we duplicate one of the balls in order to define $2N$ balls $B_i^k, i=1,...,2N$. We proceed similarly for $B_i^{0,\mathrm{per}}$, $i=1,...,N_0+N_1$. Assertions \ref{lemB1i}, \ref{lemB1ii}, \ref{lemB1iii} and \ref{lemB1v} are proved. We prove easily \ref{lemB1iv} by connectedness of $Q \setminus \overline{\mathcal{O}_0^{\mathrm{per}}}$. 
 % We recall that, because of \textbf{(A3)}, we have $\alpha_k \underset{|k| \rightarrow +\infty}{\longrightarrow} 0$. There is as a consequence only a finite number of $N_k^0$'s and $N_k^1$'s. Choosing $N := \max(N_0,N_1,\max_{k} N_k^0,\max_{k} N_k^1)$ proves \ref{lemB1i}-\ref{lemB1ii}-\ref{lemB1iii}-\ref{lemB1v} of Lemma \ref{lem:cover}. Note finally that since $Q \setminus \overline{\mathcal{O}_0^{\mathrm{per}}}$ is connected, we can relabel the balls $B_i^{0,\mathrm{per}}$, $i=1,...,2N$ such that \ref{lemB1iv} is satisfied. 
 \end{proof}
  
\subsection*{Assumptions \textbf{(A3)} and \textbf{(A4)'} imply \textbf{(A4)}}

\subsubsection*{Proof that \textbf{(A4)$_0$ is satisfied}}

Let $k \in \mathbb{Z}^3$. We formulate \cite[Theorem III.3.1]{galdi2011introduction} in our particular setting: suppose that there exists $\Omega_i^k, i=1,...,N_k$ such that
\begin{equation}
Q_k \setminus \overline{\mathcal{O}_k} = \bigcup_{i=1}^{N_k} \Omega_i^k,
\label{eq:decA4}
\end{equation}
where $\Omega_i^k$ is star-shaped with respect to a ball $B_i^k$ of radius $\rho_i^k$ such that $B_i^k \subset \subset \Omega_{i}^k$. We define for $i=1,...,N^k-1$
$$F_i^k := \Omega_i^k \cap \left( \bigcup_{s=i+1}^{N_k} \Omega_s^k \right)$$
and we assume that $F_i^k \neq \emptyset$ for all $i \in \{1,...,N_k-1\}$.
Then Problem \eqref{A4} with $f \in L^q(Q_k \setminus \overline{\mathcal{O}_k})$ admits a solution $v$ satisfying \eqref{eq:introdiv} with 
\begin{equation}
C_q^0(k) \leq C(q) \left( \frac{\mathrm{diam} (Q_k \setminus \overline{\mathcal{O}_k})}{\min_{i=1}^{N_k} \rho_i^k} \right)^{3} \left( 1 + \frac{\mathrm{diam} (Q_k \setminus \overline{\mathcal{O}_k})}{\min_{i=1}^{N_k} \rho_i^k} \right) \left(1 + \frac{|Q_k \setminus \overline{\mathcal{O}_k}|^{1 - 1/q}}{\min_{i=1}^{N_k-1} |F_i|^{1 - 1/q}} \right)^{N_k}.
\label{eq:dependanceA4}
\end{equation}
To bound $C_q^0(k)$ uniformly in $k$, it is sufficient to show that $Q_k \setminus \overline{\mathcal{O}_k}$ admits a decomposition of the form \eqref{eq:decA4} where $N_k$ is independent of $k$, $\rho_i^k$ and $|F_i^k|$ are uniformly bounded from below in $k$ and $i$. We first explain how to find such a decomposition with $N_k$ and $\rho_i^k$ independent of $k$ and $i$. By making precise the dependance on the geometry of $\partial \mathcal{O}_k$ at each step of the proof of \cite[Lemma II.1.3]{galdi2011introduction},  
%(see Figure \ref{fig:A4preuve})
we can show that \textbf{(A4)'} implies that there exists $\rho > 0$ such that for all $k \in \mathbb{Z}^3$ and $\xi^k \in \partial \mathcal{O}_k$, there exists an open set $G_{\xi^k}$ such that  $\Omega_{\xi^k} := G_{\xi^k} \cap \left(Q_k \setminus \overline{\mathcal{O}_k}\right)$ is star-shaped with respect to a ball of radius $\rho$ strictly included in $\Omega_{\xi^k}$ and $B(\xi^k,\rho) \subset G_{\xi^k}$.
%\begin{figure}[h!]
%\centering
%\begin{tikzpicture}[scale=3]\footnotesize
%\draw (-2,-0.6) node[above]{\large{$\partial \mathcal{O}_k$}};
%\draw (-4,-4)--(-4,-3);
%\draw (-3.5,-2.5)--(-4,-3);
%\draw (-3.5,-2.5) arc(90:0:1);
%\draw (-2.5,-3.5)--(-2.2,-3.8);
%\draw (-4,-4)--(-2.2,-3.8);
%\fill[lightgray] (-1,0) circle (0.3cm);
%\draw[dotted] (-1,0) circle (0.3cm);
%\draw (-2,0) ellipse(1cm and 0.7cm);
%\draw[<->] (-1,0)--(-1.3,0);
%\draw (-1.15,0) node[above]{\large{$r$}};
%\draw [->] (-0.8,0.35) arc (120:60:1);
%\draw (1,0) circle (0.6);
%\draw (0.31,-0.25) arc (135:45:1);
%\draw (1,0.1) node[above]{\large{$x$}};
%\fill[black] (1,0.04) circle(0.03);
%\draw (1.7,-0.25) node[right]{\large{$\zeta_x$}};
%\draw[dashed] (1,-0.7)--(1,0.1);
%\draw[blue] (1,-0.6)--(0.7,0);
%\draw[blue] (1,-0.6)--(1.3,0);
%\fill[gray] (1,-0.3) circle(0.1);

%\end{tikzpicture}
%\caption{Illustration of the proof of \cite[Lemma II.3.1]{galdi2011introduction}}
%\label{fig:A4preuve}
%\end{figure}
%We now have to choose a covering of $Q_k \setminus \overline{\mathcal{O}_k}$. 

\medskip

We next apply Lemma \ref{lem:cover} with $\rho$ given before and we denote by $B_i^k$, $i=1,...,2N$ the family of balls that we obtain. For $i=1,...,N$, we define $G_i^k := G_{\xi_i^k}$ and $\Omega_i^k := \Omega_{\xi_i^k}$. For $i=N+1,...,2N$, we define $\Omega_i^k := B_i^k$. Since $B_i^k \subset G_i^k$ for $i=1,...,N$ and because $B_i^k, i=1,...,2N$ covers $Q_k \setminus \overline{\mathcal{O}_k}$, we have that \eqref{eq:decA4} is satisfied with $\rho_i^k \geq \rho/32$ and $N_k = N$. 
%Besides, it is clear that $\Omega_i^k$ is star-shaped with respect to a ball of radius $\rho$. 

\medskip

It remains to check that there exists a relabeling of the $\Omega_i^k$'s such that we have that $\min_{i=1}^{2N-1} |F_i^k| \geq C$ where $C > 0$ is independent of $k$. We use Lemma \ref{lem:cover}.\ref{lemB1iii}-\ref{lemB1v}. According to Remark \ref{rem:B2}, we relabel the $\Omega_i^{k,\mathrm{per}}$ (note that this also implies a relabeling of the $\Omega_i^k$'s) such that 
$$\forall i \in \{1,...,2N-1\}, \ F_i^{k,\mathrm{per}} :=  \Omega_i^{k,\mathrm{per}} \cap \left( \Omega_{i+1}^{k,\mathrm{per}} \cup \cdots \cup \Omega_{2N}^{k,\mathrm{per}} \right) \neq \emptyset.$$ We then fix $\rho' > 0$ such that for all $i \in \{1,...,2N-1\}$, we have that $F_{i}^{k,\mathrm{per}}$ contains a ball $\left(B_{i}^{k,\mathrm{per}}\right)'$ of radius $\rho'$ such that $\left(B_{i}^{k,\mathrm{per}}\right)' \subset \subset Q_k \setminus \overline{\mathcal{O}_k^{\mathrm{per}}}$. We fix $k \in \mathbb{Z}^3$ such that Lemma \ref{lem:cover}.\ref{lemB1v} is satisfied and such that 
\begin{equation}
\alpha_k < \min_{i=1}^{2N-1} d\left(\left(B_{i}^{k,\mathrm{per}}\right)',\partial \mathcal{O}_k^{\mathrm{per}} \right). 
%\quad \left[ = \min_{i=1}^{2N-1} d\left(\left(B_{i}^{0,\mathrm{per}}\right)',\partial \mathcal{O}%_0^{\mathrm{per}} \right) \right].
\label{eq:proofA40}
\end{equation}
Then, for all $i \in \{1,...,2N-1\}$, we have that 
\begin{equation}
\left(B_{i}^{k,\mathrm{per}}\right)' \subset \subset Q_k \setminus \overline{\mathcal{O}_k}.
\label{eq:proofA40A}
\end{equation}
We then recall that
$$F_i^k = \Omega_{i}^k \cap \left(\bigcup_{s=i+1}^{2N} \Omega_{i}^k \right) = \left[ G_{i}^k \cap \left(\bigcup_{s=i+1}^{2N} G^k_{i} \right) \right] \cap \left( Q_k \setminus \overline{\mathcal{O}_k} \right).$$ By Lemma \ref{lem:cover}.\ref{lemB1v}, we have that $B_j^{k,\mathrm{per}} \subset G_j^k$ for all $j \in \{1,...,2N\}$. Together with \eqref{eq:proofA40A}, this yields that 
$\left(B_{i}^{k,\mathrm{per}}\right)' \subset F_{i}^k$ for all $i \in \{1,...,2N-1\}$. Thus, $\min_{i=1}^{2N-1} | F_{i}^k | \geq \frac{4}{3} \pi \rho'^3$. Since by \textbf{(A3)} there are only a finite number of indices $k$ such that \eqref{eq:proofA40} is not satisfied, we conclude that, after eventually relabeling the $F_i^k$'s, we have that $\min_{i=1}^{2N-1} |F_i^k| \geq C > 0$.

\subsubsection*{Proof that \textbf{(A4)$_1$ is satisfied}}

We briefly sketch the proof of \textbf{(A4)$_1$} and we refer to the proof of \textbf{(A4)$_0$} for some ingredients. Let $k \in \mathbb{Z}^3$ and $f \in W^{1,q}_0(Q_k \setminus \overline{\mathcal{O}_k})$. To solve Problem \eqref{A4}, we use a decomposition of the form \eqref{eq:decA4} with $N_k$ uniform in $k$ ($=N$) and $\Omega_i^k$ that is star-shaped with respect to a ball of radius $\rho$ uniformly bounded from below in $k$ and $i$, as constructed in the proof of \textbf{(A4)$_0$}. We then write $f = f_1 + \cdots + f_{N}$ where $f_i \in W^{1,q}_0(\Omega_i)$, $\int_{\Omega_i} f_i = 0$, $\|f_i\|_{W^{1,q}(\Omega_i^k)} \leq C_i^k \|f\|_{W^{1,q}(Q_k \setminus \overline{\mathcal{O}_k})}$ and we solve the Problem:
$$\begin{cases}
\begin{aligned}
\mathrm{div}\ v_i & = f_i \quad \mathrm{in} \quad \Omega_i^k \\
v_i & \in \left[ W^{2,q}_0(\Omega_i^k)\right]^3.
\end{aligned}
\end{cases}$$
Thanks to the estimate (III.3.23) of \cite[p. 168]{galdi2011introduction}, we have that 
$$\|v_i\|_{\left[W^{2,q}(\Omega_i)\right]^3} \leq C(q,\rho) \|f_i \|_{W^{1,q}(\Omega_i^k)} \leq C(q,\rho) C_i^k \|f\|_{W^{1,q}(Q_k \setminus \overline{\mathcal{O}_k})}.$$
Extending $v_i$ by zero to 
$Q_k \setminus \overline{\mathcal{O}_k}$ and setting $v := v_1 + \cdots + v_{N}$, we have that $v$ solves Problem \eqref{A4} with the estimate $$\|v\|_{\left[W^{2,q}(Q_k \setminus \overline{\mathcal{O}_k})\right]^3} \leq C(q,\rho) N C_i^k\|f \|_{W^{1,q}(Q_k \setminus \overline{\mathcal{O}_k})}.$$
We can conclude that \textbf{(A4)$_1$} is satisfied if $C_i^k$ is uniformly bounded in $i$ and $k$. To prove that, we make precise the dependance of the constant controlling $\|f_i\|_{W^{1,q}(\Omega_i^k)}$ in the proof of \cite[Lemma III.3.4.(vii)-(viii)]{galdi2011introduction}. This constant depends on $N$ and on the maximum of the $W^{1,\infty}-$norms of the functions $\Psi_i^k$, $i=1,...,2N$ and $\chi_i^k$, $i=1,...,2N-1$ where $\{\Psi_1^k,...,\Psi_{2N}^k\}$ is a partition of unity associated to $\{G_1^k,...,G_{2N}^k\}$ and $\chi_i^k \in \mathcal{D}(F_i^k)$ satisfies $\int_{F_i^k} \chi_i^k = 1$. Because of Lemma \ref{lem:cover}.\ref{lemB1v}, the family $\{\Psi_1^k,...,\Psi_{2N}^k\}$ may be chosen independently of $k$ (by using the periodic balls), except for a finite number of indices $k$. Besides, still after the exclusion of a finite number of indices $k$, we have shown in the proof of \textbf{(A4)$_0$} that $F_i^k$ contains a ball of radius $\rho'$ which is uniformly bounded in $i$ and $k$. Thus, $\chi_i^k$ may be chosen as the translation of a reference function $\chi$ satisfying $\chi \in \mathcal{D}(B(0,\rho'))$ and $\int_{B(0,\rho')} \chi = 1$. This proves that $\max_{i=1}^{2N} C_i^k \leq C$ for all but a finite number of $k \in \mathbb{Z}^3$. Applying \cite[Lemma III.3.4]{galdi2011introduction} for the remaining indices $k$, we conclude that $\max_{i=1}^{2N} C_i^k  \leq C$ for all $k \in \mathbb{Z}^3$. This concludes the proof of \textbf{(A4)$_1$}.

%For \textbf{(A4)$_1$}, we need to apply \cite[Theorem III.3.3]{galdi2011introduction} and in particular to understand the dependance of the constants on the domain $\Omega$. The main result is \cite[Lemma III.3.4]{galdi2011introduction} and concerns the construction of a suitable covering of $Q_k \setminus \overline{\mathcal{O}_k}$, which we can take here equal to the one given by Lemma \ref{lem:cover}. Let $\{ \Psi_1^k,...,\Psi_{2N}^k\}$ a partition of unity associated to 
%$$\{G_1,...,G_{2N}\} := \{C_{w_1}, C_{w_2},...,C_{w_n},B_{N+1}^k,...,B_{2N}^k \}.$$
%We define $\Omega_i := G_i \cap \Omega$. We check that $\nabla \Psi_i^k$ is uniformly bounded (in $k$). We then see that for all $i \in \{1,...,2N-1\}$, there exists $\theta_i \in C^{\infty}_0 \left(\Omega_i \cap \left( \bigcup_{j=i+1}^{2N} \Omega_j \right) \right)$ such that $\| \theta_k \|_{W^{1,\infty}} \leq C$ and $\int \theta_i =1$. This proves the result.

%We note that because of \cite[Theorem III.3.1]{galdi2011introduction} and \cite[Lemma II.1.3 and its proof]{galdi2011introduction}, we have that \textbf{(A4)} is implied by \textbf{(A4)'}. 

\subsection*{Assumptions \textbf{(A3)} and \textbf{(A5)'} imply \textbf{(A5)}}

 We fix $f \in \left[L^q(Q''_k \setminus \overline{\mathcal{O}_k})\right]^3$ and we consider the pair $(v,p)$ solution to \eqref{A5}. We want to prove the regularity estimate \eqref{eq:introreg}. The interior regularity property is given by the following result (see \cite[Theorem IV.4.1]{galdi2011introduction}):
\begin{equation}
\|D^2 v \|_{L^q(\Omega_k)^{3\times 3 \times 3}} + \| \nabla p \|_{L^q (\Omega_k)^3} \leq
 C \big[ \| v \|_{W^{1,q}(\Omega'_k)^3} + \| p \|_{L^q (\Omega'_k)} + \|f\|_{L^q(\Omega'_k)^3} \big],
\label{eq:geom2}
\end{equation}
where $\Omega_k \subset \subset \Omega'_k \subset \subset Q''_k \setminus \overline{\mathcal{O}_k}$ and $C$ depends only on $q$ and on the distance between $\Omega_k$ and $\left(\Omega'_k \right)^{c}$.
The regularity up to the boundary follows from the discussion \cite[pp.271-274] {galdi2011introduction}. By tracing the dependance of the constants in these arguments, we can show that, under Assumption \textbf{(A5)'}, there exist a radius $\rho > 0$, a constant $d > 1$ and a  constant $C > 0$ such that $d \rho < d(Q,\partial Q'')$ and for all $k \in \mathbb{Z}^3$ and $x \in \partial \mathcal{O}_k$, we have that
\begin{equation}
\begin{aligned}
\|D^2 v \|_{L^q\left(\left( Q_k \setminus \overline{\mathcal{O}_k} \right) \cap B(x,\rho)\right)^{3 \times 3 \times 3}}  & + \| \nabla p \|_{L^q \left(\left( Q_k \setminus \overline{\mathcal{O}_k} \right) \cap B(x,\rho)\right)^3} \leq
 C \big[ \| v \|_{W^{1,q}\left(\left( Q_k \setminus \overline{\mathcal{O}_k} \right) \cap B(x,d\rho)\right)^3} \\ & + \| p \|_{L^q \left(\left( Q_k \setminus \overline{\mathcal{O}_k} \right) \cap B(x,d\rho)\right)} + \|f\|_{L^q \left(\left( Q_k \setminus \overline{\mathcal{O}_k} \right) \cap B(x,d\rho)\right)^3} \big].
\end{aligned}
\label{eq:geom1}
\end{equation}
We combine estimates \eqref{eq:geom2} and \eqref{eq:geom1}. We fix $k \in \mathbb{Z}^3$. Let $\left(B_{i}^k\right)_{i=1,...,2N}$ be the family of balls given by Lemma \ref{lem:cover} (applied with $\rho$ defined by \eqref{eq:geom1}). Thanks to \eqref{eq:geom1} and the inequality
\begin{equation}
\forall \ a_1,...,a_p > 0, \quad a_1^q + \cdots + a_p^q \leq (a_1 + \cdots + a_p)^q \leq C_{p,q}(a_1^q + \cdots + a_p^q),
\label{eq:intermédiaire}
\end{equation}
we have for all $i \in \{1,...,N\}$,
\begin{equation}
\begin{aligned}
\|D^2 v \|_{L^q\left(\left( Q_k \setminus \overline{\mathcal{O}_k} \right) \cap B(\xi_i^k,\rho)\right)^{3 \times 3 \times 3}}^q  & + \| \nabla p \|_{L^q \left(\left( Q_k \setminus \overline{\mathcal{O}_k} \right) \cap B(\xi_i^k,\rho)\right)^3}^q \leq
 C \big[ \| v \|_{W^{1,q}\left(\left( Q_k \setminus \overline{\mathcal{O}_k} \right) \cap B(\xi_i^k,d\rho)\right)^3}^q \\ & + \| p \|_{L^q \left(\left( Q_k \setminus \overline{\mathcal{O}_k} \right) \cap B(\xi_i^k,d\rho)\right)}^q + \|f\|_{L^q \left(\left( Q_k \setminus \overline{\mathcal{O}_k} \right) \cap B(\xi_i^k,d\rho)\right)^3}^q \big].
\end{aligned}
\label{eq:geom1a}
\end{equation}
Summing \eqref{eq:geom1a} over $i\in \{1,...,N\}$ and using that 
$$U_k := \{x \in Q_k \setminus \overline{\mathcal{O}_k}, \ d(x,\partial \mathcal{O}_k) < 3 \rho/16\} \subset \bigcup_{i=1}^N B(\xi_i^k,\rho) \quad \mathrm{and} \quad \left(Q_k \setminus \overline{\mathcal{O}_k} \right) \cap B(\xi_i^k,d\rho) \subset Q''_k \setminus \overline{\mathcal{O}_k}$$ yield
\begin{equation}
\|D^2 v \|_{L^q\left(U_k\right)^{3 \times 3 \times 3}}^q  + \| \nabla p \|_{L^q \left(U_k\right)^3}^q \leq
 C N \left[ \| v \|_{W^{1,q}\left(Q''_k \setminus \overline{\mathcal{O}_k}\right)^3}^q + \| p \|_{L^q \left(Q''_k \setminus \overline{\mathcal{O}_k}\right)}^q + \|f\|_{L^q \left(Q''_k \setminus \overline{\mathcal{O}_k}\right)^3}^q \right].
\label{eq:geom1a1}
\end{equation}
We now apply \eqref{eq:geom2} to $\Omega_k = \{x \in Q_k \setminus \overline{\mathcal{O}_k}, \ d(x,\partial \mathcal{O}_k) > \rho/8\}$ and $\Omega'_k = Q''_k \setminus \overline{\mathcal{O}_k}$. We have that $d\big(\Omega_k, (\Omega'_k)^c\big) = \min\big(d(Q,\partial Q''),\rho/8 \big)$ is independent of $k$. Thus, using \eqref{eq:geom2} and \eqref{eq:intermédiaire} yield
\begin{equation}
\|D^2 v \|_{L^q(\Omega_k)^{3\times 3 \times 3}}^q + \| \nabla p \|_{L^q (\Omega_k)^3}^q \leq
 C \left[ \| v \|_{W^{1,q}(Q''_k \setminus \overline{\mathcal{O}_k})^3}^q + \| p \|_{L^q (Q''_k \setminus \overline{\mathcal{O}_k})}^q + \|f\|_{L^q(Q''_k \setminus \overline{\mathcal{O}_k})^3}^q \right],
\label{eq:geom2a}
\end{equation}
where $C$ is independent of $k$. Summing \eqref{eq:geom1a} and \eqref{eq:geom2a} and using that $U_k \cup \Omega_k = Q_k \setminus \overline{\mathcal{O}_k}$ together with \eqref{eq:intermédiaire} proves \textbf{(A5)}.

%For all $i = N+1,..,2N$, we apply \eqref{eq:geom2} to $\Omega_k = B_i^k$ and $\Omega'_k = B(\xi_i^k,\rho/30)$:
%\begin{equation}
%\|D^2 v \|_{L^q\left( B_i^k\right)^{3\times 3 \times 3}}^q + \| \nabla p \|_{L^q (B_i^k)^3}^q \leq
% C \big[ \| v \|_{W^{1,q}(B(\xi_i^k,\rho/30)^3}^q + \| p \|_{L^q (B(\xi_i^k,\rho/30)}^q + \|f\|%_{L^q(B(\xi_i^k,\rho/30} \big].
%\label{eq:geom2}
%\end{equation}

%Using \eqref{eq:geom1} and Lemma \ref{lem:cover}, we get 
% \begin{equation}
%\begin{aligned}
%\|D^2 v \|_{L^q(U_k)^{3 \times 3 \times 3}} + \| \nabla p \|_{L^q (U_k)^3} \leq
% C \big[ \| v \|_{W^{1,q}\left( Q''_k \setminus \overline{\mathcal{O}_k} \right)^3} + \| p \|_{L^q \left(Q''_k \setminus \overline{\mathcal{O}_k} \right)} + \|f\|_{L^q \left( Q''_k \setminus \overline{\mathcal{O}_k} \right)^3} \big],
%\end{aligned}
%\label{eq:geom3}
%\end{equation}
%where $U_k := \left\{ x \in Q_k \setminus \overline{\mathcal{O}_k}, \quad d(x,\partial \mathcal{O}_k) < \rho/8 \right\}$ and $C$ is independent of $k$. We now apply \eqref{eq:geom2} to $\Omega_k^i := B(w_i^k,\rho/32), i = N+1,...,2N$ and $\Omega'_k := B(w_i^k,\rho/30), i = N+1,...,2N$. Since $d(\Omega_k,\Omega'_k)^i$ is independent of $k$, we conclude that \textbf{(A5)} is proved. 

\subsection*{Counter-examples to the geometric assumptions}

\begin{figure}[h!]
\centering
\begin{subfigure}{.33\textwidth}
\centering
\begin{tikzpicture}[scale=0.8]\footnotesize
 \pgfmathsetmacro{\xone}{-0.2}
 \pgfmathsetmacro{\xtwo}{5.2}
 \pgfmathsetmacro{\yone}{-0.2}
 \pgfmathsetmacro{\ytwo}{5.2}
\begin{scope}<+->;
  \draw[step=5cm,gray,very thin] (\xone,\yone) grid (\xtwo,\ytwo);
\end{scope}

\draw (4.7,4.5) node[left]{$|k| \sim 1$};

\fill[blue!5] (2.2,2.5) ellipse (1.5cm and 1.5cm);
\draw[blue,thick] (2.2,2.5) ellipse (1.5cm and 1.5cm);
\draw[blue] (2.2,2.5) node[]{$\mathcal{O}_k^{\mathrm{per}}$};
\draw[black,thick] (2.2,2.5) ellipse (1.3cm and 1.7cm); 
\draw (2.2,0.8) node[below]{$\mathcal{O}_k$};
\fill[blue!5] (3.25,2.65) rectangle (3.55,2.35);
\draw[thick,black] (3.48,2.65)--(4.5,2.65);
\draw[thick,black] (3.48,2.35)--(4.5,2.35);
\draw[thick,black] (4.5,2.35)--(4.5,2.65);

\end{tikzpicture}  

\begin{tikzpicture}[scale=0.8]\footnotesize
 \pgfmathsetmacro{\xone}{-0.2}
 \pgfmathsetmacro{\xtwo}{5.2}
 \pgfmathsetmacro{\yone}{-0.2}
 \pgfmathsetmacro{\ytwo}{5.2}
\begin{scope}<+->;
  \draw[step=5cm,gray,very thin] (\xone,\yone) grid (\xtwo,\ytwo);
\end{scope}
\draw (4.7,4.5) node[left]{$|k| \gg 1$};

\fill[blue!5] (2.2,2.5) ellipse (1.5cm and 1.5cm);
\draw[blue,thick] (2.2,2.5) ellipse (1.5cm and 1.5cm);
\draw[blue] (2.2,2.5) node[]{$\mathcal{O}_k^{\mathrm{per}}$};
\draw[black,thick] (2.2,2.5) ellipse (1.43cm and 1.57cm); 
\draw (2.2,0.8) node[below]{$\mathcal{O}_k$};
\fill[blue!5] (3.6,2.53) rectangle (3.68,2.47);
\draw[thick,black] (3.62,2.53)--(4.5,2.53);
\draw[thick,black] (3.62,2.47)--(4.5,2.47);
\draw[thick,black] (4.5,2.54)--(4.5,2.46);

\end{tikzpicture}
\caption{Counter-example to \textbf{(A3)}, \\ close to the origin and at infinity}
\end{subfigure}%
\begin{subfigure}{.33\textwidth}
\centering
\begin{tikzpicture}[scale=0.8]\footnotesize
 \pgfmathsetmacro{\xone}{-0.2}
 \pgfmathsetmacro{\xtwo}{5.2}
 \pgfmathsetmacro{\yone}{-0.2}
 \pgfmathsetmacro{\ytwo}{5.2}
\begin{scope}<+->;
  \draw[step=5cm,gray,very thin] (\xone,\yone) grid (\xtwo,\ytwo);
\end{scope}

\draw (4.7,4.5) node[left]{$|k| \sim 1$};

\fill[blue!5] (2.2,2.5) ellipse (1.5cm and 1.5cm);
\draw[blue,thick] (2.2,2.5) ellipse (1.5cm and 1.5cm);
\draw[blue] (2.2,2.5) node[]{$\mathcal{O}_k^{\mathrm{per}}$};
\draw[black,thick] (2.2,2.5) ellipse (1.3cm and 1.7cm); 
\draw (2.2,0.8) node[below]{$\mathcal{O}_k$};
\fill[blue!5] (3.1,2.9) rectangle (3.55,2.1);
\draw[thick,black] (3.46,2.92)--(3.75,2.5);
\draw[thick,black] (3.46,2.08)--(3.75,2.5);
\end{tikzpicture}

\begin{tikzpicture}[scale=0.8]\footnotesize
 \pgfmathsetmacro{\xone}{-0.2}
 \pgfmathsetmacro{\xtwo}{5.2}
 \pgfmathsetmacro{\yone}{-0.2}
 \pgfmathsetmacro{\ytwo}{5.2}
\begin{scope}<+->;
  \draw[step=5cm,gray,very thin] (\xone,\yone) grid (\xtwo,\ytwo);
\end{scope}

\draw (4.7,4.5) node[left]{$|k| \gg 1$};

\fill[blue!5] (2.2,2.5) ellipse (1.5cm and 1.5cm);
\draw[blue,thick] (2.2,2.5) ellipse (1.5cm and 1.5cm);
\draw[blue] (2.2,2.5) node[]{$\mathcal{O}_k^{\mathrm{per}}$};
\draw[black,thick] (2.2,2.5) ellipse (1.43cm and 1.57cm); 
\draw (2.2,0.8) node[below]{$\mathcal{O}_k$};
\fill[blue!5] (3.55,2.65) rectangle (3.67,2.35);
\draw[thick,black] (3.61,2.65)--(3.8,2.61);
\draw[thick,black] (3.61,2.57)--(3.8,2.61);
\draw[thick,black] (3.61,2.57)--(3.8,2.53);
\draw[thick,black] (3.61,2.49)--(3.8,2.53);
\draw[thick,black] (3.61,2.49)--(3.8,2.45);
\draw[thick,black] (3.61,2.41)--(3.8,2.45);
\draw[thick,black] (3.61,2.41)--(3.8,2.37);
\draw[thick,black] (3.61,2.35)--(3.8,2.37);
\end{tikzpicture}
\caption{Counter-example to \textbf{(A4)'}, \\ close to the origin and at infinity}

\end{subfigure}%
\begin{subfigure}{.33\textwidth}
\centering
\begin{tikzpicture}[scale=0.8]\footnotesize
 \pgfmathsetmacro{\xone}{-0.2}
 \pgfmathsetmacro{\xtwo}{5.2}
 \pgfmathsetmacro{\yone}{-0.2}
 \pgfmathsetmacro{\ytwo}{5.2}
\begin{scope}<+->;
  \draw[step=5cm,gray,very thin] (\xone,\yone) grid (\xtwo,\ytwo);
\end{scope}

\draw (4.7,4.5) node[left]{$|k| \sim 1$};

\fill[blue!5] (2.2,2.5) ellipse (1.5cm and 1.5cm);
\draw[blue,thick] (2.2,2.5) ellipse (1.5cm and 1.5cm);
\draw[blue] (2.2,2.5) node[]{$\mathcal{O}_k^{\mathrm{per}}$};
\draw[black,thick] (2.2,2.5) ellipse (1.3cm and 1.7cm); 
\draw (2.2,0.8) node[below]{$\mathcal{O}_k$};
\fill[blue!5] (3.1,3) rectangle (3.55,2);

\draw[thick,black] (3.43,1.97) arc(-40:40:0.81);

\end{tikzpicture}

\begin{tikzpicture}[scale=0.8]\footnotesize
 \pgfmathsetmacro{\xone}{-0.2}
 \pgfmathsetmacro{\xtwo}{5.2}
 \pgfmathsetmacro{\yone}{-0.2}
 \pgfmathsetmacro{\ytwo}{5.2}
\begin{scope}<+->;
  \draw[step=5cm,gray,very thin] (\xone,\yone) grid (\xtwo,\ytwo);
\end{scope}

\draw (4.7,4.5) node[left]{$|k| \gg 1$};

\fill[blue!5] (2.2,2.5) ellipse (1.5cm and 1.5cm);
\draw[blue,thick] (2.2,2.5) ellipse (1.5cm and 1.5cm);
\draw[blue] (2.2,2.5) node[]{$\mathcal{O}_k^{\mathrm{per}}$};
\draw[black,thick] (2.2,2.5) ellipse (1.43cm and 1.57cm); 
\draw (2.2,0.8) node[below]{$\mathcal{O}_k$};
\fill[blue!5] (3.25,2.55) rectangle (3.67,2.35);
\draw[thick,black] (3.61,2.35) arc(-50:50:0.5cm and 0.13cm);

\end{tikzpicture}
\caption{Counter-example to \textbf{(A5)'}, \\ close to the origin and at infinity}
\end{subfigure}

\caption{Counter-examples to Assumptions \textbf{(A3)-(A4)'-(A5)'}. For each assumption, the picture above expresses the fact that there is no restriction on the perforation $\mathcal{O}_k$ when~$k$ remains bounded. The picture below shows a perforation $\mathcal{O}_k$ that is not allowed by Assumption \textbf{(A3)}, \textbf{(A4)'} or \textbf{(A5)'} when $|k| \rightarrow +\infty$.}
\label{fig6}
\end{figure}
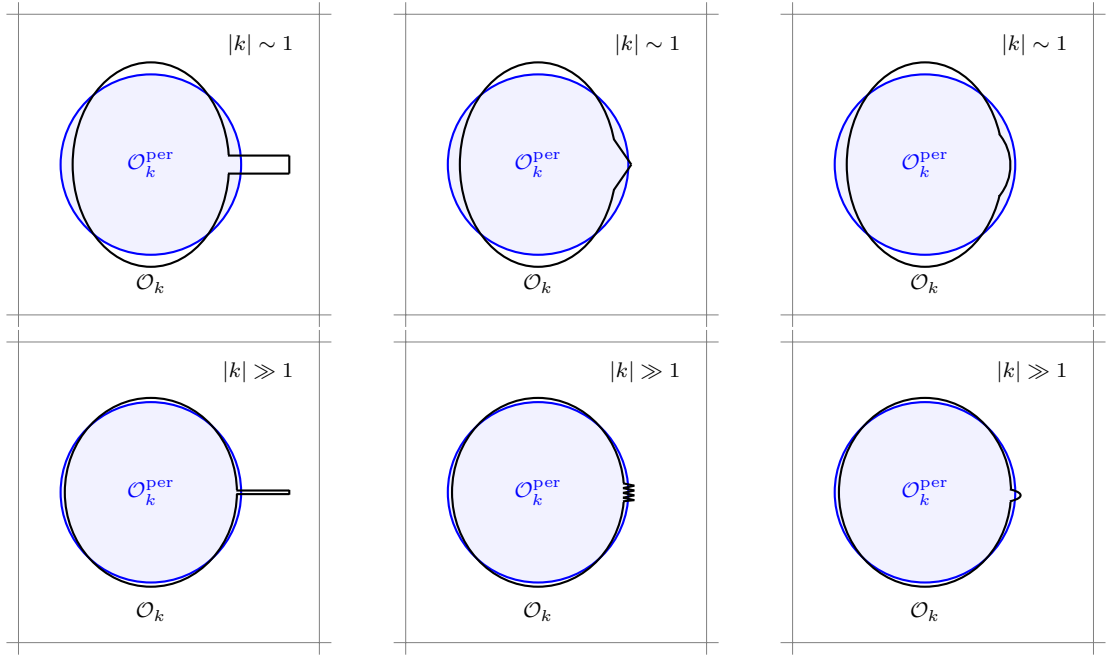

  \bibliography{article}
  \bibliographystyle{plain}

\end{document}